\documentclass[10pt]{amsart}
\pdfoutput=1
\input{preamble.sty}
\title{The Geometry of Secondary Terms in Arithmetic Statistics}
\author{Michael Kural}
\address{Harvard University}
\date{}

\begin{document}

\begin{abstract}
In this thesis, we prove the existence of a secondary term for the count of cubic extensions of the function field $\F_q(t)$ of fixed absolute norm of discriminant. We show that the number of cubic extensions with absolute norm of discriminant equal to $q^{2N}$ is $c_1 q^{2N} - c_2^{i} q^{5N/3} + O_{\eps}\left(q^{(3/2+\eps)N}\right)$, where $c_1$ and $c_2^{i}$ are explicit constants and $c_2^{i}$ only depends on $N\pmod{3}$.

This builds on the work of Bhargava--Shankar--Tsimerman and Taniguchi--Thorne, who proved the existence of a secondary term for the count of cubic extensions of $\Q$ with bounded discriminant. Our approach uses a parametrization of Miranda and Casnati--Ekedahl, which can be seen as a geometric version of the classical parametrization by binary cubic forms used by Davenport--Heilbronn. This allows us to count and sieve for smooth curves embedded in Hirzebruch surfaces, in the same spirit as Zhao and Gunther.

\vspace{2\baselineskip}
\noindent\textit{This was submitted (with minor formatting changes) to Harvard as the author's thesis on April 3, 2025.}
\end{abstract}

\maketitle
\section{Introduction}
\subsection{Background}
A problem of primary importance in arithmetic statistics is counting extensions of a global field by discriminant. Davenport and Heilbronn \cite{DH71} proved that the number of cubic extensions $K/\Q$ of bounded discriminant of absolute value at most $X$ is asymptotic to $cX$ for a constant $c$.

Datskovsky and Wright \cite{DW88} gave an asymptotic for the number of cubic extensions of bounded absolute discriminant of an arbitrary global field of characteristic not equal to $2$ or $3$.

It was noticed that despite these results, there was poor agreement in the empirical data between the main term and the true number of cubic extensions of discriminant bounded by $X$. Roberts \cite{Rob01} (and also Datskovsky and Wright \cite{DW88}) conjectured a formula for cubic extensions with a \textit{secondary term} in addition to the main term. More precisely, Roberts conjectured that the number of cubic extensions $K/\Q$ of absolute discriminant bounded by $X$ should be in the form 
\[
\#\{K/\Q: |\Disc_{K/\Q}|<X\} = c_1 X - c_2 X^{5/6} + o(X^{5/6})
\]
for some positive constants $c_1$ and $c_2$.

Roberts' conjecture was proven independently by  Bhargava, Shankar, and Tsimerman \cite{BST13} and Taniguchi and Thorne \cite{TT13} with a power saving error term. 

In recent years, there has been increasing attention on secondary terms for various counts in arithmetic statistics, including those for $2$-Selmer groups \cite{ST24}, $S_3\times A$-extensions \cite{Wan17}, and quartic extensions \cite{ST25}, all over $\Q$.

Separately, there has been an interest recently in counting extensions of function fields using geometric and topological methods. Over the function field $\F_q(t)$, each degree $n$ extension $K/\F_q(t)$ corresponds to a branched cover $X \to \PP^1$, and counting by discriminant is equivalent to counting covers with a fixed degree of branch divisor.

One may count (for example) simply branched degree $n$ covers with $b$ branch points by counting $\F_q$-points of an associated Hurwitz scheme $\Hur_{n,b}$. In general, a Hurwitz scheme parametrizes branched covers of $\PP^1$ and may specify a Galois group and ramification types of the branch points (along with data at infinity).

Ellenberg, Venkatesh, and Westerland  \cite{EVW16} (see also Achter \cite{Ach06}, Yu \cite{Yu97}) initiated a program to count certain simply branched covers by using the Grothendieck-Lefschetz trace formula and a homological stability theorem for Hurwitz spaces. This approach was recently used by Landesman and Levy \cite{LL25a} to prove a moments version of the Cohen-Lenstra conjecture over $\F_q(t)$ when $q$ is sufficiently large given the moment of interest (also see \cite{LL25b}).

Gunther \cite{Gun17} proved a main term for cubic extensions of an arbitrary function field $\F_q(C)$ (with $\Char \F_q >3$) using a geometric approach rather than a topological approach.

One may naturally conjecture that there should be a secondary term in the count for cubic extensions of $\F_q(t)$. It is conjectured that this could correspond to a form of ``secondary stability'' in the homology of the associated Hurwitz space.

However, to the author's knowledge there has been no other instance of a secondary term for the discriminants of cubic extensions of any global field other than $\Q$ (see \Cref{sec:zhao} regarding the thesis of Zhao).

In this paper, we indeed prove the existence of such a secondary term for the count of cubic extensions of $\F_q(t)$ with absolute norm of discriminant $q^{2N}$. This is the first instance known of a secondary term for cubic extensions of a function field.

\subsection{Main Result}
To state our result let $\F_q$ be a fixed finite field. For a smooth irreducible curve $C/\F_q$, we call a morphism $f: X \to C$ a \textbf{degree $d$ branched cover} if $X$ is a smooth, irreducible curve and $f$ is finite of degree $d$ and generically \'etale. 

The category of degree $3$ branched covers $X\to \PP^1$ is equivalent to the category of finite separable degree $3$ field extensions $K/\Fq(t)$. The discriminant $\Delta:= \Disc_{K/\F_{q}(t)}$ cuts out a divisor on $\PP^1$, which we identify with the \textbf{branch divisor} $B$ on $\PP^1$ from the corresponding cover $f: X\to \PP^1$. There is a nonnegative integer $N$ such that $\deg B = 2N$ \footnote{The degree of the branch divisor is even because of Riemann-Hurwitz, which holds if $X$ is geometrically irreducible. If $X$ is irreducible and geometrically reducible, then $X\to \PP^1$ must have branch degree $0$, which is still even.}; consequently the absolute norm of the discriminant $\Nm(\Delta)$ equals $q^{2N}$.

Let $\Cov_{3}(2N)$ denote the number of degree $3$ branched covers $f: X \to \PP^1$ with branch divisor of degree $2N$, up to isomorphism. (Here an isomorphism of branched covers $f_1:X_1 \to \PP^1$ and $f_2:X_2 \to \PP^1$ is a commuting isomorphism $X_1\xrightarrow{\sim} X_2$.) 

We can equivalently write
\begin{align*}
    \Cov_3(2N) &= |\{[f: X \to \PP^1] \text{ a degree $3$ branched cover}\mid  \deg B = 2N\}| \\
    &= |\{[K/\F_q(t)] \mid  \Nm \left(\Disc_{K/\F_q(t)}\right) = q^{2N}\}|,
\end{align*}
where $[A]$ denotes the isomorphism class of either a branched cover or a field extension.

Our main theorem provides a main term as well as a secondary term in an estimate for $\Cov_{3}(2N)$.

\begin{thm}
\label{thm:main-thm-intro}
Let $N$ be a nonnegative integer, and let $i \in \{0,1,2\}$ such that $N\equiv i\pmod{3}$. The number of degree $3$ extensions $K/\F_q(t)$ with $\Nm(\Disc_{K/\F_q(t)}) =q^{2N}$ is
\[
\Cov_3(2N) = c_1 q^{2N} - c_2^{i} q^{5N/3} + O\left(N^4 q^{3N/2}+1\right),
\]
where $c_1, c_{2}^{0}, c_2^{1},$ and $c_2^{2}$ are explicit constants not depending on $N$. 
\end{thm}

Here each isomorphism class $K/\F_q(t)$ is counted once. We make no assumptions on the characteristic of $\F_q$. For the most precise statement, see \Cref{thm:main-thm}.

Note that instead of a single secondary term, we have a periodic secondary term depending on $N\pmod{3}$.

\subsection{Outline of Approach}

Our method is geometric and depends on a parametrization of cubic algebras over an integral scheme due to Casnati and Ekedahl. This is the geometric version of the classical Delone--Faddeev correspondence \cite{DF40} between cubic rings over $\Z$ and integral binary cubic forms $ax^3 + bx^2y+cxy^2 + dy^3$, where $a,b,c,d \in \Z$. For a cubic cover $X\to \PP^1$, there is analogously a (twisted) binary cubic form on $\PP^1$, meaning a global section $g \in \Sym^{3} \mc{E} \otimes \wedge^2 \mc{E}^{\vee}$ for some rank $2$ vector bundle $\mc{E}$ on $\PP^1$. For a sufficiently nice (Gorenstein) cubic cover $f: X\to \PP^1$, we get this binary cubic form by embedding $X$ into a ruled surface $\PP(\mc{E}) \to \PP^1$. Here $\mc{E} = \coker(\mc{O}_{\PP^1} \to f_{*} \mc{O}_X)^{\vee}$ is the Tschirnhausen bundle associated to the cover $f: X \to \PP^1$. It will turn out that $X\subseteq \PP(\mc{E})$ can be cut out by a section
\[
s \in H^0(\PP(\mc{E}), \mc{M}_{\mc{E}})
\]
where $\mc{M}_{\mc{E}} = \mc{O}_{\PP(\mc{E})}(3) \otimes (\wedge^2 \mc{E})^{\vee}$ is a specific line bundle on $\PP(\mc{E})$. Then $s$ corresponds via pushforward to the associated binary cubic form $g \in \Sym^3\mc{E} \otimes (\wedge^2 \mc{E})^{\vee}$

So it will suffice to count smooth, irreducible curves $X \subseteq \PP(\mc{E})$ of a given specific divisor class. Our strategy is to count all sections of the relevant line bundle on $\PP(\mc{E})$ and subtract out the curves which are singular. This leads us to an inclusion-exclusion sieve.

In order to count sections singular at a specific set of points in $\PP(\mc{E})$, the key construction is an elementary transformation of ruled surfaces. This is a birational transformation from $\PP(\mc{E})$ to a new ruled surface $\PP(\mc{E}')$, given by blowing up at a certain set of points and then blowing down. We will find that sections which are singular at a set of points on $\PP(\mc{E})$ correspond to sections which vanish at a certain set of points on $\PP(\mc{E}')$. 

This is the same method used by Zhao \cite{Zha13} and is the geometric version of the ``discriminant reducing identity'' used in \cite{BST13}.

After carefully keeping track of the combinatorics of the sieve and the birational transformations of surfaces, we are reduced to estimating certain root-counting functions over sections on ruled surfaces $\PP(\mc{E})$.

To describe these functions, we note that over $\PP^1$, each ruled surface is a \textbf{Hirzebruch surface} $F_k = \Proj(\mc{O}\oplus \mc{O}(-k))$ for some nonnegative integer $k$. Our sections can be identified with binary cubic forms
\[
s \in \mc{V}(\ell,k) = \{A_0 x^3 + A_1 x^2y + A_2 xy^2 + A_3 y^3 : \deg A_i = \ell - i k\}
\]
where each $A_i \in \F_q[t_0, t_1]$ is a homogeneous polynomial. Here $(t_0:t_1)$ are the coordinates of the base $\PP^1$, while $(x:y)$ give the coordinates of the $\PP^1$ fibers.

If $s$ is thought of as such a ``bi-homogeneous polynomial'', and $P \in \PP_{[t_0: t_1]}^1$ is a closed point, let $r_P(s)$ denote the number of roots of $s$ mod $P$. For an effective reduced divisor $D$ on $\PP^1$, let $r_D(s) = \prod_{P \mid D} r_P(s)$. Then our root-counting functions 
\[
R^{\ir}(\ell,k,D) = \sum_{s \in \mc{V}(\ell,k)^{\ir}} r_D(s)
\]
add up the number of roots of $s$ over all (irreducible) binary cubic forms in the $\F_q$-vector space $\mc{V}(\ell, k)$.

An equivalent task is to count (irreducible) sections on $F_k$ passing through a finite union of points on $F_k$, one over each point in $D$.

We break up our divisors $D$ into a small range, medium range, and large range as compared to $\ell$ and $k$. 

In the small range (meaning $\deg D$ is small), we can more or less exactly count the number of sections passing through a collection of points. In the large range (meaning $\deg D$ is large), we can simply count section by section and crudely bound the number of roots on each section.

The critical range is the medium range. In this range, there could be more divisors than expected passing through a given set of points in $F_k$. The key result we prove is that if there are more divisors than expected, then all such divisors must be automatically reducible with a fixed component. Then we can separately bound reducible divisors.

Finally, we bound extensions with Galois group $C_3$ and inseparable extensions. Putting together the pieces yields a main term, secondary term, and power-saving error term for the number of cubic extensions of $\F_q(t)$ with absolute norm of discriminant $q^{2N}$.

The constructions in the first half are fairly general, and indeed we establish the geometric setup and reduce our count to a sum of root counting functions for a general function field $\F_q(C)$. Then we reduce to the case $C= \PP^1$ when we start to estimate the root-counting functions $R^{\ir}(\ell,k, D)$ and break up the sieve into three ranges.

\subsection{Relationship to Zhao's Thesis}\label{sec:zhao}
The main result of this paper was claimed to be proven by Zhao \cite{Zha13}. While Zhao's thesis has many great ideas, several of which we use and build off of in the present paper, his approach is unfortunately flawed.

Much of our setup follows the same framework as Zhao. We embed each smooth trigonal curve $X\to \PP^1$ into a Hirzebruch surface $F_k$ fibered over $\PP^1$, and conversely, we may count the number of trigonal curves by sieving out sections on $F_k$ which are singular. To do so we count sections singular at a certain set of points on a Hirzebruch surface $F_k$, which correspond to sections passing through a certain set of points on a new Hirzebruch surface $F_{k'}$ after a birational transformation.

However, our approach to counting diverges at this point. Zhao's approach is to count curves inside of a fixed Hirzebruch surface. He does so using ``pointwise'' error bounds we find to be not true, especially the last sentence of his Theorem 5.3.0.14.

For our strategy, we do not attempt to count smooth trigonal curves inside a fixed Hirzebruch surface, which we find to be too difficult. Instead, we count all trigonal curves in all surfaces. This requires carefully keeping track of the sheaf-theoretic data associated to birational transformations.

Ultimately, this allows us to set up our count as a sieve with a small, medium, and large range. Our setup simplifies the uniformity bound for the large range. In the medium range, both the author and Zhao approach the interpolation problem by restricting line bundles to curves through points, but we consider a wider class of curves. Critically we restrict line bundles to reducible curves, whereas Zhao considers irreducible curves only. This gives us the flexibility needed to characterize exactly which maps $H^0(F_k, \mc{L}) \to H^0(\mc{D}, \mc{O}_D)$ can fail to be surjective in the medium range.

As another minor difference between our work and Zhao's, our geometric setup for counting degree $3$ extensions is valid for a general function field $\F_q(C)$. In the same vein of Gunther's thesis \cite{Gun17}, we can embed degree $3$ branched covers $X\to C$ in a ruled surface $\PP(\mc{E}) \to C$.

\subsection{Sections of the Paper}
In \Cref{sec:param-cubic}, we establish the parametrization of Gorenstein cubic covers of an integral scheme $S$. This version is due to Casnati and Ekedahl (and Miranda), although the parametrization has a long history with many contributors. In particular, we will focus on the correspondence between triple covers $f:X \to S$ and sections $s$ of a certain line bundle on $\PP(\mc{E})$, where $\mc{E}$ is a rank $2$ vector bundle on $S$. 

In \Cref{sec:sieve-i} we begin to set up the geometric sieve for sections on ruled surfaces $\PP(\mc{E})$ over a nice (smooth, projective, geometrically irreducible) curve $C$. We introduce the notation $\Theta(C,N)$ for a slightly nicer count for degree $3$ branched covers which in particular weights inversely by automorphisms.

In \Cref{sec:ruled-elm}, we construct the blowup blowdown construction for an appropriate collection of points $\mc{D}$ on a ruled surface $\PP(\mc{E})$.

In \Cref{sec:sieve-ii}, we use the results of \Cref{sec:ruled-elm} to reduce our count to a sum of root-counting functions.

Starting in \Cref{sec:hirz} we restrict to the case $C=\PP^1$. We translate our general results for triple covers of a nice curve $C$ to triple covers of $\PP^1$. In particular, we discuss Hirzebruch surfaces $F_k$ and their sections.

In \Cref{sec:restrictions-line-bundles} we set up some discussion about restrictions of line bundles to curves on a Hirzebruch surface. This will be used to count sections in the medium range in the following section.

In \Cref{sec:root-counting} we estimate the root-counting function $R^{\ir}(\ell, k, D)$ and especially an adjusted root-counting function $\Phi^{\ir}(\ell,k,D)$ in the small, medium, and large ranges.

In \Cref{sec:smooth-counting} we put together the pieces our estimates for $\Phi^{\ir}(\ell, k,D)$ to estimate $\Theta(N)$ by a model function $\widehat{\Theta}(N)$ plus an error. We get a main and secondary term out of $\widehat{\Theta}(N)$ using a generating function. Then after using bounds for $C_3$ and inseparable extensions, we finally get our secondary term for discriminants of cubic extensions of $\F_q(t)$.

In \Cref{sec:c3-insep}, we bound the number of $C_3$ and inseparable extensions of $\PP^1$.

We also include three appendices. 

In \Cref{appendix}, we discuss some generalities on discriminants and their behavior under the Casnati-Ekedahl correspondence.

In \Cref{sec:appendix-hor-finite} we discuss the notion of horizontal reducibility of a section and prove a few technical lemmas regarding finiteness of sums.

Finally, in \Cref{sec:appendix-a} we prove some results about the elementary transformation. We show that as a functor, it induces an equivalence of categories, and then we construct the corresponding blowup-blowdown of projective bundles $\PP(\mc{E})$.

\subsection{Notation and Conventions}
\hfill

\noindent\textbf{Generalities}

\begin{itemize}
    \item We fix a finite field $\F_q$ once and for all.
    \item We write $\PP^1 = \PP^1_{\F_q}$.
\end{itemize}
\textbf{Analysis}

\begin{itemize}
\item We write $A = O(B)$ or $A\ll B$ if there is a uniform constant $c$ such that $|A| \le c B$ always holds. If there is such a constant $c = c(b)$ which depends on a parameter $b$, then we write $A = O_b(B)$ or $A \ll_b B$.
\item In our big $O$ notation, we consider $q$ to be a fixed constant. Our bounds are uniform in all other variables.
\end{itemize}
\textbf{Schemes}
\begin{itemize}
    \item All of our schemes will be Noetherian.
    \item We call a scheme $X/\F_q$ \textbf{smooth} if the structure map $X \to \Spec \F_q$ is smooth.
    \item A \textbf{curve} over a base field $k$ is a separated, reduced, finite type $k$-scheme of pure dimension $1$.
    \item We say a curve $C/\F_q$ is a \textbf{nice curve} if it is smooth, projective, and geometrically irreducible.
    \item If $f: X \to Y$ is a blowup morphism and $Z\subseteq Y$ is a closed subscheme, we write $Z^{st}$ for the strict transform of $Z$.
    \item We define a \textbf{Dedekind scheme} to be an  integral, regular, separated (Noetherian) scheme of dimension $1$.
\end{itemize}
\textbf{Hirzebruch Surfaces}
\begin{itemize}
    \item A Hirzebruch surface $\pi:\Fk \to \PP^1$ is defined by $\Fk:= \PP(\mc{O}_{\PP^1}\oplus \mc{O}_{\PP^1}(-k))$.
    \item We denote the line bundle $\mc{O}(m, \ell):= \mc{O}_{F_k}(m) \otimes \pi^{*} \mc{O}(\ell)$ on $F_k$.
    \item We denote the $\F_q$-vector space $\mc{V}(\ell,k):= H^0(F_k, \mc{O}(3,\ell))$.
    \item We denote $\Gamma_k:= \Aut(\mc{O}_{\PP^1}\oplus \mc{O}_{\PP^1}(-k))$.
\end{itemize}
\textbf{Reduced Divisors}

Let $C$ be a nice (smooth, projective, geometrically irreducible) curve over $\F_q$. A reduced effective divisor $D\subseteq C$ can be thought of as a union of closed points $P \in C$.

\begin{itemize}
    \item If $P$ is a closed point in $C$, we write $P\mid D$ to say $P$ is one of the closed points in $D$.
    \item We let $\omega(D)$ denote the number of closed points $P$ in $D$.
    \item We let $\mu(D):= (-1)^{\omega(D)}$ denote the M\"obius function of $D$.
    \item If $D_1$ and $D_2$ are disjoint reduced effective divisors, we write $D_1+D_2$ for their sum as divisors. 
    \item If $D_1\mid D_2$, we let $D_2-D_1$ denote their difference as divisors
    \item We denote the by $Z_C(T)$ formal power series 
    \[
    Z_{C}(T):=\prod_{P \in C} (1- T^{\deg P})^{-1}
    \].
    It satisfies 
    \[
    Z_C(T) = \sum_{D} T^{\deg D} = \left(\sum_{D} \mu(D) T^{\deg D}\right)^{-1}.
    \]
    We write $Z(T) = Z_{\PP^1}(T)$.
\end{itemize}

\textbf{Groupoids}

A \textbf{groupoid} is a category for which all morphisms are isomorphisms. Let $\mc{C}$ be an essentially small groupoid. We write $|\mc{C}|$ for the set of isomorphism classes of objects in $C$.

We write $\sum_{A \in |C|}$ to denote the sum over all isomorphism classes of objects in $C$.

\noindent\textbf{Vector Bundles and Locally Free Sheaves}

Let $S$ be a scheme. A \textbf{rank $r$ vector bundle} on $S$ is a locally free sheaf of rank $r$. A \textbf{line bundle} on $S$ is a vector bundle of rank $1$ (equivalently an invertible $\mc{O}_S$-module).

For a finite rank vector bundle $\E$ on $S$, the \textbf{dual bundle} $\E^{\vee}$ is the sheaf 
\[
\E^{\vee}:=\Hom_{S}(\E, \mc{O}_S).
\]

When $\mc{L}$ is a line bundle, we will interchangeably write $\mc{L}^{\otimes -1} = \mc{L}^{\vee}$.

If $\mc{F}$ is a quasi-coherent sheaf on $S$, we will sometimes write $s \in \mc{F}$ to mean $s$ is a global section of $\mc{F}$ on $S$ when the scheme $S$ is clear from context.

For a nice curve $C$, we let $\Bun(C)$ denote the groupoid of rank $2$ vector bundles $\mc{E}$ on $C$.

\noindent\textbf{Projective Bundles}

For a scheme $S$ and a quasi-coherent $\mc{O}_S$-module $\E$, we define the \textbf{projective bundle} 
\[
\PP(\E):= \Proj(\Sym^{\bullet}(\E))
\]
following the convention in (say) \cite{EGA}. The bundle comes equipped with a structure morphism $\pi: \PP(\E)\to S$. For an $S$-scheme $X\xrightarrow{f} S$, the set of $S$-morphisms $X\to \PP(\E)$ is in bijective correspondence with the set of pairs $(\phi, \LL)$ up to equivalence, where $\LL$ is a line bundle on $X$ and $\phi: f^{*}\E \to \LL$ is a surjection of sheaves.

\noindent\textbf{Discriminants}

A \textbf{discriminant} on a scheme $S$ is a pair $(\mc{N}, \Delta)$, where $\mc{N}$ is a line bundle on $S$ and $\Delta \in \mc{N}^{\otimes 2}$ is a global section.

Two discriminants $(\mc{N},\Delta)$ and $(\mc{N}', \Delta')$ are equivalent if there is an isomorphism $\mc{N} \xrightarrow{\sim}\mc{N}$ sending $\Delta$ to $\Delta'$.

If $S$ is an integral, regular (Noetherian) scheme, and $s$ is nonzero, then the vanishing locus of $s$ is an effective Cartier divisor $D\subseteq S$, which we call the \textbf{discriminant divisor}.

Furthermore, if $S$ is a smooth, projective irreducible curve over some base field, and $\Delta \in \mc{N}^{\otimes 2}$ is a nonzero discriminant, then the discriminant divisor $D := V(\Delta)$ has a well-defined degree. We call this the \textbf{discriminant degree} $\deg\Delta$, and it is equal to $2 \deg \mc{N}$. When $\Delta=0$, we \textit{define} the discriminant degree to be $2 \deg \mc{N}$.

\noindent\textbf{Miscellaneous Common Notation}

Much of this notation will be defined later and commonly used throughout the thesis.

\begin{itemize}
    \item $C$ - a nice (smooth, projective, geometrically irreducible) curve, usually over $\F_q$
    \item $S$ - an integral scheme, sometimes assumed to be a Dedekind scheme
    \item $\mc{E}$ - a degree $2$ vector bundle on $C$ or $S$, with $\pi: \PP(\mc{E})\to S$ or $\pi: \PP(\mc{E}) \to C$.
    \item $D$ - a reduced effective divisor on $C$ or $S$, or a union of closed points $P$
    \item $\DD$ - a reduced $0$-dimensional closed subscheme of $\PP(\mc{E})$, usually with $\pi|_{\DD} : \DD \to D$ an isomorphism
    \item Degree $3$ branched cover - a map $f:X\to C$ finite of degree $d$, generically \'etale, and $X$ is smooth and irreducible
    \item $\Cov_3(C, 2N)$ - number of degree $3$ branched covers with branch degree $2N$ (up to isomorphism)
    \item Triple cover - a map $f: X\to S$ finite, flat of degree $3$
    \item $\mc{V}(\mc{E}, \mc{L}) := H^0(\PP(\mc{E}), \mc{O}(3) \otimes \pi^{*} \mc{L})$.
    \item $\mc{V}(\mc{E}):= \mc{V}(\mc{E}, (\wedge^2 \mc{E})^{\vee})$, which is isomorphic to $\Sym^3 \mc{E} \otimes (\wedge^2\mc{E})^{\vee}$
    \item $\Delta_f$ and $\Delta(s)$ - discriminant of a triple cover $f: X\to C$ or a section $s \in \mc{V}(\mc{E})$
    \item $\Theta(C,N) = \sum_{f: X\to C}\frac{1}{|\Aut(f)|}$ over triple covers with $X$ smooth, irreducible and $\deg \Delta_f= 2N$
    \item $r_D(s):=$ the number of roots of $s|_{D}$ for $s \in \mc{V}(\mc{E})$.
    \item $a_D(s):= \prod_{P \in D} (r_P(s)-1)$.
    \item $\mc{V}(\mc{E}, \mc{L})^{\ir}$ or $\mc{V}(\mc{E})^{\ir}$ - the subset of $s$ which are horizontally irreducible
    \item $\mc{V}(\mc{E}, \mc{L})^{\sm, \ir}$ or $\mc{V}(\mc{E})^{\sm,\ir}$ - the subset of $s$ which are smooth and horizontally irreducbile (equivalently smooth and irreducible)
    \item $\Van(\mc{E}, \DD)$ and $\Sing(\mc{E},\DD)$ - sections which vanish to order at least $1$ or at least $2$ on $\DD$, respectively
    \item $R^{\ir}(\mc{E},D):= \sum_{s \in \mc{V}(\mc{E})^{\ir}} r_D(s)$, and $\Phi^{\ir}(\mc{E},D) := \sum_{s \in \mc{V}(\mc{E})^{\ir}} a_D(s)$.
    \item $\Mar(\mc{E}, D)$ - the set of $D$-markings $\DD \subseteq \PP(\mc{E})$. When $\mc{E}$ is clear from context, we write $\Mar(D) = \Mar(\mc{E}, D)$
    \end{itemize}
On $C = \PP^1$,
\begin{itemize}
    \item $\mc{V}(\ell, k):= H^0(F_k, \mc{O}(3,\ell)) = H^0(F_k, \mc{O}(3) \otimes \pi^{*} \mc{O}(\ell))  \cong \mc{V}(\mc{O}(\ell - k) \oplus \mc{O}(\ell - 2k))$
    \item $\mc{V}(\ell,k)^{\alpha}$ for $\alpha \in \{\ir, \hr, \sr, \xir, \nz, \z\}$ - subset of sections $s$ which are horizontally irreducible, (nonzero) horizontally reducible, specially reducible, x-irreducible, nonzero, zero.

    \item $R^{\alpha}(\ell,k,D) := \sum_{ s\in \mc{V}(\ell,k)^{\alpha}} r_D(s)$, $\Phi^{\alpha}(\ell,k,D):= \sum_{s \in \mc{V}(\ell,k)^{\alpha}} a_D(s)$.
    \item $\Theta(N):= \Theta(\PP^1, N)$ and $G(T):= \sum_{N} \Theta(N) T^{N}$
    \item $\Psi(N):= \sum_{\deg D + N' = N} \sum_{(\ell,k) \in \Par(N')}\frac{1}{|\Gamma_k|}\mu(D)\Phi^{\alpha}(\ell,k,D)$
    \item $\widehat{R}^{\alpha}(\ell,k,D)$, $\widehat{\Phi}^{\alpha}(\ell,k,D)$, $\widehat{\Theta}(N)$,$\widehat{G}(T)$ ,$\widehat{\Psi}(N)$ - ``model function'' estimates
\end{itemize}

\subsection{Acknowledgments}

We would like to thank Melanie Matchett Wood for helpful mathematical input and feedback during the writing process.

We would also like to acknowledge Aaron Landesman for introducing the author to the geometric parametrization of degree $n$ covers.

This work was partially supported by the National Science Foundation Graduate Research Fellowship under Grant No. DGE 2140743.
% \tableofcontents
\section{Parametrization of Cubic Covers}

\label{sec:param-cubic}
In this section, we detail the parametrization of triple covers of an integral scheme $S$. 

Every Gorenstein triple cover of $S$ corresponds to a certain (twisted) binary cubic form, meaning a global section $s \in \Sym^{3}\E \otimes (\wedge^{2} \E)^{\vee}$ for some rank $2$ vector bundle $\E$ on $S$. This version of the correspondence that we will use is due to Miranda and Casnati-Ekedahl (\cite{Mir85}, \cite{CE96}).

Let $S$ be an integral (Noetherian) scheme.
\subsection{Triple Covers}
\begin{defn}
    A \textbf{triple cover} $f: X\to S$ is a finite, flat degree $3$ morphism.
\end{defn}

Equivalently, a triple cover is a finite locally free degree $3$ morphism $X\to S$ (see \cite[\href{https://stacks.math.columbia.edu/tag/02K9}{Tag 02K9}]{stacks-project}).

\begin{defn}
Let $\Trip(S)$ be the category with objects given by triple covers $f:X\to S$. The set of morphisms from $(X_1\xrightarrow{f_1} S)$ to $(X_2\xrightarrow{f_2} S)$ is the set of isomorphisms $g: X_1\to X_2$ such that $f_2 \circ g = f_1$.
\end{defn}
(In this category, every morphism is an isomorphism, so $\Trip(S)$ is a groupoid.)

Let $f: X \to S$ be a triple cover. This means $f_{*} \mc{O}_X$ is a rank $3$ locally free $\mc{O}_S$-module. There is an associated \textbf{discriminant} $\Delta_f$, which is defined to be a section of a certain line bundle on $S$:
\[
\Delta_f\in H^0(S, (\wedge^3 f_{*} \mc{O}_X)^{\otimes -2}).
\]
(See \cite[\href{https://stacks.math.columbia.edu/tag/0BVH}{Section 0BVH}]{stacks-project}.) 

Now assume $f:X \to S$ is a \textit{Gorenstein} triple cover, meaning $f$ is a Gorenstein morphism. Following \cite{CE96} (see also \cite{LVW24}), we define the \textbf{Tschirnhausen bundle} $\E:=\coker(\mc{O}_S \to f_{*} \mc{O}_X)^{\vee}$, which turns out to be a rank $2$ vector bundle.

We can dualize the exact sequence
\[
0 \to \mc{O}_S \to f_{*} \mc{O}_X \to \E^{\vee} \to 0
\]
to obtain the exact sequence
\[
0 \to \E \to f_{*} \omega_{X/S} \to \mc{O}_S \to 0
\]
where the \textbf{relative dualizing sheaf} $\omega_{X/S}$ is an invertible sheaf on $X$ because $f$ is a Gorenstein morphism. (See \cite{CE96}, Section 1.) Now one can show that the composition $f^{*} \E \to f^{*} f_{*} \omega_{X/S} \to \omega_{X/S}$ is a surjection. By universal property, this induces a map $i: X\to \PP(\E)$ such that the diagram below commutes.
\[
\begin{tikzcd}
    X \ar[r, "i"] \ar[rd, "f"] &\PP(\E)\ar[d, "\pi"]\\
    &S
\end{tikzcd}
\]

This is known as the \textbf{relative canonical embedding} of $X$.
\begin{defn}
Let $\Trip^{\Gor}(S)$ the full subcategory of $\Trip(S)$ given by \textit{Gorenstein} triple covers.
\end{defn}

\subsection{Degree $3$ Sections and Binary Cubic Forms}

We would first like to recall a property of projective bundles. 
\begin{prop}[\protect{\cite[\S III.8]{Har77}}]
    \label{prop: proj-sym}
    Let $S$ be a Noetherian scheme and let $\E$ be a rank $2$ vector bundle with associated projective bundle $\pi : \PP(\E) \to S$. Let $m\ge 0$ be an integer and let $\LL$ be a line bundle on $S$. Then $\pi_{*} \mc{O}_{\PP(\E)}(m) \cong \Sym^{m} \E$ and we have a canonical isomorphism 
    \[
    \Lambda: H^{0}(\PP(\E), \mc{O}_{\PP(\mc{E})}(m) \otimes \pi^{*} \LL)\xrightarrow{\sim} H^{0}(S, \Sym^{m} \mc{E} \otimes \mc{L}).
    \]
\end{prop}

\begin{defn}
    A \textbf{degree $3$ section} (or a \textbf{trisection}) is a pair $(\mc{E},s)$, where $\mc{E}$ is a rank $2$ vector bundle on $S$, $\pi: \PP(\mc{E}) \to S$ is the associated projective bundle, and
    \[
    s \in H^0(\PP(\mc{E}), \mc{O}_{\PP(\mc{E})}(3) \otimes \pi^{*} (\wedge^2 \mc{E})^{\vee}).
    \]
\end{defn}
\begin{defn}
    Let $\TriSec(S)$ be the category with objects given by degree $3$ sections $(\mc{E}, s)$. A morphism from $(\mc{E}_1, s_1)$ to $(\mc{E}_2, s_2)$ is an isomorphism $\mc{E}_1 \xrightarrow{\sim} \mc{E}_2$ sending $s_1$ to $s_2$.
\end{defn}
\begin{defn}
    A \textbf{(twisted) binary cubic form} (or simply a \textbf{binary cubic form}) is a pair $(\mc{E},g)$, where $\mc{E}$ is a rank $2$ vector bundle on $S$, and
    \[
    g \in H^0(\Sym^{3} \mc{E} \otimes (\wedge^2 \mc{E})^{\vee}).
    \]
\end{defn}
\begin{defn}
    Let $\CubForm(S)$ be the category with objects given by binary cubic forms $(\mc{E}, g)$. A morphism from $(\mc{E}_1, g_1)$ to $(\mc{E}_2, g_2)$ is an isomorphism $\mc{E}_1 \xrightarrow{\sim} \mc{E}_2$ sending $g_1$ to $g_2$.
\end{defn}

By Proposition $2.1$, there is a canonical equivalence between the categories $\TriSec(S)$ and $\CubForm(S)$. By abuse of notation, we denote the functor inducing the equivalence by 
\[
\Lambda: \TriSec(S) \xrightarrow{\sim} \CubForm(S).
\]

Now consider a binary cubic form $(\mc{E},g) \in \CubForm(S)$ defined by the global section $g \in H^{0}(S, \Sym^{3} \mc{E} \otimes (\wedge^2 \mc{E})^{\vee})$.

We explain why we call $g$ a (twisted) binary cubic form, following (for example) the convention in \cite{Woo11}. This is because locally, on a sufficiently small affine open $\Spec R \subseteq S$ where $\E$ trivializes as 

\[
\E|_{R} \cong R x \oplus R y,
\]
we have that $g$ restricts to some (twisted) binary cubic form
\[
(A_0 x^3 + A_1 x^2 y + A_2 xy^2 + A_3 y^3)\otimes (x\wedge y)^{\otimes -1}
\]
with $A_0, A_1, A_2, A_3 \in R$.

For a (twisted) binary cubic form $g \in \Sym^{3} \mc{E} \otimes (\wedge^2\mc{E})^{\vee}$, there is an associated \textbf{discriminant} $\Delta(g) \in (\wedge^2 \mc{E})^{\otimes 2}$. (See \Cref{appendix} for details.) If $\Lambda(s) = g$, we also write $\Delta(s) \in (\wedge^2 \mc{E})^{\otimes 2}$ to refer to the discriminant $\Delta(g)$, by abuse of notation.

Let $(\mc{E}, s) \in \TriSec(S)$. We say $s\in H^0(\PP(\mc{E}), \mc{O}_{\PP(\mc{E})}(3) \otimes \pi^{*} (\wedge^2 \mc{E})^{\vee})$ is \textbf{primitive} if $s$ does not vanish on any fiber above a closed point $P \in S$.

Similarly, if $(\mc{E},g) \in \CubForm$, we say $g \in \Sym^3 \mc{E} \otimes (\wedge^2 \mc{E})^{\vee}$ is \textbf{primitive} if the restriction $g|_{P}$ is nonzero for every closed point $P \in S$.
\begin{defn}
    Let $\TriSec^{\Prim}(S)$ be the full subcategory of $\TriSec(S)$ given by pairs $(\mc{E},s)$ where $s$ is primitive.
\end{defn}
\begin{defn}
    Let $\CubForm^{\Prim}(S)$ be the full subcategory of $\CubForm(S)$ given by pairs $(\mc{E},g)$ where $g$ is primitive.
\end{defn}
Of course, $\Lambda$ restricts to an equivalence of categories 
\[
\TriSec^{\Prim}(S) \xrightarrow{\sim} \CubForm^{\Prim}(S).
\]

\subsection{The correspondence}
Given a pair $(\E, s) \in \TriSec^{\Prim}(S)$, we can assign to $s$ the vanishing scheme $X:= V(s)\subseteq \PP(\E)$. Because $s$ does not vanish on the fiber above any point of $S$, we have $\dim X_P=0$ for every closed point $P \in S$. It turns out that $X\to S$ is Gorenstein triple cover.

\begin{thm}[Casnati-Ekedahl, \cite{CE96}]
    \label{thm:CE}
    Let $S$ be an integral (Noetherian) scheme.
    There is a functor $V$ of categories
    \[
    V: \TriSec^{\Prim}(S) \xrightarrow{\sim} \Trip^{\Gor}(S)
    \]
    sending $(\E, s)$ to the map $f:X\to S$, where $X:= V(s)\subseteq \PP(\E)$, and $f:= \pi \circ i$, with $i: X\to \PP(\E)$ the inclusion map.
    
    Furthermore, $V$ is an equivalence of categories, and so induces a three-fold equivalence of categories
    \[
    \begin{tikzcd}
        &\TriSec^{\Prim}(S) \ar{ld}[swap]{V} \ar{rd}{\Lambda}& \\
    \Trip^{\Gor}(S)&&\CubForm^{\Prim}(S)\ar{ll}[swap]{\sim}
    \end{tikzcd}
    \]
    The inverse functor to $V$ sends a Gorenstein triple cover $f: X\to S$ to a section cutting out its relative canonical embedding 
\[
\begin{tikzcd}
    X \ar[r, "i"] \ar[rd, "f"] &\PP(\E)\ar[d, "\pi"]\\
    &S
\end{tikzcd}
\]
where $\mc{E}:= \coker(\mc{O}_S \to f_{*} \mc{O}_X)^{\vee}$ is the Tschirnhausen bundle of $f$.
\end{thm}
\begin{rem}
    We note that that inverse functor to $V$ is well-defined. Although a section $s$ cutting out $X \subseteq \PP(\mc{E})$ is only determined up to units, any two such sections are equivalent up to the action of $\Aut(\mc{E})$. So the pair $(\mc{E}, s) \in \TriSec^{\Prim}(S)$ corresponding to a triple cover is well-defined up to isomorphism.
\end{rem}

\begin{rem}
    In fact there is an equivalence $\CubForm(S) \xrightarrow{\sim} \Trip(S)$ even without restricting to primitive forms and Gorenstein covers. In the case $S = \Spec \Z$, the bijection between binary cubic forms and cubic rings was known due to Levi, Delone-Fadeev, Davenport-Heilbronn, and Gan-Gross-Savin (\cite{Lev14},\cite{DF40},\cite{DH71},\cite{GGS02}). For a general base scheme $S$, see Poonen \cite{Poo08} and Deligne \cite{Del07}.
    
    However, the general construction of a cover $f:X\to S$ from a binary cubic form $(\mc{E},s)$ is somewhat more complicated than setting $X$ to be the vanishing scheme $V(s)\subseteq \PP(\mc{E})$. See Wood \cite{Woo11} for more details.

    In any case, we will only need the version of the correspondence in \Cref{thm:CE}. We note that Casnati-Ekedahl show that the composition $V\circ \Lambda^{-1}$ is a bijection on isomorphism classes. The fact that the composition $V\circ \Lambda^{-1}$ is an equivalence of categories (that is, an equivalence of groupoids) between $\CubForm^{\Prim}(S)$ and $\Trip^{\Gor}$ is shown by Poonen \cite{Poo08} and Deligne \cite{Del07} (see also Wood \cite{Woo11}). Therefore, the map $V$ is an equivalence as well. More generally, the bijective  correspondence $V$ on isomorphism classes is functorial in $S$ (due again to \cite{Poo08}, \cite{Del07}, \cite{Woo11}).

    Classically, the bottom correspondence between triple covers and binary cubic forms has been emphasized. However, it will be most useful for us to think of our objects as elements of $\TriSec(S)$, so we choose to highlight this aspect of the correspondence as well.
\end{rem}

We furthermore note that the equivalence of categories in \Cref{thm:CE} preserves discriminants. More precisely, 

\begin{prop}\label{prop:disc-preserve}
    Let $S$ be an integral scheme, let $(\mc{E}, g) \in \CubForm^{\Prim}(S)$ be a primitive binary cubic form, and let $f: X \to S$ be the corresponding Gorenstein triple cover. Let $\Delta(g) \in H^0(S, (\wedge^2 \mc{E})^{\otimes 2})$ and $\Delta_f \in H^0(S,(\wedge^3 f_{*}\mc{O}_X)^{\otimes -2})$ be the discriminants of $g$ and $f$ respectively.
    
    Then there exists a canonical isomorphism
    \[
    \delta: (\wedge^{3} f_{*}\mc{O}_X)^{\otimes -1}\xrightarrow \sim (\wedge^2 \mc{E})
    \]
    such that $\delta^{\otimes 2}(\Delta_f) = \Delta(g)$.
    
\end{prop}
\begin{proof}
    See \Cref{appendix}.
\end{proof}

\subsection{Curves and Smooth Irreducible Covers}
Let $C$ be a smooth, projective, geometrically irreducible curve over $\F_q$. We call this a \textbf{nice curve}.

Let $\Trip^{\sm, \ir}(C)$ be the full subcategory of triple covers $f: X \to C$ such that $X$ is smooth and irreducible.

Let $\TriSec^{\sm, \ir}(C)$ be the full subcategory of pairs $(\mc{E}, s)$ such that $s \in H^0(\PP(\mc{E}), \mc{O}_{\PP(\mc{E})} \otimes \pi^{*} (\wedge^2 \mc{E})^{\vee})$ is nonzero and the vanishing scheme $V(s)\subseteq \PP(\mc{E})$ is smooth and irreducible.

Then we have
\begin{prop}
    The equivalence of categories in \Cref{thm:CE} restricts to an equivalence of full subcategories
    \[
    \TriSec^{\sm, \ir}(C) \xrightarrow{\sim} \Trip^{\sm, \ir}(C).
    \]
\end{prop}
\begin{proof}
Implicit in the statement is the claim that if $f: X\to C$ is in $\Trip^{\sm, \ir}(C)$, then $f:X \to C$ is Gorenstein, and if $(\mc{E}, s) \in \CubForm^{\sm, \ir}(C)$, then $s$ is primitive. 

For the former claim, if $X$ is a regular scheme, then $X$ is Gorenstein, and since $f:X \to C$ is flat we conclude $f$ is Gorenstein (see \cite[\href{https://stacks.math.columbia.edu/tag/0C12}{Tag 0C12}]{stacks-project}). For the latter claim, if $s\neq 0$ is not primitive, then $V(s)$ cannot be irreducible, as it contains a component at a fiber.

Now suppose $(\mc{E},s) \in \TriSec^{\sm, \ir}(C)$ satisfies $s \neq 0$ is primitive and $V(s)$ is smooth and irreducible. Then the corresponding cover is $f: X \to C$ with $X = V(s)$, so $X$ is smooth and irreducible.

Conversely, if $f: X \to C$ is a triple cover with $X$ smooth and irreducible, then in particular $f$ is Gorenstein, so $X\cong V(s)$ for some $(\mc{E},s)$ with $s$ primitive. This implies that this $(\mc{E},s)$ lies in $\TriSec^{\sm, \ir}(C)$.

\end{proof}

\section{Sieving for Smoothness}
\label{sec:sieve-i}

Let $C$ be a nice (smooth, projective, geometrically irreducible) curve over $\F_q$. Our goal is to find $\Cov_3(C,2N)$, the number of degree $3$ branched covers $f: X\to C$ of discriminant degree $2N$ up to isomorphism. We define a slightly more manageable approximation to $\Cov_3(C,2N)$.
\begin{defn}
Let $C$ be a nice curve over $\F_q$, and let $N$ be an integer. Define
\[
\Theta(C,N):= \sum_{\substack{f: X \to C\\ \deg \Delta_f = 2N\\ X\text{ smooth, irreducible}}} \frac{1}{|\Aut(f)|}
\]
where the sum ranges over all triple covers $f: X\to C$ such that $X$ is smooth and irreducible.
\end{defn}

 There are two differences between $\Theta(C,N)$ and $\Cov_3(C,2N)$. 
 
 First, the sum in $\Theta(C,N)$ is inversely weighted by automorphisms, and in particular covers with cyclic automorphism group $\Aut(f) \cong C_3$ are counted by a factor of $1/3$ in $\Theta(C,N)$. 
 
 Second, the sum in $\Theta(C,N)$ includes triple covers $f:X \to C$ such that $X$ is smooth and irreducible but $f$ is not generically \'etale. This corresponds to the situation of a (purely) inseparable map $X\to C$. (In this case, $\Delta_f=0$, but it still has a well-defined discriminant degree $\deg \Delta_f$ as the zero section of a line bundle $\mc{N}^{\otimes 2}$ with $\deg \mc{N} = N$.) This situation only occurs when $\Char \F_q =3$.

 Both cases contribute a small error to the final sum. We address both sources of error (in the case $C= \PP^1$) in \Cref{sec:c3-insep}.
\subsection{Sections Corresponding to Covers} 
Given the correspondence of \Cref{sec:param-cubic}, every triple cover $f: X\to C$ for which $X$ is smooth and irreducible corresponds to a section
\[
s \in H^0(\PP(\mc{E}), \mc{O}_{\PP(\mc{E})}(3) \otimes \pi^{*}(\wedge^2 \mc{E})^{\vee})
\]
for some unique (up to isomorphism) rank $2$ vector bundle $\mc{E}$ on $C$ with associated projective bundle $\pi: \PP(\mc{E}) \to C$. For the sake of brevity, we use the notation
\[
\mc{M}_{\mc{E}} := \mc{O}_{\PP(\mc{E})}(3) \otimes \pi^{*}(\wedge^2\mc{E})^{\vee}
\]
for the distinguished line bundle on $\PP(\mc{E})$, and
\[
\mc{V}(\mc{E}):= H^0(\PP(\mc{E}), \mc{M}_{\mc{E}})
\]
for the $\F_q$-vector space of global sections of $\mc{M}_{\mc{E}}$ on $\PP(\mc{E})$.

There is a canonical isomorphism $\Lambda: \mc{V}(\mc{E}) \xrightarrow{\sim} H^0(C, \Sym^{3} \mc{E} \otimes (\wedge^2 \mc{E})^{\vee})$.

It will also be useful to denote
\[
\mc{V}(\mc{E}, \mc{L}):= H^0(\PP(\mc{E}), \mc{O}_{\PP(\mc{E})}(3) \otimes \pi^{*} \mc{L})
\]
when $\mc{E}$ is a rank $2$ vector bundle on $C$ and $\mc{L}$ is a line bundle on $S$. (So in particular, we recover $\mc{V}(\mc{E}) = \mc{V}(\mc{E}, (\wedge^2 \mc{E})^{\vee})$ by setting $\mc{L} = (\wedge^2 \mc{E})^{\vee}$.)

If a nonzero section $s \in \mc{V}(\mc{E}, \mc{L})$ corresponds to a vanishing scheme $X_0\subseteq \PP(\mc{E})$, then we would like to identify the conditions under which $X_0$ is smooth. In particular, we would like to know when $X_0$ is reduced and irreducible. This could fail to be the case for two reasons. First, $s$ could vanish at a fiber, which would mean $X_0$ contains a ``vertical'' fiber component. Second, $X_0$ could contain multiple ``horizontal'' components (or $X_0$ could be nonreduced along a ``horizontal'' component). We would like to more precisely define a notion of horizontal reducibility which encompasses the latter situation.

\subsection{Horizontal Reducibility}\label{sec:hor-ir-i}

Let $C$ be a nice (smooth, projective, geometrically irreducible) curve over $\F_q$. Let $\mc{E}$ be a rank $2$ vector bundle on $C$ with corresponding projective bundle $\PP(\mc{E})$. Let $\mc{L}$ be a line bundle on $C$, and let $\mc{M}:= \mc{O}_{\PP(\mc{E})}(m) \otimes \pi^{*} \mc{L}$ be a line bundle on $\PP(\mc{E})$.

Let $s \in H^0(\PP(\mc{E}), \mc{M})$ be a nonzero section. Let $X_0:=V(s)\subseteq \PP(\mc{E})$ be the vanishing scheme of $s$. Note that $X_0$ is an effective Cartier divisor, which we can also think of as a Weil divisor on the smooth projective surface $\PP(\mc{E})$. We write $[X_0]$ for the Weil divisor associated to $X_0$, such that
\[
[X_0] = \sum a_i [Y_i]
\]
where $Y_i$ are integral closed $1$-dimensional subschemes of $\PP(\mc{E})$. We say the $[Y_i]$ are the components of $[X_0]$, and the integers $a_i>0$ are the multiplicities of the components.

For each integral closed $1$-dimensional subscheme $Y_i\subseteq \PP(\mc{E})$, the map $Y_i \to C$ is proper, and therefore its scheme-theoretic image is either $C$ or a closed point in $C$ (since it is a reduced closed subscheme of $C$). We say $Y_i$ is \textbf{horizontal} if its image is $C$. Equivalently, $Y_i$ is horizontal if $\mc{O}(Y_i)\cong \mc{O}(m_i) \otimes \pi^{*} \mc{L}_i$ for some $m_i>0$ and some line bundle $\mc{L}_i$ on $C$.

Otherwise, we say $Y_i$ is \textbf{vertical}. In this case, the image of $Y_i$ must be a closed point $P \in C$, and $Y_i$ is the fiber $\PP(\mc{E})_{P}\cong \PP^1_{P}$.

\begin{defn}
    Let $s \in H^0(\PP(\mc{E}), \mc{O}_{\PP(\mc{E})}(m) \otimes \pi^{*} \mc{L})$. Let $X_0\subseteq \PP(\mc{E})$ be the vanishing scheme of $s$. We say $s$ is \textbf{horizontally irreducible} if $s\neq 0$ and $[X_0]$ has exactly one horizontal component $[Y]$ of multiplicity $1$.

    Otherwise, we say $s$ is \textbf{horizontally reducible}.
\end{defn}

We will primarily be focused on sections of line bundles of the form $\mc{O}_{\PP(\mc{E})}(3) \otimes \pi^{*}\mc{L}$. In this case, we have several equivalent definitions of horizontal reducibility.

\begin{prop}\label{prop:horiz-ir}
    Let $C$ be a nice curve. Let $\mc{E}$ be a rank $2$ vector bundle on $C$, and let $\mc{L}$ be a line bundle on $C$.
    Let $s \in H^0(\PP(\mc{E}), \mc{O}_{\PP(\mc{E})}(3) \otimes \pi^{*} \mc{L})$. Let $X_0: = V(s) \subseteq \PP(\mc{E})$ be the vanishing scheme of $s$. The following are equivalent to $s$ being \textbf{horizontally reducible}:
    \begin{itemize}
    \item Either $s=0$ or there is an integral component $Y_i$ of $X_0$ such that $\mc{O}(Y_i) \cong \mc{O}_{\PP(\mc{E})}(1)\otimes \pi^{*} \mc{L}_i$ for some line bundle $\mc{L}_i$.
    \item There is a section $C_1\subseteq \PP(\mc{E})$ of $\pi: \PP(\mc{E}) \to C$ on which $s$ vanishes.
    \item The form $\Lambda(s_{\eta})$ associated to the pullback of $s$ to the generic point $\eta \in C$ is reducible as a binary cubic form over a field.\footnote{    
        That is, let $\eta = \Spec \F_q(C)$ be the generic point of $C$. The pullback of $s$ along $\eta \to C$ is a section $s_{\eta}$ of a line bundle isomorphic to $\mc{O}(3)$ on $\PP(\mc{E})_{\eta}\cong \PP^1_{\eta}$. The associated binary cubic form can be written as $a_0 x^3 + a_1 x^2y + a_2 xy^2 + a_3 y^3$, where $a_i \in \F_q(C)$ are elements of a field. Then $s$ is horizontally reducible if and only if this binary cubic form is $0$ or contains a linear factor.} 

    \item The form $\Lambda(s) \in \Sym^3 \mc{E} \otimes \mc{L}$ is reducible after base change to the generic point of $C$.

    \end{itemize}

    Furthermore, the following is equivalent to $s$ being horizontally irreducible:

    \begin{itemize}
        \item $s\neq 0$ and the finite map $(X_0)_{\eta} \to \eta$ is a degree $3$ extension of fields.
    \end{itemize}

\end{prop}
\begin{proof}
    See \Cref{sec:appendix-hor-finite}.
\end{proof}
In particular, horizontal irreducibility of a section $s$ is a property of the restriction of $s$ to generic fiber of $C$.

\begin{rem}
    Note that for our definition of horizontal irreducibility, if a section $s$ cuts out a scheme $X_0\subseteq \PP(\mc{E})$ which is irreducible but nonreduced, then $s$ is considered to be horizontally reducible.
\end{rem}

Now we define notation for horizontally irreducible sections.
\begin{defn}
    Let $C$ be a nice curve. Let $\mc{E}$ be a rank $2$ vector bundle on $C$, and let $\mc{L}$ be a line bundle on $C$. 

    We denote by
    \[
    \mc{V}(\mc{E}, \mc{L})^{\ir}\subseteq \mc{V}(\mc{E}, \mc{L})
    \]
    and
    \[
    \mc{V}(\mc{E})^{\ir} \subseteq \mc{V}(\mc{E})
    \]
    the subsets of horizontally irreducible sections.

    More generally, for set $S$ which is a subset of $V(\mc{E}, \mc{L})$ or $V(\mc{E})$, we write $S^{\ir}\subseteq S$ for the subset of horizontally irreducible sections in $S$.
\end{defn}

\begin{defn}
    Let $C$ be a nice curve, and let $\mc{E}$ be a rank $2$ vector bundle on $C$.
    
    We say a section $s \in \mc{V}(\mc{E})$ is \textbf{smooth} if $s\neq 0$ and the vanishing scheme $X_0 \subseteq \PP(\mc{E})$ is smooth. We denote by
    \[
    \mc{V}(\mc{E})^{\sm, \ir} \subseteq \mc{V}(\mc{E})
    \]
    the subset of sections which are smooth and \emph{horizontally} irreducible.
\end{defn}

We would like to say that $s \in \mc{V}(\mc{E})^{\sm, \ir}$ if and only if $V(s)$ is smooth and \emph{irreducible}. So when $s$ is smooth, there is no distinction between horizontal irreducibility and reducibility.

\begin{lemma}\label{lem:smooth-prim}
    Let $C$ be a nice curve and let $\mc{E}$ be a rank $2$ vector bundle on $C$. If $s \in \mc{V}(\mc{E})^{\sm, \ir}$ is a horizontally irreducible smooth section, then $s$ is primitive, and therefore $V(s)$ is irreducible.
\end{lemma}
\begin{proof}
    Assume for sake of contradiction that $s$ is smooth but not primitive. Then $s$ vanishes at the fiber over a point $P \in C$. In particular, $s$ lies in the image of the map
    \[
    H^0(\PP(\mc{E}), \mc{M}_{\mc{E}} \otimes \pi^{*}\mc{O}_C(-P))\to H^0(\PP(\mc{E}), \mc{M}_{\mc{E}}).
    \]
    Now $\mc{M}_{\mc{E}}\otimes \pi^{*} \mc{O}_{C}(-P)$ restricts to $\mc{O}(-3)$ the fiber $\PP(\mc{E})_{P} \cong \PP^1_{P}$, so in particular the preimage of $s$ must vanish at some closed point $\mc{P}$ on the fiber. Therefore $s$ vanishes to order $2$ at $\mc{P}$, so $X_0$ is not regular at $\mc{P}$, and therefore $X_0$ is not smooth, a contradiction.

\end{proof}

In particular, we have $s \in \mc{V}(\mc{E})^{\sm, \ir}$ if and only if $(\mc{E}, s) \in \TriSec^{\sm, \ir}(C)$ 
\subsection{Counting Covers to Counting Sections}

Now we would like to rewrite 
\[
\Theta(C,N) = \sum_{\substack{f: X \to C\\ \deg \Delta_f = 2N\\ X\text{ smooth, irreducible}}} \frac{1}{|\Aut f|}
\]
using our correspondence from \Cref{sec:param-cubic}. Here the sum ranges over isomorphism classes of triple covers $f:X \to C$ such that $X$ is smooth and irreducible.

Any such triple cover $f$ corresponds to a unique (up to isomorphism) pair $(\mc{E}, s)$, where $\mc{E}$ is the Tschirnhausen bundle of $f$ and $s \in \mc{V}(\mc{E})^{\ir, \sm}$ is a horizontally irreducible smooth section. 

Furthermore, because we have an equivalence of groupoids, we have an isomorphism between $\Aut(f)$ and the stabilizer $\Stab(s)$ of $s$ by the action of $\Aut(\mc{E})$ on $\mc{V}(\mc{E})$.

By the orbit-stabilizer theorem, we can write
\[
\sum_{[s] \in \mc{V}(\mc{E})^{\ir, \sm}} \frac{1}{|\Stab(s)|} = \frac{1}{|\Aut \mc{E}|}\sum_{[s] \in \mc{V}(\mc{E})^{\ir, \sm}} |\Orb(s)| = \frac{|\mc{V}(\mc{E})^{\sm,\ir}|}{|\Aut \mc{E}|}
\]
where $[s]$ denotes an equivalence class of elements of $ \mc{V}(\mc{E})^{\ir, \sm}$ under the action of $\Aut(\mc{E})$.
(Alternatively, both sides give the cardinality of the quotient groupoid $\mc{V}(\mc{E})^{\ir, \sm}// \Aut(\mc{E})$.)

We will use the following notation for the groupoid of rank $2$ vector bundles on $C$.
\begin{defn}
    Let $C$ be a nice curve. We let $\Bun(C)$ denote the category of rank $2$ vector bundles on $C$ where all morphisms are isomorphisms.
\end{defn}

\begin{defn}
    Let $\mc{C}$ be a groupoid which is an essentially small category. We let $|\mc{C}|$ denote the set of isomorphism classes of objects in $\mc{C}$, and we write $\sum_{x \in |\mc{C}|}$ to denote a sum over all isomorphism classes of objects in $\mc{C}$.
\end{defn}

Now we can write
\[
\sum_{\substack{f: X \to C \\ \deg \Delta_f = 2N\\ X\text{ smooth, irreducible}}} \frac{1}{|\Aut f|} = \sum_{\substack{\mc{E} \in |\Bun(C)| \\ \deg \mc{E} = N} }\sum_{[s] \in \mc{V}(\mc{E})^{\ir, \sm}} \frac{1}{|\Stab(s)|} = \sum_{ \substack{\mc{E} \in |\Bun(C)| \\\deg \mc{E} = N} }\frac{1}{|\Aut \mc{E}|} |\mc{V}(\mc{E})^{\ir, \sm}|
\]
and therefore
\[
\Theta(C, N) = \sum_{\substack{\mc{E} \in |\Bun(C)| \\ \deg \mc{E} = N} }\frac{1}{|\Aut \mc{E}|} |\mc{V}(\mc{E})^{\sm,\ir}|.
\]

\subsection{Sieving for Smooth Sections}

For a fixed rank $2$ vector bundle $\mc{E}$, the set $\mc{V}(\mc{E})$ is a finite dimensional $\F_q$-vector space. We will sieve out the smooth sections in this set by an inclusion-exclusion process.

For a closed point $P \in C$, we would like to define the subset of sections $s \in \mc{V}(\mc{E})$ which are ``bad'' above $P$, meaning they fail to be smooth at some point on the fiber $\pi^{-1}(P)\subseteq \PP(\mc{E})$ of $P$.

Then a nonzero section $s \in \mc{V}(\mc{E})$ is smooth if and only if it is not bad above any closed point $P \in C$. By keeping track of sections which are ``bad'' above every closed point in a reduced effective divisor $D\subseteq C$, we can exactly count smooth sections by an alternating sum.

More precisely, we define the following notation.

\begin{defn}
    Let $C$ be a nice curve and let $\mc{E}$ be a rank $2$ vector bundle on $C$. Let $P$ be a closed point in $C$.

    Let $s \in \mc{V}(\mc{E})$ be a nonzero section, and let $X_0:= V(s) \subseteq \PP(\mc{E})$ be the vanishing scheme of $s$. If there exists a closed point $\mc{P}\in X_0$ which lies in the fiber above $P$ such that $X_0$ is not regular at $P$, then we say $s$ is \textbf{bad} above $P$. (If $s=0$, then $s$ is bad above $P$ as well.)

    Let $D$ be a reduced effective divisor on $C$. We let 
    \[
    \Bad(\mc{E}, D)\subseteq \mc{V}(\mc{E})
    \]
    denote the set of sections that are bad above $P$ for every $P \in D$.
\end{defn}

\begin{rem}
    Since every closed point of $X_0$ lies in the fiber of a point of $C$, we conclude that $X_0$ is regular of dimension $1$ at every point if and only if it is not bad above any fiber. Furthermore, $X_0$ is regular if and only if it is regular at every closed point (because $X_0$ is locally Noetherian), and $X_0$ is smooth if and only if it is regular because the base field $\F_q$ is perfect.
\end{rem}

By the above discussion, we have
\begin{equation}\label{eq:alt-bad}
\Theta(C, N) = \sum_{\substack{\mc{E} \in |\Bun(C)| \\\deg \mc{E}  = N}}\frac{1}{|\Aut \mc{E}|} \sum_{D} \mu(D) |\Bad(\mc{E}, D)^{\ir}|,
\end{equation}
where $\Bad(\mc{E}, D)^{\ir}$ denotes the subset of horizontally irreducible sections. Here $\mu(D) = (-1)^{\omega(D)}$ is the M\"obius function, where $\omega(D)$ is the number of closed points in $D$.

\begin{rem}
    We remark that the above sum is finite. This is because we have:
    \begin{enumerate}
        \item[(i)] For a fixed integer $N$, there are finitely many rank $2$ degree $N$ vector bundles $\mc{E}$ such that $\mc{V}(\mc{E})^{\ir}$ is nonempty.
        \item[(ii)] For a fixed rank $2$ vector bundle $\mc{E}$, if $\deg D$ is sufficiently large, then $\Bad(\mc{E}, D)^{\ir}$ is empty.
    \end{enumerate}

    For a proof, see \Cref{sec:appendix-hor-finite}.
\end{rem}

\subsection{Refining the Sieve}
In this section, we analyze in further detail the ways in which a section $s$ can be bad above a point $P \in C$, and we use this to write a more refined decomposition of $\Theta(C,N)$. This is essentially the same analysis as that in Zhao \cite{Zha13}, Chapter 5 (especially Section 5.4).

We begin with some generalities about vanishing and singularity of a section of a line bundle at a point.

\begin{defn}
    Let $X$ be a regular scheme, and let $x \in X$ be a closed point. Let $s \in H^0(X, \mc{M})$ be a global section of a line bundle $\mc{M}$ on $X$. Let $i\ge 0$ be an integer. We say that the \textbf{order of vanishing} of $s$ at $x$ is \textbf{at least $i$} if the restriction of $s$ to the local ring $\mc{O}_{X,x}$ lies in $\mf{m}_{x}^{i}$.
\end{defn}
\begin{defn}
    Let $\mc{E}$ be a rank $2$ vector bundle on a nice curve $C$. Let $\mc{D} \subseteq \PP(\mc{E})$ be a reduced finite subscheme that is the union of finitely many closed points $\mc{P} \in \PP(\mc{E})$. 
    
    \begin{itemize}
        \item We write
    \[
    \Van(\mc{E}, \mc{D}) \subseteq \mc{V}(\mc{E})
    \]
    for the subspace of points vanishing to order at least $1$ at each point $\mc{P}\in \mc{D}$. We say $s \in \Van(\mc{E}, \DD)$ \textbf{vanishes at} $\mc{D}$.
    \item We write
    \[
    \Sing(\mc{E}, \mc{D}) \subseteq \mc{V}(\mc{E})
    \]
    for the subspace of points vanishing to order $2$ at each point $\mc{P} \in \mc{D}$. We say $s \in \Sing(\mc{E}, \mc{D})$ is \textbf{singular at} $\mc{D}$.
    \item Let $D\subseteq C$ be a reduced effective divisor, thought of as a finite union of closed points $P$. We write 
    \[
    \Fib(\mc{E},D) \subseteq \mc{V}(\mc{E})
    \]
    for the subspace of sections which vanish on each fiber $\PP(\mc{E})_{P} \cong \PP^{1}_{P}$ for $P \in D$. We say $s \in \Fib(\mc{E}, D)$ is \textbf{fibral at} $D$.
    \item Assume now that the image of $\mc{D}\subseteq \PP(\mc{E})$ under $\pi: \PP(\mc{E})\to C$ is $D\subseteq C$, and furthermore assume that the map $\pi|_{\mc{D}}: \mc{D} \xrightarrow{\sim} D$ is an isomorphism. In this case, we write
    \[
    \SingFib(\mc{E}, \mc{D}):= \Sing(\mc{E}, \mc{D}) \cap \Fib(\mc{E}, D) \subseteq \mc{V}(\mc{E})
    \]
    for the subspace of sections which are fibral at $D$ and singular at $\mc{D}$.
    \end{itemize}
\end{defn}
As a remark, if $\mc{P} \in \PP(\mc{E})$ is a closed point, then the local ring at $\mc{P}$ in $\PP(\mc{E})$ is a regular local ring of dimension $2$. If $s \in \mc{V}(\mc{E})$ is a nonzero section vanishing at $\mc{P} \in \PP(\mc{E})$, then the local ring of $\mc{P}$ in $X_0 = V(s)\subseteq \PP(\mc{E})$ is a regular local ring of dimension $1$ if and only if $s \not \in \mf{m}_{\mc{P}}^2$. So when $s$ is nonzero, $X_0$ is non-regular at $\mc{P}$ if and only if $s \in \mf{m}_{\mc{P}}^2$.

By our earlier definition, we have that $s \in \Bad(\mc{E}, P)$ if and only if there exists $\mc{P}$ in the fiber of $P$ such that $s \in \Sing(\mc{E}, \mc{P})$. We note that we only need to test singularity at points $\mc{P}$ of relative degree $1$.

\begin{lemma}
    If $s \in \Bad(\mc{E}, P)$, then either $s \in \Fib(\mc{E}, P)$ is fibral at $P$, or $s$ is singular at a closed point $\mc{P}$ of relative degree $1$ in the fiber of $P$.
\end{lemma}
\begin{proof}
    Assume for sake of contradiction that $s$ is not fibral at $P$, but $s$ is singular at a closed point $\mc{P}$ lying above $P$ of relative degree at least $2$. Restricting to the fiber $\PP^1_{P}$, we get in particular that the restriction $s|_{P} \in H^0(\PP^1_P, \mc{O}(3))$ vanishes to order at least $2$ on a point of relative degree $2$. This is impossible if $s|_P$ is nonzero since $2\cdot 2>3$.
\end{proof}

The proof of the lemma also shows that a non-fibral section cannot be singular at multiple degree $1$ points. Therefore we can write
\[
\Bad(\mc{E}, P) = \Fib(\mc{E}, P) \cup \bigcup_{\mc{P} \in \PP^1_{P}} \Sing(\mc{E}, \mc{P})
\]
where $\mc{P}$ ranges over relative degree $1$ points in the fiber $\PP^1_{P}$. (Recall that if $s$ is fibral at $P$, then it is singular at some closed point $\mc{P}$ in the fiber above $P$, by the proof of \Cref{lem:smooth-prim}.)

We can write this as a disjoint union
\begin{equation}\label{eq:bad-disj-union}
\Bad(\mc{E}, P) = \Fib(\mc{E},P) \sqcup \bigsqcup_{\mc{P} \in \PP^1_P} \left(\Sing(\mc{E}, \mc{P}) \setminus \SingFib(\mc{E}, \mc{P})\right)
\end{equation}
since a non-fibral section cannot be singular at two distinct points.

In particular, we have
\[
|\Bad(\mc{E}, P)^{\ir}| = \sum_{\mc{P} \in \PP^1_{P}} |\Sing(\mc{E}, \mc{P})^{\ir}| - \sum_{\mc{P} \in \PP^1_{P}}|\SingFib(\mc{E}, \mc{P})^{\ir}| + |\Fib(\mc{E}, P)^{\ir}|
\]
where $\mc{P}$ ranges over relative degree $1$ points in the fiber $\PP^1_{P}$.

In order to write down a more general formula for $|\Bad(\mc{E}, D)^{\ir}|$, we'll first introduce notation for a certain kind of $0$-dimensional scheme on a ruled surface.

\begin{defn}
    Let $C$ be a nice curve and let $D$ be a reduced effective divisor. Let $\mc{E}$ be a rank $2$ vector bundle on $C$, and let $\pi: \PP(\mc{E}) \to C$ be the associated projective bundle.
    
    A \textbf{$D$-marking} of $\mc{E}$ is a $0$-dimensional closed reduced subscheme $\mc{D}\subseteq \PP(\mc{E})$ such that $\mc{D} \subseteq \pi^{-1}(D)$ and the map $\pi|_{\DD}: \mc{D} \to D$ is an isomorphism. Let $\Mar(\mc{E}, D)$ denote the set of $D$-markings of $\mc{E}$.\footnote{In particular, for a closed point $P \in C$ we can identify $\Mar(\mc{E}, P)$ with the set of relative degree $1$ points $\mc{P} \in \PP^1_P$. Therefore $|\Mar(\mc{E}, P)| = q^{\deg P} +1$, and generally
    \[
    |\Mar(\mc{E}, D)| = \prod_{P \in D} \left(q^{\deg P} + 1\right).
    \]
    }
\end{defn}

One can see that a $D$-marking of $\mc{E}$ is equivalent to the choice of a relative degree $1$ point $\mc{P}$ in the fiber $\PP(\mc{E}) \cong \PP^1_P$ for each closed point $P \in D$.

With this notation, following Zhao \cite{Zha13} (Section 5.4), we can write
\begin{prop}[Zhao \cite{Zha13}]\label{prop:zhao-identity}
Let $C$ be a nice curve. Let $\mc{E}$ be a rank $2$ vector bundle on $C$, and let $D$ be a reduced effective divisor on $C$. Then the number of horizontally irreducible sections which are bad above every point $P \in D$ is
\[
|\Bad(\mc{E}, D)^{\ir}| = \sum_{\substack{D_1, D_2, D_3\\D_1+D_2+D_3 = D}} \mu(D_2) \sum_{\substack{\DD_1 \in \Mar(\mc{E}, D_1)\\  \DD_2 \in \Mar(\mc{E}, D_2)}} \left| \Sing(\mc{E}, \DD_1) \cap \SingFib(\mc{E}, \DD_2) \cap \Fib(\mc{E}, D_3)^{\ir}\right|.
\]
Here the first sum ranges over disjoint reduced effective divisors $D_1,D_2,D_3$ on $C$ such that $D_1+D_2+D_3 = D$. 
\end{prop}
\begin{proof}
    For a fixed section $s \in \mc{V}(\mc{E})$, \eqref{eq:bad-disj-union} implies that we have the identity of indicator functions
    \[
    1_{s \in \Bad(\mc{E}, P)} = 1_{s \in \Fib(\mc{E}, P)} + \sum_{\mc{P} \in \PP^1_P}1_{s \in \Sing(\mc{E}, \mc{P})} - \sum_{\mc{P} \in \PP^1_P} 1_{s \in \SingFib(\mc{E}, \mc{P})}.
    \]

    Furthermore, by definition we have multiplicativity
    \[
    1_{s \in \Bad(\mc{E}, D)} = \prod_{P \in D} 1_{s \in \Bad(\mc{E}, P)}.
    \]
    So by expanding the product we find that
    \[
    1_{s \in \Bad(\mc{E}, D)}=
    \sum_{\substack{D_1, D_2, D_3\\D_1+D_2+D_3 = D}} \mu(D_2) \left(\prod_{P \mid D_1} \sum_{\mc{P} \in \PP^1_{P}} 1_{s \in \Sing(\mc{E}, \mc{P})} \right)\left(\prod_{P_2 \mid D_2} \sum_{\mc{P} \in \PP^1_P} 1_{s \in \SingFib(\mc{E}, \mc{P})}\right) 1_{s \in \Fib(\mc{E}, D_3)}
    \]
    and so
    \[
    1_{s \in \Bad(\mc{E}, D)}=
     \sum_{\substack{D_1, D_2, D_3\\D_1+D_2+D_3 = D}} \mu(D_2)\left(\sum_{\DD_1 \in \Mar(\mc{E}, D_1)} 1_{s \in \Sing(\mc{E}, \DD_1)} \right)\left(\sum_{\DD_2 \in \Mar(\mc{E}, D_2)} 1_{s \in \SingFib(\mc{E}, D_2)}\right) 1_{s \in \Fib(\mc{E}, D_3)}.
    \]
    Summing over all $s \in \mc{V}(\mc{E})^{\ir}$ yields the desired result.
\end{proof}

\begin{rem}
    Although Zhao only considered the case $C=\PP^1$, the statement and proof is essentially the same for a nice curve $C$, so we attribute the proposition to him.
\end{rem}

With this notation, we can put together the pieces from above and conclude the following proposition.
\begin{prop}\label{prop:version-05}
Let $C$ be a nice curve, and let $N\ge 0$ be an integer. Then we have
\[
\Theta(C, N) =\sum_{\substack{\mc{E} \in |\Bun(C)| \\\deg \mc{E} = N} }\frac{1}{|\Aut \mc{E}|}
\sum_{D_1, D_2, D_3} \mu(D_1)\mu(D_3) 
\sum_{\substack{\DD_1 \in \Mar(\mc{E}, D_1) \\ \DD_2 \in \Mar(\mc{E}, D_2)}} |\Sing(\mc{E}, \DD_1)\cap \SingFib(\mc{E}, \DD_2) \cap \Fib(\mc{E}, D_3)^{\ir}|
\]
where $D_1, D_2, D_3$ range over all triples of disjoint reduced effective divisors on $C$.
\end{prop}
\begin{proof}
    We substitute \Cref{prop:zhao-identity} into \eqref{eq:alt-bad}.
\end{proof}

To make progress from here, we will have to construct a correspondence which reduces singularity on one ruled surface to vanishing on another ruled surface. This is the goal of the next section.

\section{Ruled Surfaces and Elementary Transformations}\label{sec:ruled-elm}
Let $C$ be a smooth, projective, geometrically irreducible curve over $\F_q$. In order to count smooth, irreducible triple covers of $C$, we need to count smooth, horizontally irreducible sections in each 
\[
\mc{V}(\mc{E}) = H^0(\PP(\mc{E}), \mc{M}_\mc{E})
\]
where $\mc{E}$ ranges over rank $2$ vector bundles on $C$ and $\mc{M}_\mc{E}:= \mc{O}_{\PP(\mc{E})}(3) \otimes \pi^{*} (\wedge^2 \mc{E})^{\vee}$ is a specific line bundle on $\PP(\mc{E})$. In particular, we'd like to count sections which are singular at a specific collection of points in $\PP(\mc{E})$ (as in \Cref{sec:sieve-i}).

In order to do so, it will be beneficial to construct a certain birational transformation of $\PP(\E)$. The transformation will be exhibited by the composition of a blowup at a collection of points in $\PP(\E)$ and a blowdown onto a new ruled surface. Singular sections $s \in V(\E)$ will correspond to certain sections vanishing at a collection of points on the new ruled surface. This strategy and construction is essentially the same as that of Zhao \cite{Zha13} (Section 5.1).

\subsection{Motivation}
As a starting point, for motivation let $C$ be a smooth projective irreducible curve over an algebraically closed field $k$. We recall a construction of a certain birational transformation of ruled surfaces over $C$. For more details, see for example see \cite[Section V.5]{Har77}.

Let $\pi: X\to C$ be a ruled surface, meaning all fibers are isomorphic to $\PP^1: = \PP^1_k$ and $\pi$ admits a section. Let $\mc{P} \in X$ be a point with $\pi(\mc{P}) = P \in C$, and let $F = \pi^{-1}(P)\cong \PP^1$ be the fiber passing through $\mc{P}$. Let $f: Y\to X$ be the blowup of $X$ at $\mc{P}$. Now the strict transform $F^{st}$ and the exceptional divisor $E$ are both isomorphic to $\PP^1$. Furthermore $F^{st} \sim f^{*}F - E$, so 
\[
F^{st}\cdot F^{st} = (f^{*}F - E)\cdot (f^{*}F - E) = (f^{*} F) \cdot (f^{*}F) - 2 (f^{*} F\cdot E) + E\cdot E = -1.
\]
Therefore $F^{st}$ is a $(-1)$-curve, so by Castelnuovo's Contraction Theorem we can blow it down to obtain another smooth surface $X'$. Let $g: Y \to X'$ be this contraction, which is a blowup centered at some point $\mc{P}' \in X'$ and with exceptional divisor $F^{st}$. We have $g(E) \cong \PP^1$.

There is a map $X' - \mc{P}' \to C$ by the isomorphism $Y - F^{st} \cong X'-\mc{P}'$, which must extend to a morphism $\pi': X'\to C$ (since $\mc{P}'$ has codimension $2$). Over $U:= C-P$, the complement of $P$, we get isomorphisms $X_{U} \cong Y_{U} \cong X'_{U}$ of $U$-schemes. At $P$, the fiber of $\pi'$ is $g(E) \cong \PP^1$. In particular, all fibers of $\pi': X' \to C$ are isomorphic to $\PP^1$. By taking the pushforward of the strict transform of a section of $\pi$, we obtain a section of $\pi'$ as well. 

In summary, given a ruled surface $\pi: X\to C$ over an algebraicaly closed field $k$ and a point $\mc{P} \in X$, we can construct a new ruled surface $\pi': X' \to C$. The surface $X'$ is birationally equivalent to $X$ as exhibited by a blowup-blowdown correspondence $X\leftarrow Y \rightarrow X'$ of schemes over $C$. This birational transformation of $X$ is sometimes known as the \textbf{elementary transform at $\mc{P}$}.

Now consider a smooth, projective, geometrically irreducible curve $C/\Fq$. The main goal of the rest of the section is to construct an analogue of an elementary transformation for a ruled surface $\PP_1\to C$. We will construct a blowup-blowdown correspondence $\PP_1  = \PP(\E_1)\leftarrow Y \rightarrow \PP_2$ with a different ruled surface $\PP_2 \to C$.

The essence of the construction is the same; however, we will need a few specific properties.
\begin{itemize}
    \item Our surfaces are over $\Fq$, not an algebraically closed field, so it would be less direct to use Castelnuovo's Contraction Theorem to construct a blowdown.
    \item We blow up and down at many points at once. In particular our blowup center will be a $0$-dimensional scheme $\mc{D}_1 \subseteq \PP_1$.
    \item Given the data of a pair $(\E_1, \mc{D}_1)$ of a rank $2$ vector bundle $\E_1$ such that $\PP_1 \cong \PP(\E_1)$ and our blowup center $\mc{D}_1 \subseteq \PP(\E_1)$, we explicitly keep track of the resulting pair $(\E_2, \mc{D}_2)$ where $\PP_2 \cong \PP(\E_2)$ and $\mc{D}_2 \subseteq \PP(\E_2)$ is the new blowup center.
\end{itemize}
These properties (especially the third) will be necessary for our future counting arguments.
\begin{rem}
    Zhao's construction is similar, especially regarding the first two points. However, we keep track of the data associated to $D$-marking $\DD_1 \subseteq \PP(\mc{E}_1)$ in a different manner than he does. Zhao (in the case $C=\PP^1$) measures a nonnegative integer $h(\DD_1)$, thought of as the height of $\DD_1$, and compares it to the height $h(\DD_2)$.

    In contrast, we understand the construction as a transformation of sheaves on $C$, which allows us to prove \Cref{prop:elm-groupoid-equiv}.
\end{rem}

\subsection{$D$-Markings of Bundles and Elementary Modifications}
\label{sec:hecke-mod}
In fact, we will define our construction in a slightly more general situation. We consider a Dedekind scheme $S$, meaning an integral regular separated (Noetherian) scheme of dimension $1$. A nice curve is an example of a Dedekind scheme, as well as $\Spec \mc{O}_K$ where $\mc{O}_K$ the ring of integers in a number field.

We will also consider a reduced effective divisor $D$ on $S$, thought of as a closed immersion $i: D \hookrightarrow C$ of a dimension $0$ reduced closed subscheme. We can write $D =\cup P$ for some finite disjoint collection of closed points $P \in C$.

Earlier, we defined a certain kind of $0$-dimensional subscheme $\mc{D} \subseteq \PP(\mc{E})$ called a $D$-marking. We constructed it for a nice curve $C$, but the same definition will work for a Dedekind scheme $S$.

\begin{defn}
    Let $S$ be a Dedekind scheme and let $D$ be a reduced effective divisor. Let $\E$ be a rank $2$ vector bundle on $S$, and let $\pi: \PP(\E) \to S$ be the associated projective bundle.
    
    A \textbf{$D$-marking} of $\E$ is a dimension $0$ closed reduced subscheme $\DD\subseteq \PP(\E)$ such that $\DD\subseteq \pi^{-1}(D)$ and the map $\pi|_{\DD}: \DD \to D$ is an isomorphism. Let $\Mar(\E,D)$ denote the set of $D$-markings of $\E$.
\end{defn}

We can alternately understand the data of a $D$-marking of rank $2$ vector bundle $\E$ in several equivalent ways. If $\mc{D}\subseteq \PP(\E)$ is a $D$-marking, then $\mc{D}$ can also be thought of as the image of a section of the map $\pi|_{D}: \PP^1_{D} \to D$, or equivalently, a map $j: D \to \PP(\E)$ such that $j\circ \pi = i$. By the universal property of the projective bundle, such a map $j$ is equivalent to the data of a surjection $i^{*}\E \twoheadrightarrow \LL$ for some line bundle $\LL$ on $D$.

We formalize this correspondence in the next proposition.
\begin{prop}
\label{prop:dmar-interp}
    Let $S$ be a Dedekind scheme, and let $i: D \hookrightarrow S$ be a reduced effective divisor. Let $\E$ be a rank $2$ vector bundle on $C$, and let $\pi: \PP(\E) \to S$ be the associated projective bundle.

    There is a (canonical) bijective correspondence between $\Mar(\E, D)$ and each of the following sets:
\begin{itemize}

    \item The set given by a choice of a relative degree 1 closed point 
    \[
    \mc{P} \in \pi^{-1}(P) \cong \PP^1_{P}.
    \]
    for each closed point $P \in D$.
    \item The set of maps $j: D\to \PP(\E)$ such that $j\circ \pi = i$.
    \item The set of pairs $(\LL, \psi)$ up to equivalence, where $\LL$ is a line bundle on $D$ and 
    \[
    \psi: i^{*}\E\twoheadrightarrow \LL
    \]
    is a surjection of sheaves on $D$. Two pairs $(\LL_1, \psi_1)$ and $(\LL_2, \psi_2)$ are equivalent if there is an isomorphism $a: \LL_1\to \LL_2$ such that $a\circ \psi_1 = \psi_2$.
    \item The set of pairs $(\LL, \phi)$ up to equivalence, where $\LL$ is a line bundle on $D$ and 
    \[
    \phi: \E \twoheadrightarrow i_{*}\LL
    \]
    is a surjection of sheaves on $S$. Two pairs $(\LL_1, \phi_1)$ and $(\LL_2, \phi_2)$ are equivalent if there is an isomorphism $a: \LL_1 \to \LL_2$ such that $i_{*}a\circ \phi_1 = \phi_2$.
\end{itemize}
\end{prop}
\begin{proof}
    The correspondence between $\Mar(\E,D)$ and the first two sets is clear.

    By the universal property of the projective bundle, the second set is equivalent to the third set.

    To show the equivalence of the third and fourth sets, we need to show that the bijection (by adjunction) $\Hom_{D}(i^{*}\E, \LL)\cong \Hom_{S}(\E, i_{*}\LL)$ preserves surjectivity.

    Indeed, any surjection $i^{*}\E\twoheadrightarrow \LL$ is associated by adjunction to the composite morphism
    \[
    \E \twoheadrightarrow i_{*} i^{*} \E \twoheadrightarrow i_{*} \LL
    \]
    which is a surjection because $i$ is a closed immersion. Conversely, a surjection $\E \twoheadrightarrow i_{*}\LL$ corresponds to the map
    \[
    i^{*}\E \twoheadrightarrow i^{*}i_{*} \LL = \LL
    \]
    which is a surjection. 

    Therefore, the third and fourth sets are in bijection.
\end{proof}

It will be useful for us to also define the category of rank $2$ vector bundles with a given $D$-marking. We only consider isomorphisms, so our category will be a groupoid.

\begin{defn}
    Let $S$ be a Dedekind scheme and let $i: D\hookrightarrow S$ be a reduced effective divisor. A \textbf{$D$-marked bundle} is a pair $(\E, \mc{D})$, where $\E$ is a rank $2$ vector bundle on $S$ and $\mc{D} \in \Mar(\mc{E}, D)$. 
    
    We define $\DBun$ to be the groupoid of $D$-marked bundles $(\E, \mc{D})$. An isomorphism of $D$-marked bundles $(\E_1, \mc{D}_1)$ and $(\E_2, \mc{D}_2)$ is defined to be an isomorphism $\E_1\xrightarrow{\sim} \E_2$ which sends $\mc{D}_1$ to $\mc{D}_2$.
\end{defn}
We can equivalently view the groupoid $\DBun$ of $D$-marked bundles in terms of sheaves on $S$.
\begin{defn}
 Let $\DBun_{1}$ denote the groupoid with objects given by triples $(\E, \LL, \phi)$, where $\E$ is a rank $2$ vector bundle on $S$, $\LL$ is a line bundle on $D$, and
    \[
    \phi: \E \twoheadrightarrow i_{*} \LL 
    \]
    is a surjection of sheaves. An isomorphism of objects is a pair of isomorphisms $\E_1\xrightarrow{\sim} \E_2$ and $\LL_1\xrightarrow{\sim} \LL_2$ such that the induced diagram commutes: 
    \[
\begin{tikzcd}
\E_1 \arrow[r, "\phi_1"] \arrow[d]
& i_{*}\LL_1\arrow[d] \\
\E_2 \arrow[r, "\phi_2"]
& i_{*}\LL_2
\end{tikzcd}
\]

\end{defn}
\begin{prop}
    We have an isomorphism of groupoids $\DBun \cong \DBun_1$.
\end{prop}
\begin{proof}
    It is clear that $\DBun\cong \DBun_{1}$ by our fourth interpretation of $\DMar(\E)$ from \Cref{prop:dmar-interp}.
    
\end{proof}

Now consider a given $D$-marked bundle $(\E, \mc{D})$. We would like to define a certain transformation of $(\mc{E}, \mc{D})$.

By \Cref{prop:dmar-interp}, it has associated surjective morphisms $\phi: \E \twoheadrightarrow i_{*}\LL$ and $\psi: i^{*} \E \twoheadrightarrow \LL$. Define the coherent sheaf $\E':= \ker(\phi)$. It is a subsheaf of a locally free sheaf, so it is torsion-free, and therefore locally free. (This is because $S$ is a Dedekind scheme --- see for example \cite[\href{https://stacks.math.columbia.edu/tag/0CC4}{Lemma 0CC4}]{stacks-project}.) Furthermore $\phi: \E' \to \E$ is an isomorphism away from $D$, so $\E'$ has rank $2$.

Similarly, $\LL':= \ker(\psi)$ is a torsion free coherent sheaf on $D$, so it is locally free of rank $1$.

\begin{defn}
    We define the \textbf{elementary transform functor}
    \[
    \ElmD: \DBun\to \DBun
    \] as follows. Let $(\E, \mc{D})$ be a $D$-marked bundle with associated surjective morphisms $\phi: \E\twoheadrightarrow i_{*} \LL$ and $\psi: i^{*} \E \twoheadrightarrow \LL$.

    Let $\E':= \ker(\phi)$ and let $\LL':= \ker(\psi)$. Pulling back the exact sequence 
    \[
    0\to \E' \to \E \to i_{*} \LL\to 0 
    \] yields the right exact sequence
    \[
    i^{*} \E' \to i^{*} \E \to \LL\to 0
    \]
    and therefore a unique surjection $\psi': i^{*} \E' \twoheadrightarrow \LL'$. We define
    \[
    \ElmD(i^{*} \E\xrightarrowdbl{\psi} \LL) = (i^{*}\E' \xrightarrowdbl{\psi'} \LL').
    \]
\end{defn}

We would like to show that $\ElmD$ induces a predictable twist on $\wedge^2 \mc{E}$, and furthermore that $\ElmD$ is invertible.

\begin{defn}
    Let $\mc{N}$ be a line bundle on $S$. Let $\DBun_{N}$ denote the full subcategory of $\DBun$ consisting of pairs $(\E, \mc{D})$ such that $\wedge^2 \E \cong \mc{N}$.
\end{defn}

\begin{prop}\label{prop:elm-groupoid-equiv}
    $\ElmD$ induces an equivalence of groupoids 
    \[
    \DBun_{\mc{N}}\xrightarrow{\sim} \DBun_{\mc{N}(-D)}.
    \]
\end{prop}
\begin{proof}
See \Cref{sec:appendix-a}.
\end{proof}

Now consider the case of a rank $2$ vector bundle $\E$ on a nice curve $C$. Let $X = \PP(\E)$, and let $\mc{D} \subseteq X$ be a $D$-marking - equivalently, a choice of a relative degree $1$ point $\mc{P} \in X_P \cong \PP^1_P$ for each closed point $P$ in $D$. We can blow up $X$ at the closed subscheme $\mc{D}$ to get a new scheme $Y\to S$. The exceptional divisor $E$ of the blowup $f: Y\to X$ is the union of the fibers $E_P:= f^{-1}(\mc{P})\cong\PP^1_{\mc{P}} \cong \PP^1_{P}$ at each $\mc{P}$ in $\mc{D}$. 

The map $f: Y\to X$ exhibits an isomorphism $X_P^{st} \xrightarrow{\sim} X_P$. Let $\mc{P}'$ be the preimage of $\mc{P}$ under this isomorphism. We can see that the fiber $Y_P \cong E_P \cup X^{st}_P$, the union of two copies of $\PP^1_{P}$ glued together at $P$.

We will show that we can blow down each $X^{st}_P$ to get a new smooth surface $X'$. It will follow that $X'$ is in fact a ruled surface $X' \to C$, where the fiber at $P$ is the image of $E_P\cong \PP^1_P$ under the blowdown.

Furthermore, we will exactly identify the new ruled surface $X'\to C$ in terms of the image of our elementary transform functor $\ElmD$. That is: if $(\E', \mc{D}') = \ElmD(\E, \mc{D})$, then $X'\to S$ is isomorphic to $\PP(\E')\to C$. Furthermore, the new map $Y\to X'$ is isomorphic to the blowup of $X' \cong \PP(\mc{E}')$ at $\mc{D}'$.

We establish this construction more precisely, and for a general Dedekind scheme $S$, in the following proposition.

\begin{prop}\label{prop:blow-up-down}
    Let $S$ be a Dedekind scheme, and let $D$ be a reduced effective divisor on $S$. Let $\E_1$ be a rank $2$ vector bundle on $S$, with associated projective bundle $\PP(\E_1)\to S$, and let $\mc{D}_1\subseteq \PP(\E_1)$ be a $D$-marking.

    Let $(\E_2, \mc{D}_2)$ be the image of $(\E_1, \mc{D}_1)$ by $\ElmD$. Let $\pi_2: \PP(\E_2)\to S$ be the associated projective bundle.

    Then the blowup of $\PP(\E_1)$ at $\DD_1$ is the blowup of $\PP(\E_2)$ at $\DD_2$.

    More precisely: there exists a scheme $Y$ with $r: Y \to S$ along with maps $f_1: Y \to \PP(\E_1)$, $f_2: Y \to \PP(\E_2)$ such that:
    \begin{itemize}
    \item 
    The following diagram commutes:
\[
\begin{tikzcd}[column sep=small]
& Y \arrow[dl, "f_1"] \arrow[dr, "f_2"]\arrow[dd,"r"] & \\
\PP(\E_1) \arrow[rd,"\pi_1"] & & \PP(\E_2)\arrow[ld, "\pi_2"]\\
&S & 
\end{tikzcd}
\]
\item We have that $f_1: Y\to \PP(\E_1)$ exhibits $Y$ as the blowup of $\PP(\E_1)$ at $\DD_1$, and $f_2: Y \to \PP(\E_2)$ exhibits $Y$ as the blowup of $\PP(\E_2)$ at $\DD_2$.
\item Let $E_1\subseteq Y$ be the exceptional divisor of $f_1$, and let $E_2\subseteq Y$ be the exceptional divisor of $f_2$. Then $E_2 = \pi_1^{-1}(D)^{st}$ and $E_1 = \pi_2^{-1}(D)^{st}$.
\item We have
\[
f_1^{*} \mc{O}_{\PP(\E_1)}(1) \otimes \mc{O}_{Y}(-E_1) \cong f_2^{*}\mc{O}_{\PP(\E_2)}(1).
\]
\end{itemize}

\end{prop}

\begin{proof}
    See \Cref{sec:appendix-a}.
\end{proof}
\subsection{A bijection between singular and vanishing sections}

We'd like to use this birational transformation to find a correspondence between singular sections on one surface and vanishing sections on another surface.

\begin{defn}
    Let $S$ be a Dedekind scheme and let $\E$ be a rank $2$ vector bundle on $S$. Let $X:= \PP(\E)$ with $\pi: X\to S$. Let $D$ be a reduced effective divisor on $S$ and let $\mc{D}\subseteq X$ be a $D$-marking. Let $\LL$ be a line bundle on $C$.

Let
    \[
    \V(\E, \LL):= H^{0}(X, \mc{O}_X(3)\otimes \pi^{*} \LL).
    \]
We define
    \[
    \Van(\E, \LL, \DD)\subseteq \V(\E, \LL)
    \]
to be the sections vanishing at $\DD$ and
    \[
    \Sing(\E, \LL, \DD)\subseteq \V(\E, \LL)
    \]
to be the sections vanishing to order at least $2$ at $\DD$.

\end{defn}

We recall some general facts about blowups. Suppose $\DD\subseteq X$ is a closed subscheme with ideal sheaf $\mc{I}_{\DD}$. Let $Y := \Bl_{\DD} X$ be the blowup of $X$ at $\DD$ with $f: Y \to X$ the structure morphism and $E$ the exceptional divisor. Let $\mc{M}$ be a line bundle on $X$.

Now $f_{*}\mc{O}_{Y}(-E) = \mc{I}_{\DD}^{n}$ for $n\ge 0$. Therefore
\begin{align*}
H^{0}(Y, f^{*} \mc{M} \otimes\mc{O}_{Y}(-nE)) 
&= H^{0}(X, f_{*}\left(f^{*} \mc{M} \otimes O_{Y}(-nE)\right))\\
&= H^{0}(X, \mc{M}\otimes \mc{I}_{\DD}^{n})\subseteq H^{0}(X, \mc{M}).
\end{align*}

This identifies $H^{0}(f^{*}\mc{M}\otimes \mc{O}_{Y}(-nE))$ with the space of sections in $H^{0}(X, \mc{M})$ which vanish to order at least $n$ at $\DD$.

\begin{prop}\label{prop:zhao-lemma}
    Let $S, D, \mc{E}_1, \DD_1,\mc{E}_2, \DD_2$ be as in \Cref{prop:blow-up-down}. Let $\LL$ be a line bundle on $C$. Then there is an isomorphism of abelian groups
    \begin{equation}\label{eq:ab}
    \Sing(\E_1, \LL, \DD_1)\xrightarrow{\sim} \Van(\E_2, \LL(D), \DD_2).
    \end{equation}

    Now assume $S=C$ is a nice curve over $\F_q$. Then \Cref{eq:ab} is an isomorphism of $\F_q$-vector spaces. Furthermore, suppose a nonzero section $s_1 \in \Van(\E_1, \LL, \mc{D}_1)$ corresponds to $s_2 \in \Sing(\E_2, \LL, \mc{D}_2)$. Let $S_{1}\subseteq \PP(\mc{E}_1)$ be the vanishing scheme of $s_1$ and let $S_2\subseteq \PP(\mc{E}_2)$ be the vanishing scheme of $s_2$. Let $U = C-D$. Then $S_1|_{U}$ and $S_2|_{U}$ are isomorphic as $U$-schemes.

    In particular, if $S=C$ is a nice curve, then this isomorphism preserves horizontal irreducibility of sections.
\end{prop}
\begin{proof}
    This is proved by Zhao \cite{Zha13} (Proposition 5.1.0.10) in the case $S= \PP^1$; for completeness we provide a proof here. 

    Let $r: Y\to S$ be the common blowup scheme, as in \Cref{prop:blow-up-down}.  Let $E_1$ be the exceptional divisor of $f_1:Y \to \PP(\mc{E}_1)$, which is the blowup of $\PP(\mc{E}_1)$ at $\DD_1$. Let $E_2$ be the exceptional divisor of $f_2: Y \to \PP(\mc{E}_2)$, which is the blowup of $\PP(\mc{E}_2)$ at $\DD_2$.

    By the last part of \Cref{prop:blow-up-down}, we have
    \begin{equation}\label{eq:O1-isom}
    f_1^{*} \mc{O}_{\PP(\mc{E}_1)}(1) \otimes \mc{O}_{Y}(-E_1) \cong f_2^{*} \mc{O}_{\PP(\mc{E}_2)}(1).
    \end{equation}

    Taking the third power of \eqref{eq:O1-isom} and tensoring with $r^{*} \mc{L}$, we obtain an isomorphism
    \[
    r^{*}\LL\otimes f_1^{*}\mc{O}_{\PP(\mc{E}_1)}(3) \otimes \mc{O}_{Y}(-3E_1)\cong r^{*}\LL\otimes f^{*}_{2}\mc{O}_{\PP(\mc{E}_2)}(3).
    \]
    Note that $r^{*} \mc{O}(-D) = \mc{O}_{Y}(-E_1-E_2)$. So we can write
    \begin{equation}\label{eq:lb-isomo}
    r^{*}\LL \otimes f_1^{*} \mc{O}_{\PP(\mc{E}_1)}(3) \otimes \mc{O}_Y(-2E_1)\cong r^{*}\LL \otimes r^{*} \mc{O}(D) \otimes f_2^{*}\mc{O}_{\PP(\mc{E}_2)}(3)\otimes \mc{O}_{Y}(-E_2).
    \end{equation}
    Taking sections, we recover an isomorphism
    \begin{equation}\label{eq:sec-isomo}
    \Sing(\E_1, \LL, \mc{D}_1) \cong \Van(\E_2, \LL\otimes \mc{O}_{C}(D), \mc{D}_2)
    \end{equation}
    as desired.

    To prove the second half, note that over $U$, our maps $f_1$ and $f_2$ restrict to isomorphisms
    \[
    \PP(\mc{E}_1)_{U} \xleftarrow{\sim} Y_{U} \xrightarrow{\sim}\PP(\mc{E}_2)_{U}.
    \]
    So over $U$, \eqref{eq:lb-isomo} restricts to an isomorphism
    \begin{equation}\label{eq:common-line-bund}
    r^{*} \mc{L}\otimes f_1^{*} \mc{O}_{\PP(\mc{E}_1)}(3)|_{U} \cong r^{*} \mc{L} \otimes f_2^{*} \mc{O}_{\PP(\mc{E}_2)}(3)|_{U}
    \end{equation}
    of line bundles on $Y_{U}$. Therefore, if $S_1 \in \Sing(\mc{E}, \mc{L}, \DD_1)$ and $S_2 \in \Van(\mc{E}_2, \mc{L}\otimes \mc{O}_{C}(D) , \DD_2)$ correspond under \eqref{eq:sec-isomo}, then $S_1|_{U}$ and $S_{2}|_{U}$ are both isomorphic to a common subscheme of $Y_{U}$, which is the vanishing scheme of a section of the line bundle in \eqref{eq:common-line-bund}.
\end{proof}
\begin{rem}
We can describe the bijection in \Cref{prop:zhao-lemma} on the level of divisors, although we won't need this. The varieties $\PP(\mc{E}_1)$, $\PP(\mc{E}_2)$, and $Y$ are regular, integral, separated (Noetherian) schemes, so we identify Weil and Cartier divisors.

If $S_1\subseteq \PP(\mc{E}_1)$ and $S_2\subseteq \PP(\mc{E}_2)$ are effective divisors identified via the bijection (as in the notation in \Cref{prop:zhao-lemma}), we have
\[
S_2 = (f_2)_{*}\left(f_1^{*}S_1-2E_1\right)
\]
and
\[
S_1 = (f_1)_{*}\left( f_2^{*}S_2 - E_2\right).
\]
\end{rem}
\subsection{Comparison for Integers}

Let's consider the construction \Cref{prop:blow-up-down} and bijection \Cref{prop:zhao-lemma} for the example $\Spec \Z$.
\begin{example}
    Let $S = \Spec \Z$. Let $P$ be a prime, and let $D=P$ be the corresponding reduced effective divisor.

    Since $\Z$ is a PID, any rank $2$ vector bundle can be written as $\E_1 \cong \Z x_1 \oplus \Z y_1$, and the corresponding projective bundle is isomorphic to $\PP^1_{\Z}$. We write this projective bundle as 
    
    \[
    \PP(\E_1) = \PP^1_{[x_1:y_1]}.
    \]
    We can identify 
    \[
    \mc{V}(\E_1)\cong \Sym^3 \E_1\otimes (\wedge^2 \E_1)^{\vee}
    \]
    with the set of forms
    \[
    s = (A_0x_1^3 + A_1 x_1^2 y_1 + A_2 x_1 y_1^2 + A_3 y_1^3) (x_1\wedge y_1)^{\vee}
    \]
    where $A_0, A_1, A_2, A_3 \in \Z$.

    Choose a $D$-marking 
    \[
    \phi_1: \Z x_1 \oplus \Z y_1 \twoheadrightarrow \Z/P\Z.
    \]
    By a $\GL_2(\Z)$ change of basis, we may assume $\phi$ sends $x_1$ to $0$ and $y_1$ to $1$. 

    The corresponding subscheme $\DD_1 \subseteq \PP^1_{[x_1:y_1]}$ is the relative point $(x_1:y_1) = (0:1)$ over $P$. A section $s = (A_0, A_1, A_2, A_3)$ is singular at $\DD_1$ if and only if $s \in \mc{I}_{\DD_1}^2$, if and only if $P^2 \mid A_3$ and $P \mid A_2$.

    Let
    \[
    \E_2 = \ker(\phi) = \Z x_2 \oplus \Z y_2
    \]
    where the map $\E_2 \to \E_1$ is the matrix
    \[
    \begin{pmatrix} 
    1 & 0 \\
    0 & P
    \end{pmatrix}
    \]
    sending $x_2 \to x_1$ and $y_2 \to P y_1$. We can see that $\ElmD$ sends $\phi_1$ to the $D$-marking
    \[
    \phi_2: \Z x_2 \oplus \Z y_2 \twoheadrightarrow \Z/P\Z
    \]
    sending $x_2$ to $1$ and $y_2$ to $0$.

    Let $\PP_1:= \PP(\E_1) = \PP_{[x_1: y_1]}$ and $\PP_2:= \PP(\E_2) = \PP_{[x_2: y_2]}$. 
    
    We will not prove this, but one can show that the common blowup of $\PP_1$ and $\PP_2$ is the vanishing scheme
    \[
    V(x_1y_2 - P x_2 y_1) \subseteq \PP_1 \times_{\Z} \PP_2.
    \]

    We also claim without proof that the bijection in \Cref{prop:zhao-lemma} associates to the form
    \[
    s_1 = A_0 x_1^3 + A_1 x_1^2y_1 + P A_2' x_1y_1^2 + P^2 A_3' y_1^3
    \]
    a new form 
    \[
    s_2 = P A_0 x_2^3 + A_1 x_2^2 y_2 + A_2' x_2 y_2^2 + A_3' y_2^3
    \]
    which vanishes at $\DD_2$. (This can be checked affine-locally on $\PP_2$.)

    Heuristically, if we take a form $s_1$ singular at $(0:1)$, ``substitute'  $(x_1, y_1) = (Px_2,y_2)$, and divide by $P^2$, we get the new form $s_2$ which vanishes at $(1:0)$.
    
    This is exactly the construction in the proof of Proposition 16 of \cite{BST13}.
    
\end{example}

On a nice curve $C$ with a rank $2$ vector bundle $\mc{E}$ and a $D$-marking $\mc{D}$, there usually no automorphism $\Aut(\mc{E})$ which takes all of the points (or even one point) of $\DD$ to $(0:1)$. However, we can think of our transformation as a formal operation which executes the same procedure as the previous example on small open subsets of $C$.

\subsection{Summing over vector bundles}
Now we restrict again to the case of a nice (smooth, projective, geometrically irreducible) curve $C$, which is a special case of a Dedekind scheme.

Even in light of the bijection in \Cref{prop:zhao-lemma}, it can be challenging to count and sieve out singular sections from a given surface $\PP(\mc{E})$. If $\mc{D}_1 \subseteq \PP(\mc{E}_1)$ is a $D$-marking, it can be difficult to identify and work with the exact $D$-marking $(\mc{E}_2, \mc{D}_2)$ on the other side of the bijection.

We observe that the type of sum we aim to find is roughly in the form
\[
\sum_{\substack{\mc{E} \in |\Bun(C)| \\\deg \mc{E} = N} }\frac{1}{|\Aut \mc{E}|} \sum_{\mc{D} \in \Mar(\mc{E}, D)} |\Sing(\mc{E}, \DD)^{\ir}|.
\]
Therefore, instead applying the bijection ``pointwise'' for each pair $(\mc{E}, \DD)$, we take advantage of our weighted sum over all pairs $(\mc{E}, \DD)$ with (say) $\deg \mc{E}=N$ fixed. In particular, we will not attempt to count smooth curves inside a fixed ruled surface $\PP(\mc{E})$. (For discussion of the count of trigonal curves embedded inside a fixed Hirzebruch surface, see \cite{EW15} (Theorem 9.5).)

This is one significant point of departure between our approach and that of \cite{Zha13}.

\begin{prop}\label{prop:bijection-enhanced}
    Let $C$ be a nice curve, and let $D$ be a reduced effective divisor on $C$. Let $\mc{N}$ be a line bundle on $C$. Then
    \[
    \sum_{\substack{\E_1 \in |\Bun(C)| \\ \wedge^2\E_1 \cong \mc{N}}} \frac{1}{|\Aut \E_1|}\sum_{\DD_1 \in \Mar(\E_1, D)} |\Sing(\E_1, \DD_1)^{\ir}| = \sum_{\substack{\E_2\in |\Bun(C)|\\\wedge^2 \E_2 \cong \mc{N}(-D)}}\frac{1}{|\Aut{\E_2}|}\sum_{\DD_2 \in \Mar(\E_2,D)} |\Van(\E_2, \DD_2)^{\ir}|.
    \]

    More generally, let $\xi_U(s_1)$ be a function of a section $s_1$ which only depends on the isomorphism class of the vanishing scheme $(C_1)_{U} \to U$. Then
    \begin{equation}\label{eq:A}
    \sum_{\substack{\mc{E}_1 \in |\Bun(C)| \\ \wedge^2 \mc{E}_1 \cong \mc{N}}} \frac{1}{|\Aut \mc{E}_1|} \sum_{\mc{D}_1 \in \Mar(\mc{E}_1, D)} \sum_{s_1 \in \Sing(\mc{E}_1, \DD_1)^{\ir}} \xi_{U}(s_1)
    \end{equation}
    is equal to
    \begin{equation}\label{eq:B}
    \sum_{\substack{\mc{E}_2 \in |\Bun(C)|\\ \wedge^2 \mc{E}_2 \cong \mc{N}(-D)}} \frac{1}{|\Aut \mc{E}_2|} \sum_{\DD_2 \in \Mar(\mc{E}_2, D)} \sum_{s_2 \in \Van(\mc{E}_2, \DD_2)^{\ir}} \xi_U(s_2).
    \end{equation}
\end{prop}
\begin{proof} We'll prove the second half of the proposition, which implies the first half.

Let $(\mc{E}_2, \DD_2) \in \DBun$ be a $D$-marking. Define the function
\begin{equation}\label{eq:s2-formula}
\Xi(\mc{E}_2, \DD_2):= \sum_{s_2 \in \Van(\mc{E}_2, \DD_2)^{\ir}} \xi_U(s_2).
\end{equation}
This is a well-defined function that only depends on the isomorphism class of $(\mc{E}_2, \mc{D}_2) \in \DBun$. Now if $(\mc{E}_1, \DD_1) \in \DBun$ satisfies $(\mc{E}_2, \DD_2) = \ElmD(\mc{E}_1, \DD_1)$, then by \Cref{prop:zhao-lemma}, we have
\[
\sum_{s_2 \in \Van(\mc{E}_2, \DD_2)^{\ir}} \xi_U(s_2)= \sum_{s_1 \in \Sing(\mc{E}_1, \DD_1)^{\ir}} \xi_U(s_1).
\]
So we have
\begin{equation}\label{eq:s1-formula}
\Xi(\ElmD(\mc{E}_1, \DD_1)) = \Xi(\mc{E}_2, \DD_2) = \sum_{s_1 \in \Sing(\mc{E}_1, \DD_1)^{\ir}} \xi_U(s_1).
\end{equation}

Now recall by \Cref{prop:elm-groupoid-equiv} that 
\[
\ElmD: \DBun_{\mc{N}} \to \DBun_{\mc{N}(-D)}
\]
is an equivalence of groupoids.

Furthermore, the set of automorphisms of a pair $(\mc{E}_1, \DD_1) \in \DBun$ is isomorphic to the stabilizer of $\DD_1$ under the action of $\Aut \mc{E}_1$ on $\Mar(\mc{E}_1, D)$. In particular, by the Orbit-Stabilizer Theorem we have
\begin{equation}\label{eq:middle}
\begin{split}
    &\sum_{\substack{\mc{E}_1 \in |\Bun(C)|\\ \wedge^2 \mc{E}_1 \cong \mc{N}}} \frac{1}{|\Aut \mc{E}_1|} \sum_{\mc{D}_1 \in \Mar(\mc{E}_1, D)} \Xi(\ElmD(\mc{E}_1, \DD_1))\\
    &= \sum_{\substack{(\mc{E}_1, \DD_1)\in |\DBun|\\ \wedge^2 \mc{E}_1 \cong \mc{N}}} \frac{1}{|\Aut(\mc{E}_1, \DD_1)|} \Xi(\ElmD(\mc{E}_1, \DD_1))\\
    &= \sum_{\substack{(\mc{E}_2, \DD_2) \in |\DBun|\\ \wedge^2 \mc{E}_2 \cong \mc{N}(-D)}} \frac{1}{|\Aut(\mc{E}_2, \DD_2)|} \Xi(\E_2, \DD_2)\\
    &=\sum_{\substack{\mc{E}_2\in |\Bun(C)|\\ \wedge^2 \mc{E}_2 \cong \mc{N}(-D)}} \frac{1}{|\Aut \mc{E}_2|} \sum_{\DD_2 \in \Mar(\mc{E}_2, D)} \Xi(\mc{E}_2, \DD_2).
\end{split}
\end{equation}
So we have
\begin{align*}
        &\sum_{\substack{\mc{E}_1 \in |\Bun(C)| \\ \wedge^2 \mc{E}_1 \cong \mc{N}}} \frac{1}{|\Aut \mc{E}_1|} \sum_{\mc{D}_1 \in \Mar(\mc{E}_1, D)} \sum_{s_1 \in \Sing(\mc{E}_1, \DD_1)^{\ir}} \xi_{U}(s_1)\\
        &=     \sum_{\substack{\mc{E}_1 \in |\Bun(C)|\\ \wedge^2 \mc{E}_1 \cong \mc{N}}} \frac{1}{|\Aut \mc{E}_1|} \sum_{\mc{D}_1 \in \Mar(\mc{E}_1, D)} \Xi(\ElmD(\mc{E}_1, \DD_1))\\
        &=\sum_{\substack{\mc{E}_2\in |\Bun(C)|\\ \wedge^2 \mc{E}_2 \cong \mc{N}(-D)}} \frac{1}{|\Aut \mc{E}_2|} \sum_{\DD_2 \in \Mar(\mc{E}_2, D)} \Xi(\mc{E}_2, \DD_2)\\
        &=    \sum_{\substack{\mc{E}_2 \in |\Bun(C)|\\ \wedge^2 \mc{E}_2 \cong \mc{N}(-D)}} \frac{1}{|\Aut \mc{E}_2|} \sum_{\DD_2 \in \Mar(\mc{E}_2, D)} \sum_{s_2 \in \Van(\mc{E}_2, \DD_2)^{\ir}} \xi_U(s_2),
\end{align*}
where the first equality holds by \eqref{eq:s1-formula}, the second equality holds by \eqref{eq:middle}, and the third equality holds by \eqref{eq:s2-formula}.

\end{proof}

\section{Further Reduction of the Main Sum}
\label{sec:sieve-ii}
Using the results of \Cref{sec:hecke-mod}, we will continue to transform our main sum $\Theta(C,N)$ into a more manageable form.

Recall that we showed
\begin{equation}\label{eq:theta-version-1}
\begin{split}
\Theta(C, N) =
\sum_{\substack{\mc{E} \in |\Bun(C)| \\\deg \mc{E} = N}} \frac{1}{|\Aut \mc{E}|} \sum_{D_1, D_2, D_3} \mu(D_1)\mu(D_3) \sum_{\substack{\DD_1 \in \Mar(\mc{E}, D_1) \\ \DD_2 \in \Mar(\mc{E}, D_2)}} |\Sing(\mc{E}, \DD_1)\cap \SingFib(\mc{E}, \DD_2) \cap \Fib(\mc{E},D_3)^{\ir}|
\end{split}
\end{equation}
in \Cref{prop:version-05}, where $D_1, D_2, D_3$ range over all triples of disjoint reduced effective divisors on $C$.

We would like to apply the correspondence \Cref{prop:zhao-lemma} to reduce a count of sections singular at points to a count of sections vanishing at points on a different ruled surface. This reduction is analagous to the ``discriminant reducing identity'' from \cite{BST13} (Section 9.1).

To count the total number of sections vanishing at various points on a ruled surface, it will be useful to consider the number of (relative degree $1$) points vanishing at one section on a fiber.

\begin{defn}
    Let $C$ be a nice curve, and let $\mc{E}$ be a rank $2$ vector bundle on $C$. Let $s \in \mc{V}(\mc{E})$ and let $P \in C$ be a closed point. 
    
    We let $r_P(s)$ be the the number of degree $1$ roots of $s|_{P} \in \mc{O}(3)$ on $\PP(\mc{E}) \cong \PP^1_{P}$. 
\end{defn}
Equivalently, if we associate to $s|_{P}$ a binary cubic form 
\[
\Lambda(s|_{P}) = a_0 x^3 + a_1 x^2y + a_2 xy^2 + a_3 y^3
\]
with coefficients $a_i$ in the finite residue field $\F_q(P)$, then $r_P(s)$ is the number of distinct linear factors of the form.

In particular, we will always have $r_P(s) \in \{0,1,2,3, q^{\deg P}+1\}$, with the final case occurring if and only if $s|_{P}=0$.

\begin{defn}
    Let $C$ be a nice curve, and let $\mc{E}$ be a rank $2$ vector bundle on $C$. Let $s \in \mc{V}(\mc{E})$, and let $D$ be a reduced effective divisor on $C$.

    We let $r_D(s)$ denote the number of degree $1$ roots of $s|_{D} \in \mc{O}(3)$ on $\PP(\mc{E}) \cong \PP^1_{D}$. Equivalently, we have
    \begin{itemize}
        \item $r_D(s)$ is the number of $D$-markings $\mc{D} \in \Mar(\mc{E}, D)$ such that $s$ vanishes on $\DD$.
        \item We can write
        \[
        r_D(s) = \prod_{P \in D} r_P(s).
        \]
    \end{itemize}
\end{defn}

In order to simplify $\eqref{eq:theta-version-1}$, we will start by considering fibral sections.

\begin{lemma}\label{lem:fib-mult}
    Let $C$ be a nice curve, and let $\mc{E}$ be a rank $2$ vector bundle on $C$. Let $D$ be a reduced effective divisor on $C$, and let $U = C-D$. Then we have an isomorphism
    \[
    \mc{V}(\mc{E}, (\wedge^2 \mc{E})^{\vee}\otimes \mc{O}(-D)) \xrightarrow{\times \sigma_D} \Fib(\mc{E}, D) \subseteq V(\mc{E}).
    \]
    given by multiplication by $\sigma_D \in \pi^{*} \mc{O}(D)$ is the section cutting out $\pi^{-1}(D)$. Therefore the composite map preserves $s|_{U}$ for a section $s$.
\end{lemma}
\begin{proof}
    Let $\pi: \PP(\mc{E}) \to C$ be the projective bundle. The section of $\mc{O}_C(D)$ cutting out $D$ pulls back to a section of $\pi^{*} \mc{O}_{C}(D)$ cutting out $\pi^{-1}(D)$.
    
    Therefore 
    \[
    \Fib(\mc{E}, D) \cong H^0(\PP(\mc{E}), \mc{O}_{\PP(\mc{E})} \otimes (\wedge^2 \mc{E})^{\vee} \otimes \pi^{*} \mc{O}(-D)) = \mc{V}(\mc{E}, (\wedge^2 \mc{E})^{\vee} \otimes \mc{O}(-D)).
    \]
    It is clear that since $\Fib(\mc{E}, D)$ injects into $\mc{V}(\mc{E})$ by the multiplication map by a section of $\pi^{*} \mc{O}(D)$, we have that the inclusion is an isomorphism when restricted to $U$.
\end{proof}
\begin{lemma}
    Let $C$ be a nice curve. Let $\mc{E}_1$ be a rank $2$ vector bundle on $C$, and let $D$ be a reduced effective divisor on $C$. Let $\mc{E}_2:= \mc{E}_1\otimes \mc{O}(-D)$. Then
    \[
    \mc{V}(\mc{E}_1, (\wedge^2 \mc{E}_1)^{\vee} \otimes \mc{O}(-D)) \cong \mc{V}(\mc{E}_2).
    \]
\end{lemma}
\begin{proof}
    Unpacking definitions, we just want to show that
    \[
    H^0(\PP(\mc{E}_1), \mc{O}_{\PP(\mc{E}_1)}(3) \otimes (\wedge^2 \mc{E}_1)^{\vee} \otimes \mc{O}(-D) ) \cong H^0(\PP(\mc{E}_2), \mc{O}_{\PP(\mc{E}_2)}(3) \otimes (\wedge^2 \mc{E}_2)^{\vee}).
    \]
    Now the twisted isomorphism $\mc{E}_2 = \mc{E}_1 \otimes \mc{O}(-D)$ induces an isomorphism $\PP(\mc{E}_2) \cong \PP(\mc{E}_2)$ such that $\mc{O}_{\PP(\mc{E}_2)}\cong \mc{O}_{\PP(\mc{E}_1)} \otimes \mc{O}(-D)$. Also, $\wedge^2 \mc{E}_2 \cong \wedge^2 \mc{E}_1 \otimes \mc{O}(-2D)$. Together, these give the desired isomorphism.
\end{proof}

Combining the last two lemmas, we get:
\begin{lemma}
    Let $C$ be a nice curve, let $\mc{E}$ be a rank $2$ vector bundle on $C$, and let $D$ be a reduced effective divisor on $C$. Then we have an isomorphism
    \begin{equation}\label{eq:fib-reduce}
\mc{V}(\mc{E}(-D)) \xrightarrow{\sim} \Fib(\mc{E}, D)
\end{equation}
which is given by the composition of an isomorphism $\PP(\mc{E}(-D)) \xrightarrow{\sim} \PP(\mc{E})$ with the multiplication by $\sigma_D$. 
If $U = C-D$, then this map preserves sections over $U$. More precisely, if \eqref{eq:fib-reduce} sends $s' \in \mc{V}(\mc{E}(-D))$ to $s \in \Fib(\mc{E}, D) \subseteq \mc{V}(E)$, then above $U$ the map is the natural identification $s'|_{U} \to s|_{U}$ under the isomorphism $\PP(\mc{E}(-D)) \cong \PP(\mc{E})$.

In particular, the map in \eqref{eq:fib-reduce} preserves horizontal irreducibility of sections.
\end{lemma}

Therefore we have:
\begin{lemma}\label{lem:singvan}
    Let $C$ be a nice curve. Let $\mc{E}$ be a rank $2$ vector bundle. Let $D_1, D_2, D_3$ be disjoint reduced effective divisors on $C$. Let $\DD_1\subseteq \PP(\mc{E})$ be a $D_1$-marking, let $\DD_2 \subseteq \PP(\mc{E}))$ be a $D_2$-marking. Let $\mc{E}' = \mc{E}(-D_2-D_3)$. Let $\DD_1'$ and $\DD_2'$ denote the images of $\DD_1$ and $\DD_2$ under the natural isomorphism $\PP(\mc{E}) \to \PP(\mc{E}')$. Then there is an isomorphism
    \[
    \Sing(\mc{E}, \DD_1) \cap \SingFib(\mc{E}, \DD_2) \cap \Fib(\mc{E}, D_3) \cong \Sing(\mc{E}', \DD_1') \cap \Van(\mc{E}', \DD_2')
    \]
    which restricts to a bijection on the subsets of horizontally irreducible sections
    \[\Sing(\mc{E}, \DD_1) \cap \SingFib(\mc{E}, \DD_2) \cap \Fib(\mc{E}, D_3)^{\ir} \cong \Sing(\mc{E}', \DD_1') \cap \Van(\mc{E}', \DD_2')^{\ir}.
    \]
\end{lemma}
\begin{proof}
Consider the first isomorphism. The left hand side is a subspace of $\Fib(\mc{E}, D_2+D_3)$, and the right hand side is a subspace of $\mc{V}(E') = \mc{V}(E(-D_1-D_2))$. We claim that the two sides are identified under the isomorphism 
\[
\gamma: \mc{V}(\mc{E}') = \mc{V}(\mc{E}(-D_1-D_2)) \xrightarrow{\sim} \Fib(\mc{E}, D)
\]
from \eqref{eq:fib-reduce}.

The isomorphism $\gamma$ sends $\Sing(\mc{E}', \DD_1')$ to $\Sing(\mc{E}, \DD_1)$, since it preserves sections away from $D_2+D_3$.

Next, we claim the isomorphism $\gamma$ sends $\Van(\mc{E}', \DD_2')\subseteq \mc{V}(\mc{E}')$ to $\SingFib(\mc{E}, \DD_2)\subseteq \Fib(\mc{E},D_2+D_3)$. 

This follows because a section $s' \in  \mc{V}(\mc{E}')$ vanishes to order $i$ at $\DD_2'$ if and only if the product $\sigma_D s$ vanishes to order $i+1$ at $\DD_2'$, where $\sigma_D$ was defined in \Cref{lem:fib-mult}.

Putting both claims together gives the first isomorphism. The second isomorphism follows because the sections are preserved away from $D_2+D_3$, so in particular horizontally irreducible sections are preserved.
\end{proof}

Using \Cref{lem:singvan}, we can now find a new, slightly simpler formula for $\Theta(C, N)$.

\begin{prop}
Let $C$ be a nice curve and let $N$ be an integer. We have
\[
\Theta(C,N) = \sum_{D_1, D_2, D_3} \mu(D_1) \mu(D_3) \sum_{\substack{\mc{E}' \in |\Bun(C)| \\\deg \mc{E}' = N'}} \frac{1}{|\Aut{\mc{E}'}|}\sum_{\substack{\DD_1 \in \Mar(\mc{E}', D_1)\\ \DD_2 \in \Mar(\mc{E}', D_2)}} |\Van(\mc{E}', \mc{D}_1 \cup \mc{D}_2)^{\ir}|
\]
where $N' = N-\deg D_1 - 2 \deg D_2 - 2 \deg D_3$, and $D_1,D_2,D_3$ range over disjoint reduced effective divisors on $C$.
\end{prop}
\begin{rem}
    We can view this as a function field analogue of the ``discriminant reducing identity'' in Section 9.1 of \cite{BST13}.
\end{rem}
\begin{proof}

The assignment $\mc{E} \to \mc{E}(-D_2-D_3)$ induces an equivalence of groupoids between rank $2$ vector bundles on $C$ of determinant $\mc{N}$ and rank $2$ vector bundles on $C$ of determinant $\mc{N}(-2D_2-2D_3)$.

Therefore, for a fixed line bundle $\mc{N}$ on $C$, we can write
\begin{align*}
&\sum_{\substack{\mc{E}\in |\Bun(C)| \\\wedge^{2}\mc{E}\cong \mc{N}} }\frac{1}{|\Aut \mc{E}|} \sum_{\substack{\DD_1 \in \Mar(\mc{E}, D_1) \\ \DD_2 \in \Mar(\mc{E}, D_2)}} |\Sing(\mc{E}, \DD_1)\cap \SingFib(\mc{E}, \DD_2) \cap \Fib(\mc{E}, D_3)^{\ir}|\\
&=\sum_{\substack{\mc{E}' \in |\Bun(C)|\\\wedge^{2}\mc{E}'\cong \mc{N}(-2D_2-2D_3)} }\frac{1}{|\Aut \mc{E}'|} \sum_{\substack{\DD_1 \in \Mar(\mc{E'}, D_1) \\ \DD_2 \in \Mar(\mc{E'}, D_2)}} |\Sing(\mc{E}', \DD_1) \cap \Van(\mc{E}', \DD_2)^{\ir}|\\
&=\sum_{\substack{\mc{E}' \in |\Bun(C)|\\\wedge^2 \mc{E}' \cong \mc{N}(-2D_2-2D_3)} }\frac{1}{|\Aut \mc{E}'|} \sum_{\DD_1 \in \Mar(\mc{E}', D_1)} \sum_{s \in \Sing(\mc{E}', \DD_1)^{\ir}} r_{D_2}(s).
\end{align*}

Now by \Cref{prop:bijection-enhanced} applied to the function $\xi_{U}(s) = r_{D_2}(s)$, we can further rewrite
\begin{align*}
&\sum_{\substack{\mc{E}'\in |\Bun(C)|\\\wedge^2 \mc{E}' \cong \mc{N}(-2D_2-2D_3)}} \frac{1}{|\Aut \mc{E}'|} \sum_{\DD_1 \in \Mar(\mc{E}', D_1)} \sum_{s \in \Sing(\mc{E}', \DD_1)^{\ir}} r_{D_2}(s)\\
&=\sum_{\substack{\mc{E}'' \in |\Bun(C)|\\\wedge^2 \mc{E}'' \cong \mc{N}(-D_1-2D_2-2D_3)}} \frac{1}{|\Aut \mc{E}''|} \sum_{\mc{D}_1 \in \Mar(\mc{E}'', D_1)} \sum_{s \in \Van(\mc{E}'', \DD_1)^{\ir}} r_{D_2}(s)\\
&=\sum_{\substack{\mc{E}''\in |\Bun(C)|\\\wedge^2 \mc{E}'' \cong \mc{N}(-D_1-2D_2-2D_3) }} \frac{1}{|\Aut \mc{E}''|}\sum_{\substack{\DD_1 \in \Mar(\mc{E}'', D_1)\\ \DD_2 \in \Mar(\mc{E}'', D_2)}} |\Van(\mc{E}'', \mc{D}_1 \cup \mc{D}_2)^{\ir}|.
\end{align*}

Combining the last two equalities, we get that for a fixed line bundle $\mc{N}$ on $C$, we have
\[
\sum_{\substack{\mc{E}\in |\Bun(C)| \\\wedge^{2}\mc{E}\cong \mc{N}} }\frac{1}{|\Aut \mc{E}|} \sum_{\substack{\DD_1 \in \Mar(\mc{E}, D_1) \\ \DD_2 \in \Mar(\mc{E}, D_2)}} |\Sing(\mc{E}, \DD_1)\cap \SingFib(\mc{E}, \DD_2) \cap \Fib(\mc{E}, D_3)^{\ir}|
\]
\[
=\sum_{\substack{\mc{E}''\in |\Bun(C)|\\\wedge^2 \mc{E}'' \cong \mc{N}(-D_1-2D_2-2D_3) }} \frac{1}{|\Aut \mc{E}''|}\sum_{\substack{\DD_1 \in \Mar(\mc{E}'', D_1)\\ \DD_2 \in \Mar(\mc{E}'', D_2)}} |\Van(\mc{E}'', \mc{D}_1 \cup \mc{D}_2)^{\ir}|.
\]

By adding this up for all line bundles $\mc{N}$ of degree $N$, we find
\[
\sum_{\substack{\mc{E} \in |\Bun(C)|\\\deg \mc{E} = N} }\frac{1}{|\Aut \mc{E}|} \sum_{\substack{\DD_1 \in \Mar(\mc{E}, D_1) \\ \DD_2 \in \Mar(\mc{E}, D_2)}} |\Sing(\mc{E}, \DD_1)\cap \SingFib(\mc{E}, \DD_2) \cap \Fib(\mc{E}, D_3)^{\ir}|
\]
\[
=\sum_{\substack{\mc{E}''\in |\Bun(C)|\\\deg \mc{E}'' = N - \deg D_1 - 2\deg D_2 - 2 \deg D_3}} \frac{1}{|\Aut \mc{E}''|}\sum_{\substack{\DD_1 \in \Mar(\mc{E}'', D_1)\\ \DD_2 \in \Mar(\mc{E}'', D_2)}} |\Van(\mc{E}'', \mc{D}_1\cup \mc{D}_2)^{\ir}|.
\]

So finally, we have a new formula for $\Theta(C,N)$ as desired. $\Theta(C,N)$ is equal to 

\begin{align*}
 & \sum_{D_1, D_2, D_3} \mu(D_1)\mu(D_3) \sum_{\substack{\mc{E}\in |\Bun(C)| \\\deg \mc{E} = N}} \frac{1}{|\Aut \mc{E}|} \sum_{\substack{\DD_1 \in \Mar(\mc{E}, D_1) \\ \DD_2 \in \Mar(\mc{E}, D_2)}} |\Sing(\mc{E}, \DD_1)\cap \SingFib(\mc{E}, \DD_2) \cap \Fib(\mc{E}, D_3)^{\ir}|\\
&= \sum_{D_1, D_2, D_3} \mu(D_1) \mu(D_3) \sum_{\substack{\mc{E}'' \in |\Bun(C)|\\ \deg \mc{E}'' = N'}} \frac{1}{|\Aut{\mc{E}''}|}\sum_{\substack{\DD_1 \in \Mar(\mc{E}'', D_1)\\ \DD_2 \in \Mar(\mc{E}'', D_2)}} |\Van(\mc{E}'', \mc{D}_1 \cup \mc{D}_2)^{\ir}|
\end{align*}
where $N' = N - \deg D_1 - 2 \deg D_2 - 2 \deg D_3$.
\end{proof}

We will often encounter the sums of the form
\[
\sum_{\DD \in \Mar(\mc{E}, D)} |\Van(\mc{E}, D)^{\ir}| = \sum_{s \in \mc{V}(\mc{E})^{\ir}} r_{D}(s)
\]
for some rank $2$ vector bundle $\mc{E}$ on $C$ and some reduced effective divisor $D$ on $C$. As such, we will define the following notation:
\begin{defn}
    Let $C$ be a nice curve. Let $\mc{E}$ be a rank $2$ vector bundle on $C$, and let $D$ be a reduced effective divisor. We define the \textbf{root-counting function}
    \[
    R^{\ir}(\mc{E}, D):= \sum_{\DD \in \Mar(\mc{E}, D)} |\Van(\mc{E}, D)^{\ir}| = \sum_{s \in \mc{V}(\mc{E})^{\ir}} r_{D}(s).
    \]
\end{defn}

So for example, we could write
\begin{equation}\label{eq:theta-Rir}
\Theta(C,N) = \sum_{D_1, D_2, D_3} \mu(D_1) \mu(D_3) \sum_{\substack{\mc{E}' \in |\Bun(C)| \\\deg \mc{E}' = N'}} \frac{1}{|\Aut \mc{E}'|} R^{\ir}(\mc{E}',D_1+D_2)
\end{equation}
where $N' = N- \deg D_1 - 2 \deg D_2 - 2 \deg D_3$. 

It will also be useful to define an auxiliary function $a_P(s)$.
\begin{defn}
    Let $C$ be a nice curve, and let $\mc{E}$ be a rank $2$ vector bundle on $C$. Let $P \in C$ be a closed point.

    We define 
    \[
    a_P(s):= r_P(s)-1.
    \]

    Let $D$ be a reduced effective divisor. We define
    \[
    a_D(s) := \prod_{P \in D} a_P(s).
    \]

    Finally, we define the \textbf{adjusted root-counting function}
    \[
    \Phi^{\ir}(\mc{E}, D):= \sum_{s \in \mc{V}(\mc{E})^{\ir}} a_D(s)=\sum_{D_1 \mid D} \mu(D-D_1) R^{\ir}(\mc{E}, D_1).
    \]
\end{defn}
\subsection{Generating Function}

We find that it is easiest to deal with $\Theta(C,N)$ if we package all of the $\Theta(C,N)$ into a generating function for varying $N$.

We would like to consider the formal power series $\sum_{N \ge 0}  \Theta(C,N)T^{N}$. For technical reasons, we will make the following assumption.

\begin{assumption}\label{assumption}
    Assume either $C \cong \PP^1$ or $\Char \F_q \neq 3$.
\end{assumption}

We make this assumption solely in order to ensure that the generating function given by $\sum_{N} \Theta(C,N) T^{N}$ only has terms with $N \ge 0$. This follows from the following lemma:

\begin{lemma}
    Assume either $C \cong \PP^1$ or $\Char \F_q \neq 3$. Then:
    \begin{itemize}
        \item If $\mc{V}(\mc{E})^{\ir}$ is nonempty, then $\deg \mc{E} \ge 0$.
        \item If $\Theta(C,N) \neq 0$, then $N\ge 0$.
    \end{itemize}
\end{lemma}
\begin{proof}
See \Cref{sec:appendix-hor-finite}.

\end{proof}

\begin{rem}
    It may be surprising that it is possible that 
    \[
    \Theta(C,N) = \sum_{\substack{f: X \to C\\ \deg \Delta_f =2N\\ X\text{ smooth, irreducible}}}\frac{1}{|\Aut f|}
    \]
    can be nonzero when $N<0$. (Recall here the sum ranges over triple covers $f: X \to C$ such that $X$ is smooth and irreducible.)

    Indeed, if the discriminant $\Delta_f \in \mc{N}^{\otimes 2}$ is nonzero, then its degree must be nonnegative. However, if $\Char \F_q=3$, then we could have $\Delta_f=0$ even when $X$ is smooth and irreducible. In this case the discriminant degree $\deg \Delta_f$ is (by definition) $\deg \mc{N}^{\otimes 2} = \deg (\wedge^{3} f_{*} \mc{O}_X)^{\otimes -2}$, which could be negative when the genus of $C$ is large enough.

    For example, if $\Char \F_q=3$ and $f: X \to C$ is a (purely) inseparable triple cover with $X$ smooth, projective, and geometrically irreducible, corresponding to a smooth section $s \in \mc{V}(\mc{E})^{\ir}$, then 
    \[
    \deg \Delta_f = 2 \deg \mc{E} = 2(1-g_C)<0
    \]
    when $g_C> 1$.\footnote{To show $2 \deg \mc{E} = 2(1-g_C)$ in this case, first recall $g_X = g_C$ for a purely inseparable extension of nice curves. We can also apply the adjunction formula to the smooth divisor $X\subseteq \PP(\mc{E})$ to conclude $(2g_X-2) = 3(2g_C-2) = 2 \deg \mc{E}$, from which the formula follows.}

    If desired, we could easily adjust our generating function in this case to be a formal Laurent series with finitely many terms of negative degree. However, in the following section, we will assume $C \cong \PP^1$ anyway, so for the sake of simplicity we will take the assumption as a given from now on. 
\end{rem}

We move forward under \Cref{assumption}. We define a generating function $G_C(T) \in \C[[T]]$ attached to $\Theta(C,N)$ by the formal power series
\begin{equation}\label{eq:GC-defn}
G_C(T):= \sum_{N\ge 0} \Theta(C,N) T^{N}.
\end{equation}

We also recall the \textbf{formal zeta function} associated to the curve $C$, defined by
\[
Z_C(T) := \prod_{P \in C} \left(1 - T^{- \deg P}\right)^{-1} \in \C[[T]].
\]

It turns out that it will be natural to consider a ``normalized'' generating function $G_C(T) Z_C(T)$.
\begin{defn}\label{defn:power-series-eq}
    Let $\Psi^{\ir}(C,N)$ be the unique set of coefficients such that
    \[
    G_C(T) Z_C(T) = \sum_{N \ge 0} \Psi^{\ir}(C,N) T^{N}.
    \]
\end{defn}
We claim that there is a relatively simple formula for  $\Psi^{\ir}(C,N)$.
\begin{prop}\label{prop:Psi-formula}Assume either $C\cong \PP^1$ or $\Char \F_q \neq 3$. Then we have
    \[
    \Psi^{\ir}(C,N) = \sum_{\substack{N',d \ge 0 \\ N' + d = N}}\sum_{\substack{\mc{E}' \in |\Bun(C)| \\\deg \mc{E}' = N'}} \frac{1}{|\Aut \mc{E}'|}\sum_{\deg D = d} \mu(D)\Phi^{\ir}(\mc{E}',D).
    \]
\end{prop}
\begin{proof}
    We expand the formal power series $G_C(T)$ using \eqref{eq:theta-Rir}. We can write $G_C(T)$ as 
    \begin{align*}
    &\sum_{N \ge 0} T^{N} \Theta(C,N)\\
        &= \sum_{N \ge 0} T^{N} \sum_{D_1, D_2, D_3} \mu(D_1) \mu(D_3) \sum_{\substack{\mc{E}' \in |\Bun(C)| \\\deg \mc{E}' = N'}} \frac{1}{|\Aut \mc{E}'|} R^{\ir}(\E', D_1+D_2)\\
        &= \sum_{N' \ge 0} T^{N'} \sum_{D_1, D_2, D_3} T^{\deg D_1 + 2 \deg D_2 + 2 \deg D_3} \mu(D_1) \mu(D_3) \sum_{\substack{\mc{E}' \\\deg \mc{E}'  = N'}} \frac{1}{|\Aut \mc{E}'|} R^{\ir}(\E', D_1 + D_2)\\
        &=\sum_{N' \ge 0} T^{N'} \sum_{D_1, D_2, D_3} T^{\deg D_1 + 2 \deg D_2 + 2 \deg D_3} \mu(D_1) \mu(D_3) \sum_{\substack{\mc{E}' \in |\Bun(C)| \\\deg \mc{E}'  = N'}} \frac{1}{|\Aut \mc{E}'|} \sum_{s \in \mc{V}(\mc{E}')^{\ir}} r_{D_1}(s) r_{D_2}(s)\\
        &=\sum_{N' \ge 0} T^{N'} \sum_{ \substack{\mc{E}' \in |\Bun(C)| \\ \deg \mc{E'} = N'} }\frac{1}{|\Aut \mc{E}'|}\sum_{s \in \mc{V}(\mc{E}')^{\ir}}\sum_{D_1, D_2, D_3} \mu(D_1) \mu(D_3) r_{D_1}(s) r_{D_2}(s)T^{\deg D_1 + 2 \deg D_2 + 2 \deg D_3}.
    \end{align*}
    Recall that here $N'$ is defined such that $N' + \deg D_1 + \deg D_2 = N$. Now we focus on the innermost sum. We can write it as
    \begin{align*}
    & \sum_{D_1, D_2, D_3} \mu(D_1) \mu(D_3) r_{D_1}(s) r_{D_2}(s) T^{\deg D_1 + 2 \deg D_2 + 2 \deg D_3}\\
    &= \prod_{P \in C} (1- r_P(s) T^{\deg P} + r_P(s) T^{2 \deg P} - T^{2 \deg P})\\
    &= \prod_{P \in C} (1-T^{\deg P}) \prod_{P \in C} (1-a_P(s) T^{\deg P})\\
    &= Z_C(T)^{-1} \sum_{D} \mu(D) T^{\deg D} a_D(s).
    \end{align*}
    where the products are over all closed points $P \in C$. 

    Therefore we can write
    \begin{align*}
    G_C(T) &=\sum_{N \ge 0}  T^{N} \Theta(C,N)\\
    &= \sum_{\mc{E}' \in |\Bun(C)|} T^{\deg \mc{E}'} \frac{1}{|\Aut \mc{E}'|} \sum_{s \in \mc{V}(\mc{E}')^{\ir}} Z_C(T)^{-1} \sum_{D} \mu(D) T^{\deg D} a_D(s)\\
    &=Z_C(T)^{-1} \sum_{\mc{E}' \in |\Bun(C)|} \sum_{D} T^{\deg \mc{E}' + \deg D} \frac{1}{|\Aut \mc{E}'|} \mu(D)\sum_{s \in \mc{V}(\mc{E}')^{\ir}} a_D(s)\\
    &= Z_C(T)^{-1} \sum_{\mc{E}'\in |\Bun(C)|} \sum_{D} T^{\deg \mc{E}' + \deg D} \frac{1}{|\Aut \mc{E}'|} \mu(D)\Phi^{\ir}(\mc{E}',D)\\
    &= Z_C(T)^{-1} \sum_{N' \ge 0} \sum_{d \ge 0} T^{N' + d} \sum_{\substack{\mc{E}' \in |\Bun(C)|\\ \deg \mc{E}' = N'} }\frac{1}{|\Aut \mc{E}'|} \sum_{\deg D = d} \mu(D) \Phi^{\ir}(\mc{E}', D)\\
    &= Z_C(T)^{-1} \sum_{N \ge 0} T^{N} \Psi^{\ir}(C,N).
    \end{align*}
    So
    \begin{align*}
    \sum_{N \ge 0} T^{N} \Psi^{\ir}(C,N) &= Z_C(T) G_C(T) \\
    &= \sum_{N \ge 0}T^{N} \sum_{\substack{N', d \ge 0\\ N' + d = N}} \sum_{\substack{\mc{E}' \in |\Bun(C)| \\\deg \mc{E}' = N'}} \frac{1}{|\Aut \mc{E}'|} \sum_{\substack{D\\\deg D = d}} \mu(D) \Phi^{\ir}(\mc{E}', D)
    \end{align*}
    and taking coefficients yields the desired result.
\end{proof}

\section{Hirzebruch Surfaces}\label{sec:hirz}
We specialize now to the case of the base curve $C = \PP^{1}_{\Fq}$.

By the Birkhoff-Grothendieck theorem, every rank $2$ vector bundle $\mc{E}$ on $\PP^1$ splits as $\mc{E} \cong \mc{O}_{\PP^1}(a_1) \oplus \mc{O}_{\PP^1}(a_2)$ for some pair of integers $a_1 \le a_2$. Recall also that $\PP(\E) \cong \PP(\E')$ if $\E' \cong \E \otimes \LL$ for some line bundle $\LL$. Consequently, every ruled surface $\PP(\E)$ over $\PP^1$ is isomorphic to a \textbf{Hirzebruch surface} $F_k$, defined by
\[
F_k:= \PP(\mc{O}_{\PP^1}\oplus \mc{O}_{\PP^1}(-k)),
\]
for some unique integer $k\ge 0$.

We write $(t_0: t_1)$ for the projective coordinates on the base $C = \PP^1_{[t_0:t_1]}$. For an integer $n$, a global section in $H^0(\PP^1, \mc{O}_{\PP^1}(n))$ can be canonically identified with a degree $n$ homogeneous polynomial $A(t_0, t_1)   \in \F_q[t_0, t_1]$. (When $n<0$, the only global section is $A(t_0, t_1) =0$.)

\subsection{Background on Hirzebruch Surfaces}
As a basic point of reference for ruled surfaces and Hirzebruch surfaces, see Hartshorne \cite[\S V.2]{Har77}.\footnote{Hartshorne works over an algebraically closed field, but it will not be hard to show that the properties we need are true over over $\F_q$. In particular, intersection numbers and cohomological vanishing can be checked after base change to an algebraic closure.}

For a Hirzebruch surface $\pi: F_k \to \PP^1$, there is a divisor $B_k \subseteq F_k$ known as the \textbf{directrix}. It is the image of a section of $\pi: F_k \to \PP^1$, induced by the surjective morphism
\[
\mc{O}\oplus \mc{O}(-k) \xrightarrow{(0, \Id)} \mc{O}(-k)
\]
on the base $\PP^1$.

The directrix $B_k$ satisfies $\mc{O}(B_k) = \mc{O}_{F_k}(1)$, and the self intersection of the directrix is $B_k\cdot B_k = -k$. (See \cite{Har77} Proposition 2.8, Example 2.11.3, and Corollary 2.13). As such, when $k\ge 1$, $B_k$ is also sometimes known as the \textbf{negative section} of $F_k$.

We can decompose
\[
\Pic F_k \cong \Z \oplus \pi^{*} \Pic \PP^1
\]
where the first $\Z$ is generated by the line bundle $\mc{O}_{F_k}(1)$, or equivalently by $\mc{O}(B_k)$. (See \cite{Har77} Proposition 2.3.) Furthermore, on $\PP^1$ the degree map $\deg: \Pic \PP^1 \xrightarrow{\sim} \Z$ is an isomorphism. We let $f$ denote the divisor class of a fiber of $\pi$. Therefore we have
\[
\Pic F_k \cong \Z B_k \oplus \Z f.
\]

The intersection matrix is $B_k\cdot B_k = -k$, $B_k \cdot f = 1$, and $f\cdot f = 0$ (see \cite{Har77} Proposition 2.3 and Example 2.11.3).

We define the notation
\[
\mc{O}(m, \ell):= \mc{O}(m B_k + \ell f)\cong \mc{O}_{F_k}(m) \otimes \pi^{*}(\mc{O}_{\PP^1}(\ell)).
\]
for a pair of integers $(m, \ell)$. Every line bundle on $F_k$ is isomorphic to $\mc{O}(m, \ell)$ for some pair of integers $(m,\ell)$.

\subsection{Sections and Cohomology}

Let $\mc{E} = \mc{O} \oplus \mc{O}(-k)$ be a (necessarily split) rank $2$ vector bundle on $\PP^1$. We use the notation $\mc{E} = \mc{O} x \oplus \mc{O}(-k) y$ where $x$ and $y$ are formal, to mean that for sections $a$ of $\mc{O}$ and $b$ of $\mc{O}(-k)$, $ax+by$ denotes the section $(a,b)$ of $\mc{E}$. We can think of $[x:y]$ as the coordinates of the $\PP^1$ fibers on the projective bundle $F_k = \PP(\mc{E})$.

By \Cref{prop: proj-sym}, for any integers $m$ and $\ell $ with $m\ge 0$, we have a canonical isomorphism
\[
H^{0}(F_k, \mc{O}(m, \ell)) \cong H^0(\PP^1, \Sym^{m}(\mc{O} x \oplus \mc{O}(-k) y)\otimes \mc{O}(\ell)).
\]
\[
H^{0}(F_k, \mc{O}(m, \ell)) \cong H^0(\PP^1, \Sym^{m}(\mc{E})\otimes \mc{O}(\ell)).
\]
Again assuming $m\ge 0$, we can write a splitting for the vector bundle:
\begin{align*}
\Sym^{m} \mc{E} \otimes \mc{O}(\ell) &\cong \Sym^{m}(\mc{O} x \oplus \mc{O}(-k) y)\otimes \mc{O}(\ell)\\
&\cong \mc{O}(\ell) x^{m} \oplus \mc{O}(\ell-k) x^{m-1}y \oplus \cdots \oplus \mc{O}(\ell-mk) y^m.
\end{align*}

Therefore, if $m \ge 0$, we may canonically identify $H^0(F_k, \mc{O}(m, \ell))$ with the $\F_q$-vector space of ``bi-homogeneous'' polynomials
\[
s = A_0(t_0,t_1) x^m + A_1(t_0, t_1) x^{m-1}y + \cdots + A_{m}(t_0, t_1) y^{m},
\]
\[
\deg A_i(t_0, t_1) = \ell - ik
\]
in coordinates $[x:y]$ and $[t_0: t_1]$.

It will be useful later to identify when the cohomology of line bundles on Hirzebruch surfaces vanishes.
\begin{lemma}\label{lem:hirz-cohom}
    Let $m, \ell$ be integers, and let $k\ge 0$ be an integer.
    \begin{itemize}
        \item If $m \ge 0$ and $\ell - mk \ge -1$, then $H^1(F_k, \mc{O}(m,\ell)) = 0$.
        \item If $m \ge 0$, then $H^2(\mc{O}(m, \ell))=0$.
    \end{itemize}
\end{lemma}
\begin{proof}
Without loss of generality we can base change to an algebraic closure of $\F_q$. By \cite{Har77} Lemma V.2.4, we have
\[
H^{j}(F_k, \mc{O}(m, \ell)) = H^{j}(\PP^1, \Sym^{m}(\mc{E}) \otimes \mc{O}(\ell))
\]
for all $j \ge 0$. Therefore, if $\ell - mk \ge -1$, we have
\[
H^1(F_k, \mc{O}(m, \ell)) = H^1(\PP^1, \Sym^m(\mc{E})\otimes \mc{O}(\ell)) = \bigoplus_{i=0}^{m} H^1(\PP^1,  \mc{O}(\ell - ik)) = 0,
\]
and $H^2(F_k, \mc{O}(m, \ell)) = H^2(\PP^1(\Sym^{m}(\mc{E}) \otimes \mc{O}(\ell))) = 0$ since the dimension of $\PP^1$ is $1$.
\end{proof}

\subsection{Sections of Vertical Degree 3}
We will primarily be concerned with sections
\[
s \in H^{0}(\PP(\mc{E}),\mc{O}_{\PP(\mc{E})}(3)\otimes \pi^{*}(\wedge^2 \mc{E})^{\vee})
\]
for a rank $2$ vector bundle $\mc{E}$ on $\PP^1$ with associated projective bundle $\pi: \PP(\mc{E}) \to \PP^1$.

If $\mc{E}$ and $\mc{E}'$ are rank $2$ vector bundles with a given isomorphism $\mc{E}' \cong \mc{E} \otimes \mc{L}_1$, then on the isomorphic surfaces $\PP(\mc{E}) \cong \PP(\mc{E}')$ we have $\mc{O}_{\PP(\mc{E'})}(1) \cong \mc{O}_{\PP(\mc{E})}(1)\otimes \pi^{*} \mc{L}_1$, and therefore
\[
H^0(\PP(\mc{E'}), \mc{O}_{\PP(\mc{E'})}(m)) \cong H^0(\PP(\mc{E}), \mc{O}_{\PP(\mc{E})}(m)\otimes \pi^{*} \mc{L}_1^{\otimes m}).
\]

We will define a bit of notation.
\begin{defn}
    If $\ell, k$ are integers such that $k\ge 0$, we denote
    \[
    \mc{V}(\ell, k) := H^{0}(F_k, \mc{O}(3, \ell)) = H^0(F_k, \mc{O}_{F_k}(3) \otimes \pi^{*} \mc{O}(\ell)).
    \]
\end{defn}
Explicitly, a section $s \in \mc{V}(\ell, k)$ is a degree $3$ binary cubic form
\[
s = A_0(t_0, t_1)x^3 + A_1(t_0, t_1) x^2y + A_2(t_0, t_1) xy^2 + A_3 (t_0,t_1) y^3
\]
such that
\begin{align*}
    \deg A_0 &=\ell,\\
    \deg A_1 &= \ell - k,\\
    \deg A_2 &= \ell - 2k,\\
    \deg A_3 &= \ell - 3k.
\end{align*}

We remark that this notation matches with our previously defined notation $\mc{V}(\mc{E})$ if we define 
\[
\mc{E} = \mc{E}_{\ell, k}:= \mc{O}(\ell -k) \oplus \mc{O}(\ell -2k).
\]
Then $\mc{E}_{\ell,k}$ is a twist of $\mc{O} \oplus \mc{O}(-k)$ by $\mc{O}(\ell-k)$, so
\begin{align*}
\mc{V}(\mc{E}_{\ell,k}) &\cong H^{0}(\PP(\mc{E}_{\ell,k}), \mc{O}_{\PP(\mc{E}_{\ell,k})}(3) \otimes \pi^{*}(\wedge^2 \mc{E}_{\ell,k})^{\vee})\\
&\cong H^0(F_k, \mc{O}_{F_k}(3) \otimes \pi^{*} \mc{O}(3\ell - 3k) \otimes \pi^{*}\mc{O}(3k - 2\ell))\\
&\cong H^0(F_k, \mc{O}_{F_k}(3) \otimes\pi^{*} \mc{O}(\ell))\\
&\cong H^0(F_k, \mc{O}(3,\ell))\\
&= \mc{V}(\ell, k).
\end{align*}

\subsection{Horizontal Irreducibility}
 By the results of \Cref{prop:horiz-ir}, a nonzero form 
 \[
 s = A_0 x^3 + A_1x^2 y + A_2 xy^2 + A_3 y^3 \in \mc{V}(\ell, k)
 \]
 is considered to be \textbf{horizontally reducible} if and only if there are integer $b,c$ and a pair of nonzero binary forms
\[
B_0 x^2 + B_1 xy + B_2 y^2 \in H^0(F_k, \mc{O}(2, b))
\]
and
\[
C_0 x + C_1 y \in H^0(F_k, \mc{O}(1, c))
\]
such that $s$ factors as
\[
A_0 x^3 + A_1 x^2y + A_2 xy^2 + A_3 y^3 = (B_0 x^2 + B_1 xy + B_2 y^2) (C_0x + C_1 y).
\]

(The form $s=0$ is considered to be horizontally reducible.)

Indeed, if $s$ is horizontally reducible, then it has an integral component $Y_i$ which is the vanishing scheme of $C_0x+C_1 y \in \mc{O}(1,c)$ for some $c$. Conversely, if $s$ factors as the product of a linear and a quadratic form, then the base change of $s$ to the fraction field $\F_q(t) = \F_q(t_0/t_1)$ has a linear factor, which means $s$ is not horizontally irreducible (again by \Cref{prop:horiz-ir}.

We remark that if $A_3=0$, then $s$ is horizontally reducible, since $s$ has a factor of $x$. We would like to isolate this type of horizontal reducibility in particular. Heuristically, we will see later on that ``most'' horizontally reducible sections $s$ contain a factor of $x$. If $s$ is horizontally reducible but $s$ does not contain a factor of $x$, we think of this behavior as ``special'' and comparatively rare.

(Note that $\deg A_0 \ge \deg A_1 \ge \deg A_2 \ge \deg A_3$, so $A_3$ has a smaller degree than the other coefficients. Heuristically, demanding that $A_3=0$, meaning $s$ has a factor of $x$, places fewer constraints on the binary cubic form $s$ than demanding $s$ has any other linear factor. This is why we break symmetry by singling out the specific factor $x$.)

\begin{defn}
    Let $\ell, k \ge 0$ be integers. A section 
    \[
    s = A_0 x^3 + A_1 x^2 y + A_2 xy^2 + A_3 y^3 \in \mc{V}(\ell, k)
    \]
    is \textbf{$x$-reducible} if it is divisible by $x$ (equivalently if $A_3 = 0$). We say $s$ is \textbf{$x$-irreducible} if it is not $x$-reducible.

    If $s$ is horizontally reducible but not $x$-reducible, we say $s$ is \textbf{specially reducible.}
\end{defn}

\begin{defn}
    For a set $S$ with $S\subseteq \mc{V}(\ell, k) = H^0(F_k, \mc{O}(\ell, k))$, we write
    \begin{itemize}
        \item $S^{\ir}\coloneq \{ s \in S: s\text{ is horizontally irreducible}\}$
        \item $S^{\hr}\coloneq \{ s \in S: s\text{ is nonzero and horizontally reducible}\}$
        \item $S^{\xir}\coloneq\{ s \in S: s\text{ is $x$-irreducible}\}$
        \item $S^{\sr}\coloneq\{ s \in S: s\text{ is specially reducible}\}$.
        \item $S^{\nz}\coloneq\{ s \in S: s\text{ is nonzero}\}$.
        \item $S^{\z}\coloneq\{ s \in S: s\text{ is zero}\}$.
    \end{itemize}
\end{defn}

Note that a section can be forced to be $x$-reducible purely for degree reasons.
\begin{lemma}
    If $\mc{V}(\ell,k)^{\ir}$ is nonzero, then $\ell - 3k \ge 0$.
\end{lemma}
\begin{proof}
    If $s \in \mc{V}(\ell,k)$ and $\ell-3k<0$, then $\deg A_3<0$, so $A_3=0$ and $s$ is automatically $x$-reducible.
\end{proof}
We define the \textbf{degree} of a pair $(\ell,k)$ of nonnegative integers to be the degree $N = \deg \mc{E}_{\ell,k} = 2\ell -3k$ of the corresponding vector bundle. In particular, if $V(\ell, k)^{\ir}$ is nonzero, then $N\ge 0$.

For a fixed degree $N$, we would like to define notation for the set of possible parameter pairs $(\ell,k)$ of degree $N$ such that $\mc{V}(\ell,k)^{\ir}$ could possibly be nonempty.

\begin{defn}
Let $N\ge 0$ be an integer. We denote
\[
\Par(N) := \{(\ell,k)\in \Z_{\ge 0}\times \Z_{\ge 0}: 2\ell-3k=N\text{ and } \ell \ge 3k \ge 0 \}.
\]
\end{defn}
\subsection{Notation}
Here we introduce some more notation. We already defined
\[
\mc{V}(\mc{\ell},k):= H^0(F_k, \mc{O}(3, \ell)) 
\]
as a replacement for the notation $\mc{V}(\mc{E})$.

Let $D$ be a reduced effective divisor on $\PP^1$, and let $\DD \subseteq F_k$ be a $D$-marking. We analogously define
\[
\Van(\mc{\ell}, k, \DD)\subseteq \mc{V}(\mc{\ell},k)= H^0(F_k, \mc{O}(3, \ell)) 
\]
and
\[
\Sing(\mc{\ell}, k, \DD)\subseteq \mc{V}(\mc{\ell},k)= H^0(F_k, \mc{O}(3, \ell)) 
\]
to be the $\F_q$-subspaces of sections that are vanishing and singular at $\DD$, respectively.

\begin{defn}
    Let $\ell, k \ge 0$ be integers. Let $D$ be a reduced effective divisor on $\PP^1$.
    
    For $\alpha \in \{\ir, \hr, \xir, \sr, \nz, \z\}$, we let
    \[
    R^{\alpha}(\ell, k, D):= \sum_{s \in \mc{V}(\ell, k)^{\alpha}} r_D(s).
    \]

    Similarly, let
    \[
    \Phi^{\alpha}(\ell, k, D):= \sum_{s \in \mc{V}(\ell, k)^{\alpha}} a_D(s).
    \]
\end{defn}

All of these definitions align with our previous definitions when we set $\mc{E}_{\ell,k} = \mc{O}(\ell-k) \oplus \mc{O}(\ell-2k)$. Note that every rank $2$ vector bundle $\mc{E}$ with $\deg \mc{E} = N$ can be written as $\mc{E}\cong \mc{E}_{\ell,k}$ for a unique pair of integers $(\ell,k)$ with $2\ell -3k = N$.

\begin{defn}
    Let $k\ge 0$. We define the group $\Gamma_k:= \Aut(\mc{O}_{\PP^1} \oplus \mc{O}_{\PP^1}(-k))$.
\end{defn}
Therefore for any rank $2$ vector bundle $\mc{E} = \mc{E}_{\ell, k}$, we have $\Aut(\mc{E}) \cong \Gamma_k$.

Note that if $k\ge 1$, then $|\Gamma_k| = (q-1)^2q^{k+1}$. If $k=0$, then $|\Gamma_0| = |\GL_2(\Fq)| = (q^2-q)(q^2-1)$.

Finally, for the set $\Mar(\mc{E}, D)$ of $D$-markings of $\PP(\mc{E})$, we will often write $\Mar(D)$ when the projective bundle $F_k \to \PP^1$ is clear.

\subsection{Main Sum For $\PP^1$}

In the case $C = \PP^1$, we will translate and simplify some of our previous notation. We define
\begin{align*}
    \Theta(N)&\coloneq \Theta(\PP^1,N),\\
    \Psi^{\ir}(N)&\coloneq \Psi^{\ir}(\PP^1, N),\\
    G(T)&\coloneq G_{\PP^1}(T),\\
    Z(T)&\coloneq Z_{\PP^1}(T).
\end{align*}
The main term we are after can now be written as
\[
\Psi^{\ir}(N) = \sum_{\substack{N', d' \ge 0 \\ N' + d = N}} \sum_{(\ell, k) \in \Par(N')} \frac{1}{|\Gamma_k|}\sum_{\deg D = d} \mu(D) \Phi^{\ir}(\ell, k , D).
\]

Also, it is well known that the formal zeta function\footnote{$Z(T)$ is the unique rational function such that 
\[Z(q^{-s}) = \zeta_{\PP^1}(s) = \frac{1}{(1-q^{-s})(1-q^{1-s})}.\]} $Z(T) = Z_{\PP^1}(T)$ satisfies 
\[
Z(T) = \prod_{P \in \PP^1} (1- T^{\deg P})^{-1} = \frac{1}{(1-T)(1-qT)}.
\]

(Indeed, we can write 
\[
Z(T) = \sum_{d \ge 0} \sum_{\deg D= d }T^{\deg d} = \sum_{d \ge 0} \left(\frac{q^{d+1}-1}{q-1}\right)T^{d}  = \frac{1}{(1-T)(1-qT)}
\]
where $D$ ranges over the $(q^{d+1}-1)/(q-1)$ reduced effective divisors of degree $d$ on $\PP^1$.)

Thus once we estimate $\Psi^{\ir}(N)$, we can easily solve for $\Theta(N)$ using the equality
\[
G(T) = \sum_{N \ge 0} \Theta(N)T^{N} =(1-T)(1-qT)\left(\sum_{N \ge 0} \Psi^{\ir}(N) T^{N}\right)
\]
from \Cref{defn:power-series-eq} and \Cref{prop:Psi-formula}.

\subsection{Strategy for the Rest of the Thesis}

The main task for the rest of the thesis is to estimate $\Psi^{\ir}(N)$. This is a weighted sum of terms $\Phi^{\ir}(\ell, k, D)$ for various triples $(\ell, k, D)$. The terms in the sum satisfy the conditions $2\ell - 3k + \deg D = N$, $\deg D \ge 0$, and $\ell \ge 3k\ge 0$.

We will divide up the collection of possible triples $(\ell, k,D)$ satisfying these three conditions into a ``small range'', ``medium range'', and ``large range'' depending on the size of $\deg D$ as a function of $\ell$ and $k$. For triples in each range, we will estimate $\Phi^{\ir}(\ell, k, D)$ and the related root-counting function $R^{\ir}(\ell, k, D)$.

We'll find a ``model'' function $\widehat{\Phi}^{\ir}$ such that
\[
\Phi^{\ir}(\ell, k, D) = \widehat{\Phi}^{\ir}(\ell, k, D) + O(\Error_1)
\]
where the error depends on the range of $(\ell, k, D)$.

Adding up the errors will yields a model term and error term for 
\[
\Psi^{\ir}(N) = \widehat{\Psi}^{\ir}(N) + O(\Error_2)
\]
which will easily give
\[
\Theta(N) = \widehat{\Theta}(N) + O(\Error_3)
\]
where $\widehat{\Theta}(N)$ will be defined in terms of $\widehat{\Psi}^{\ir}(N)$.

 We can then package the values $\widehat{\Theta}(N)$ into a generating function, which will give a main term, a secondary term, and an error term for $\Theta(N)$.

 Finally, we account for the error coming from $C_3$ and inseparable covers, which is the only discrepancy between $\Theta(N)$ and $\Cov_{3}(2N)$.

 After all of these steps, we will have a main term, a secondary term, and an error term for $\Cov_3(2N)=\Cov_{3}(\PP^1, 2N)$, the number of triple covers of $\PP^1$ of branch degree $2N$ up to isomorphism.
\section{Restrictions of Line Bundles}
\label{sec:restrictions-line-bundles}
Let $D$ be a reduced effective divisor on $\PP^1$, and let $\DD \subseteq F_k$ be a $D$-marking. Consider a line bundle $\LL = \mc{O}(3, \ell)$ on $\Fk$.

We aim to count sections $s \in H^{0}(\Fk, \LL)$ passing through each point of $\DD$. In other words, we wish to understand the kernel of the restriction morphism
\[
H^{0}(\Fk, \LL) \to H^{0}(D, \LL|_{D}) \cong H^{0}(D, \mc{O}_D).
\]
Our strategy is to find a curve $C_1\subseteq \Fk$ passing through $\DD$ and to analyze the composition 
\[
H^{0}(\Fk, \LL) \to H^{0}(C_1, \LL|_{C_1}) \to H^{0}(D, \mc{O}_D).
\]
In particular, it will be beneficial to consider curves $C_1 \subseteq \Fk$ containing $\DD$ such that $\mc{O}(C_1) \cong \mc{O}(1,h)$ for some nonnegative integer $h$.

\begin{rem}
    For us, a curve is reduced by definition, but not necessarily irreducible.
\end{rem}

Any such curve $C_1$ with $\mc{O}(C_1) \cong \mc{O}(1,h)$ can be decomposed into irreducible components
\begin{equation}
C_1 = C_0 \cup \bigcup_{i} f_i \label{eq:c1-decomp}
\end{equation}
where the $f_i$ are fibers above closed points in $\PP^1$ and $C_0$ is the unique horizontal irreducible component of $C_1$ (meaning $\mc{O}(C_0) \cong \mc{O}(1,h_0)$ for some $0\le h_0 \le h$).

The closed embedding $C_1 \hookrightarrow \Fk$ defines an exact sequence
\[
0 \to \mc{O}_{\Fk}(-C_1) \to \mc{O}_{\Fk} \to i_{*} \mc{O}_{C_1}\to 0
\]
of sheaves on $\Fk$. By tensoring with $\LL$, we obtain the long exact sequence
\[
\begin{tikzcd}[arrows=to]
0 \rar & H^0(\Fk, \LL(-C_1)) \rar & H^0(\Fk, \LL)\arrow[d, phantom, ""{coordinate, name=Z}] \rar & H^0(C_1, \LL|_{C_1})
\arrow[dll,rounded corners,to path={ -- ([xshift=2ex]\tikztostart.east)
|- (Z) [near end]\tikztonodes
-| ([xshift=-2ex]\tikztotarget.west)
-- (\tikztotarget)}] \\
 & H^1(\Fk, \LL(-C_1)) \rar & H^1(\Fk, \LL) \rar & H^1(C_1, \LL|_{C_1}).
\end{tikzcd}
\]

So we look for $C_1$ passing through $\DD$ such that $H^1(F_k, \mc{L}(-C_1))=0$, because in this case the restriction map $H^0(F_k, \mc{L}) \to H^0(C_1, \mc{L}|_{C_1})$ is surjective.

\begin{defn}
    We say $C_1\subseteq \Fk$ is a \textbf{minimal $(\DD,h)$-curve} if the following conditions hold:
    \begin{itemize}
    \item $C_1$ contains $\DD$,
    \item $\mc{O}(C_1) \cong \mc{O}(1,h)$,
    \item If we decompose $C_1$ into irreducible components
    \[
    C_1 =  C_0\cup \bigcup_{i \in I} f_i
    \]
    as in \eqref{eq:c1-decomp}, then every $f_i$ contains a point $\mc{P}_i$ of $\DD$, and furthermore this point is disjoint from the intersection $f_i \cap C_0$.
    
    \end{itemize}
\end{defn}
A minimal $(\DD,h)$-curve is minimal in the sense that any sub-curve given by removing some fibral components cannot still contain $\DD$.
\begin{prop}
    \label{prop:h-dichotomy}
    Let $D$ be a reduced effective divisor on $\PP^1$. Let $\DD\subseteq \Fk$ be a $D$-marking. Let $h\ge 0$ be a nonnegative integer. Let $C_1 \subseteq \Fk$ be a minimal $(\DD, h)$-curve, and let $C_0$ be its horizontal component (as in \eqref{eq:c1-decomp}). Let $\mc{D}_0$ be the subscheme of $\mc{D}$ on $C_0$. Let $\ell$ be a nonnegative integer, and set $\LL = \mc{O}(3, \ell)$.

    Now assume
    \[
    \ell-2k \ge h.
    \]

    Then either $\deg \mc{L}|_{C_0} \ge \deg \mc{D}_0$, or $\deg \mc{L}|_{C_0}< \deg \mc{D}_0$. 
    
    In the first case, the map $H^0(\Fk, \LL) \to H^0(\DD, \mc{O}_{\DD})$ is surjective.

    In the second case, if a section $s \in H^0(\Fk, \LL)$ satisfies $s|_{\DD} = 0$, then $s|_{C_0}=0$.

\end{prop}
Heuristically, this proposition says that the size of the linear space of divisors cut out by sections in $\mc{O}(3,\ell)$ passing through $\DD$ is exactly ``as large as expected'' unless a numerical condition forces the space to have a fixed component.

\begin{lemma}\label{lem:h1-or-h0}
    Let $C_1\subseteq \Fk$ be a (reduced) curve such that $\mc{O}(C_1)\cong \mc{O}(1,h)$. Consider the decomposition into irreducible components
    \[
    C_1 = C_0 \cup \bigcup_{i} f_i
    \]
    as in \eqref{eq:c1-decomp}, where $C_0$ is the horizontal component and the $f_i$ are fibers. Furthermore let $\mc{M}$ be a line bundle on $C_1$, and assume $\mc{M}|_{f_i} \cong \mc{O}(a_i)$ with $a_i \ge 0$ for all $i$.

    Then either $\deg \mc{M}|_{C_0} \ge 0$ or $\deg \mc{M}|_{C_0}<0$. 
    \begin{itemize}
        \item If $\deg \mc{M}|_{C_0}\ge 0$, then we have \[H^1(C_1,\mc{M})=0.\]
        \item If $\deg \mc{M}|_{C_0}<0$, then we have 
        \[
        H^0(C_0, \mc{M}|_{C_0})=0.
        \]
    \end{itemize}
\end{lemma}
\begin{proof}
We have a short exact sequence of sheaves
\[
0 \to \mc{O}_{C_1} \to \mc{O}_{C_0} \oplus \bigoplus \mc{O}_{f_i} \to \mc{O}_{Z}\to 0
\]
where $Z\coloneq  \left(\cup f_i\right)\cap C_0$ is the scheme theoretic intersection (see \cite[\href{https://stacks.math.columbia.edu/tag/0B7M}{Tag 0B7M}]{stacks-project}). Note that each $C_0$ and $f_i$ intersect transversely, meaning $Z$ is a union of closed points. (We can see this to be true because $C_0$ is a section of $\pi:F_k \to \PP^1$, and because each $f_i$ is reduced.)

After twisting by $\mc{M}$ and taking cohomology, we recover
\[
0 \to H^{0}(C_1, \mc{M}) \to H^{0}(C_0, \mc{M}|_{C_0}) \oplus \bigoplus H^0(f_i, \mc{M}|_{f_i}) \to H^{0}(Z, \mc{M}|_{Z}).
\]
Now note that each $\mc{M}|_{f_{i}} \cong \mc{O}(a_i)$ by assumption. Therefore,
\[
\bigoplus H^{0}(C_0, \mc{M}|_{f_{i}}) \to H^{0}(Z, \mc{M}|_{Z})
\]
is a surjection, since each $a_i \ge 0$ and $Z$ has exactly one closed point in each $f_i$.

Therefore we have an exact sequence
\[
0 \to H^{1}(C_1, \mc{M}) \to H^1(C_0, \mc{M}|_{C_0}) \to 0
\]
noting that $H^{1}(f_i, \mc{O}(a_i)) = 0$ for $a_i \ge 0$ and $H^1(Z, \mc{M}|_Z)=H^1(Z, \mc{O}_Z)=0$ as $Z$ is a union of points.

Finally, note that $C_0 \cong \PP^1$ and so the line bundle $\mc{M}|_{C_{0}}$ satisfies 
\[
\mc{M}|_{C_{0}} \cong \mc{O}(m)
\]
for some integer $m$. Therefore we either have $H^{0}(C_0, \mc{O}(m)) = 0$ if $m \le -1$ or $H^1(C_0, \mc{O}(m))=0$ if $m\ge -1$, which implies the desired result.
\end{proof}
Now we can prove \Cref{prop:h-dichotomy}.
\begin{proof}[Proof of \Cref{prop:h-dichotomy}]

    Note that the subscheme $\DD\subseteq C_1$ is an effective Cartier divisor, since by assumption each point $\mc{P}\in \DD$ is disjoint from the singular set in $C_1$.
    
    Because $\ell - 2k \ge h$, the line bundle $\mc{L}(-C_1) \cong \mc{O}(2, \ell -h)$ satisfies $H^1(F_k, \mc{L}(-C_1)) = 0$, by \Cref{lem:hirz-cohom}. So
    \begin{equation}\label{eq:first-sur}
    H^0(F_k, \mc{L}) \twoheadrightarrow H^0(C_1, \mc{L}|_{C_1})
    \end{equation}
    is a surjection. Furthermore the restriction morphism $H^0(F_k, \mc{L}) \to H^0(\DD, \mc{O}_{\DD})$ factors through this surjection.

    Now define the line bundle $\mc{M}:= \mc{L}|_{C_1}\otimes\mc{O}_{C_1}(-\mc\DD)$. Note that $\mc{M}|_{C_0} = \mc{L}|_{C_0}(-\DD_0)$, where $\DD_0$ was defined to be the subscheme of $\DD$ lying on $C_0$. In particular we have $\deg \mc{M}|_{C_0} = \deg \mc{L}|_{C_0} - \deg \DD_0$.

    Now if $\deg \mc{L}|_{C_0} - \deg \DD_0 \ge 0$, then $\deg \mc{M}|_{C_0} \ge 0$, so by \Cref{lem:h1-or-h0} we have $H^1(C_1, \mc{M})=0$. So in this case we have a surjection
    \begin{equation}\label{eq:second-sur}
    H^0(C_1, \mc{L}|_{C_1}) \twoheadrightarrow H^0(\DD, \mc{O}_{\DD})
    \end{equation}
    and therefore the composite map $H^0(F_k, \mc{L}) \to H^0(F_k, \mc{L}|_{C_1}) \to H^0(\DD, \mc{O}_{\DD})$ is a surjection by \eqref{eq:first-sur} and \eqref{eq:second-sur}.

    On the other hand, if $\deg \mc{L}|_{C_0} - \deg \DD_0<0$, then $\deg \mc{M}|_{C_0}<0$. So by \Cref{lem:h1-or-h0}, we have $H^0(C_0, \mc{M}|_{C_0}) = H^0(C_0, \mc{L}|_{C_0}(-\DD_0))=0$. In particular, if a section $s \in H^0(F_k, \mc{L})$ satisfies $s|_{\DD}=0$, then $s|_{C_0} \in H^0(C_0, \mc{L}|_{C_0})$ vanishes at $\DD_0$, so $s|_{C_0}=0$ as desired.
\end{proof}

In order to make use of \Cref{prop:h-dichotomy}, we need an effective bound on the parameter $h$ controlling the numerical class of a curve passing through $\DD$.
\begin{rem}
    This is the same $h$ as the ``height'' function used by Zhao \cite{Zha13} (Section 5.2).
\end{rem}

\begin{lemma}\label{lem:h-section}
    Let $D$ be a reduced effective divisor on $\PP^1$. Let $\DD\subseteq F_k$ be a $D$-marking.
    
    Then there is a nonnegative integer $h$ with 
    \[
    h\le \frac{\deg D + k}{2}
    \]
    such that there exists a nonzero section $s \in H^{0}(\Fk, \mc{O}(1,h))$ vanishing at $\DD$.
\end{lemma}

\begin{proof}
    First, we consider the case when $\deg D\ge k$. Let $h = \lfloor (\deg D + k)/2 \rfloor$. Then $h\ge k$ and $2h \ge \deg D+k -1$, implying
    \[
    \dim H^{0}(\Fk, \mc{O}(1,h)) = (h+1)+(h-k+1)> \dim H^{0}(\DD, \mc{O}_{\DD})
    \]
    so there is a nonzero element $s \in H^{0}(\Fk, \mc{O}(1,h))$ vanishing at $\DD$.

    Now consider the case when $\deg D<k$. Consider the union of the negative curve $B_k \subseteq F_k$ (cut out by $x=0$) and the fibers $\pi^{-1}(D)$. This is in the divisor class $B_k+(\deg D) f $, and so is cut out by a section $s \in H^{0}(F_k, \mc{O}(1, \deg D))$. Then we can set $h = \deg D< \frac{ \deg D + k}{2}$.
\end{proof}

By removing vertical components from the section constructed in \Cref{lem:h-section}, we can in fact find a (reduced) minimal $(\DD,h)$-curve for some (possibly smaller) value of $h$.
\begin{lemma}\label{lem:h-curve}
    Let $D$ be a reduced effective divisor on $\PP^1$. Let $\DD \subseteq F_k$ be a $D$-marking. Then there is a nonnegative integer $h$ with
    \[
    h\le \frac{\deg D + k}{2}
    \]
    such that there exists a minimal $(\DD,h)$-curve $C_1 \subseteq F_k$.
\end{lemma}
\begin{proof}
    By \Cref{lem:h-section}, there exists a nonnegative integer $h_1\le (\deg D + k)/2$ and a nonzero section $s \in H^0(F_k, \mc{O}(1,h_1))$ vanishing at $\DD$. The corresponding Weil divisor $\sum a_i [Y_i]$ has one horizontal component $[C_0]$ and several vertical components $a_i[f_i]$ with multiplicity. We define a subdivisor
    \[
    [C_1] = [C_0] +\sum a_i' [f_i]
    \]
    with $a_i'\le a_i$, where the $a_i'$ are defined as follows. If there is a point $\mc{P} \in \mc{D}$ such that $\mc{P} \in f_i$ but $\mc{P} \neq f_i \cap C_0$, then let $a_i'=1$. Otherwise, let $a_i'=0$.

    Then since $a_i' \le 1$ for all $i$, we know that $C_1$ is a reduced curve. Furthermore, $\mc{O}(C_1)  = \mc{O}(1, h)$ for some integer $h \le h_1\le (\deg D_1 + k)/2$, so $C_1$ is a minimal $(\DD,h)$-curve by construction.
\end{proof}

Putting together the pieces above, we are able to conclude the following.
\begin{prop}
\label{prop:dichotomy}
    Let $D\subseteq \PP^1$ be a reduced effective divisor. Let $\DD \subseteq \Fk$ be a $D$-marking. Let $\ell\ge 0$ be an integer and set $\LL = \mc{O}(3,\ell)$ on $\Fk$. Consider the restriction map
    \[
    \alpha: H^{0}(\Fk, \LL) \to H^{0}(\DD, \mc{O}_{\DD}).
    \]
    If we assume the condition
    \[
    \ell-2k \ge \frac{\deg D + k}{2},
    \]
    then at least one of the following conditions holds:
    \begin{itemize}
    \item $\alpha$ is surjective, or
    \item any section $s \in \ker(\alpha)$ is horizontally reducible.
    \end{itemize}
\end{prop}
\begin{proof}
    By \Cref{lem:h-curve}, there exists a nonnegative integer $h\le (\deg D +k )/2$ and a minimal $(\DD, h)$-curve $C_1 \subseteq F_k$. In particular, this curve has a decomposition
    \[
    C_1 = C_0 \bigcup_{i} f_i
    \]
    into irreducible components, where $C_0$ is horizontal and the $f_i$ are fibers above closed points of the base $C = \PP^1$.

    By the assumption $\ell - 2k \ge (\deg D + k)/2$, so in particular we have
    \[
    \ell - 2k \ge \frac{\deg D + k}{2} \ge h.
    \]

    So the hypotheses of \Cref{prop:h-dichotomy} hold. In particular, either $\alpha: H^0(F_k, \mc{L}) \to H^{0}(\DD, \mc{O}_{\DD})$ is surjective, or $\ker(\alpha)$ only contains sections which vanish on $C_0$. In the latter case, every section $s$ in $\ker(\alpha)$ must be horizontally reducible in particular by the second criterion of \Cref{prop:horiz-ir}.
\end{proof}
\subsection{Remarks on Reducibility and Unexpected Curves}
The discussion in this section is not related to or necessary for any subsequent results in this thesis.
.

    The result in \Cref{prop:dichotomy} answers a specific case of a general question of interest in algebraic geometry. 
    
    \begin{question}
        Let $X$ be a smooth, rational\footnote{By rational, we mean birationally equivalent to $\PP^2$.} projective surface. Let $\mc{L}$ be a line bundle on $X$, and let $\DD\subseteq X$ be a $0$-dimensional reduced subscheme (that is, a finite union of closed points). Let $V \subseteq H^0(X, \mc{L})$ be the linear subsystem of divisors vanishing at $\DD$. We say $V$ satisfies property P if either 
        \begin{itemize}
            \item $V$ has the ``expected'' dimension $\dim V = \max(0, \dim H^0(X, \mc{L}) - \deg \DD)$, or 
            \item $V$ has a fixed component.
        \end{itemize} 
        
        What numerical conditions on $X, \mc{L}$, and $\DD$ are required to guarantee that $V$ satisfies property P?
    \end{question}

    Harbourne \cite{Har97} (Lemma 2.7) gives the following condition over an algebraically closed field. Let $f:Y\to X$ be the blowup of $X$ at $\DD$, and let $\mc{M} = f^{*}\mc{L}(- E)$, where $E$ is the exceptional divisor. If $\mc{M}$ is effective, $(-K_Y)\cdot \mc{M}>0$, and $H^1(Y, \mc{M}) \neq 0$, then $\mc{M}$ has a fixed component.

    (Note that the global sections of $\mc{L}$ vanishing at $\DD$ can be canonically identified with the global sections of $\mc{M}$.)
    
    It is possible that one could adapt this argument in the case $X = F_k$ to improve the numerical hypothesis $\ell - 2k \ge (\deg D +k)/2$ in \Cref{prop:dichotomy}. However, Harbourne's result does not construct a specific fixed component of $\mc{M}$ in $Y$, and we need our fixed component to be horizontal for our sieve arguments. More work would be necessary to eliminate the possibility of fixed vertical components (meaning components contained in fibers of the composite map $Y \to F_k \to \PP^1$).

    We note that \Cref{prop:dichotomy} will be later used in the ``medium'' range of a sieve argument, which will be described in \Cref{sec:root-counting}. We think of this as the critical range, which allows us to reduce the error term small enough in order to reveal a secondary term.
    
    In general, it would be interesting to further understand the connection between secondary terms for arithmetic statistics over function fields, and unexpected divisors related to fixed components of linear systems.

\section{Estimates for Root Counting}\label{sec:root-counting}
In this section, we will estimate the root-counting function
\[
R^{\ir}(\ell,k,D) = \sum_{\DD \in \Mar(D)}|\Van(\ell,k, \DD)^{\ir}|
\]
in certain ranges as a function of nonnegative integers $\ell,k,$ and $\deg D$. We will also estimate the adjusted root counting function in these ranges, which we recall is defined by
\[
\Phi^{ir}(\ell,k,D) = \sum_{D_1\mid D} \mu(D-D_1)R^{\ir}(\ell,k,D_1).
\]

For fixed nonnegative integer values of $t$ and $k$, we will divide the divisors $D$ into three ranges depending on the degree $\deg D$.

\begin{itemize}
\item If $\ell\ge \deg D$, then we will say that $(\ell, k ,D)$ is in the \textbf{small range}.
\item If $\ell<\deg D$ and $\ell-2k\ge \frac{\deg D+k}{2}$, then we will say that $(\ell, k, D)$ is in the \textbf{medium range}.
\item If $\ell<\deg D$ and $\ell-2k < \frac{\deg D+k}{2}$, then we will say that $(\ell, k, D)$ is in the \textbf{large range.}
\end{itemize}

Note that for fixed values of $\ell$ and $k$, the medium range could become empty.

\subsection{Estimates on Reducibility}

\begin{defn}
Let $\ell,k,c\ge 0$ be integers and let $\sigma \in H^{0}(\Fk, \mc{O}(1,c))$ be a nonzero section. We define the notation $\Img(\sigma)$ by
\[
\Img(\sigma)\coloneq  \Img
\left(H^{0}(\Fk, \mc{O}(2, \ell-c))  \xrightarrow{\times \sigma} H^{0}(\Fk, \mc{O}(3,\ell))\right),
\]
that is, the image of the multiplication map by $\sigma$.
\end{defn}
\begin{defn}
Let $c \ge 0$. Let $s \in H^{0}(\Fk, \mc{O}(3,\ell))$ be a section and let $\sigma \in H^{0}(\Fk, \mc{O}(1,c))$ be a nonzero irreducible section. We say $s$ is
\begin{itemize}
    \item \textbf{$\sigma$-reducible} if $s \in \Img(\sigma)$,
    \item \textbf{$c$-reducible} if $s$ is $\sigma$-reducible for some nonzero irreducible section $\sigma \in H^{0}(\Fk, \mc{O}(1,c))$.
\end{itemize}
\end{defn}

More concretely, a section $s \in H^{0}(\Fk, \mc{O}(3,\ell))$ is $c$-reducible if it can be written as
\[
s = A_0x^3 + A_1x^2y + A_2 xy^2 + A_3 y^3 = (B_0 x^2 + B_1 xy + B_2 y^2)(C_0x + C_1y)
\]
with $c = \deg C_0$ and $C_0, C_1$ having no nontrivial common divisors.

Remark that if some $s \in H^{0}(\Fk, \mc{O}(3,\ell))$ is specially reducible, then in particular $s$ is $c$-reducible for some $c \ge k$. This is clear if $k=0$. If $k\ge 1$, note that if 
\[
\sigma = C_0X + C_1 Y \in H^{0}(\Fk, \mc{O}(1,c))
\]
satisfies $C_1 \neq 0$, then we must have $c \ge k$.
\begin{prop}
\label{prop:sigma-root-bound}
Let $\ell,k,c \ge 0$ be integers and let $\sigma \in H^{0}(\Fk, \mc{O}(1,c))$ be a nonzero irreducible section. Assume $\ell-3k \ge 0$ and $\ell-2k \ge c$. Then
\[
\sum_{s \in \Img(\sigma)^{\nz}} r_D(s) = O\left(5^{\omega(D)}q^{3\ell-3k-3c}\right)
\]
and
\[
\sum_{s \in \Img(\sigma)^{\nz}} |a_D(s)| = O\left(3^{\omega(D)} q^{3\ell-3k-3c}\right).
\]
\end{prop}
In order to prove \Cref{prop:sigma-root-bound}, we start by proving the following helpful lemma.
\begin{lemma}
\label{lemma:fib-bound}
Let $\ell_1,a,k$ be nonnegative integers. Let $D\subseteq \PP^1$ be a reduced effective divisor, and let $\LL = \mc{O}(a,\ell_1)$ on $\Fk$. Then the number of nonzero sections $s \in H^{0}(\Fk, \LL)$ vanishing on the fibers above $D$ is at most
\[
|H^{0}(\Fk, \LL)^{\nz}| q^{-\deg D}.
\]
\end{lemma}
\begin{proof}
Let $\pi: \Fk \to \PP^1$ be the Hirzebruch surface.

There is an exact sequence
\[
0 \to H^{0}(\Fk, \LL\otimes \pi^{*} \mc{O}(-D)) \to H^{0}(\Fk, \LL) \to H^{0}(\Fk, \LL\otimes \pi^{*} \mc{O}_{D})
\] which identifies $H^{0}(\Fk, \LL\otimes \pi^{*}\mc{O}(-D))$ with the vector space of sections $s \in H^{0}(\Fk, \LL)$ vanishing on the fibers above $D$. Therefore we aim to count the the number of nonzero sections $|H^{0}(\Fk, \LL\otimes \pi^{*}\mc{O}(-D))^{\nz}|$.

Let $d = \deg D$, so $\mc{O}(-D) \cong \mc{O}(-d)$ as sheaves on $\PP^1$. Define the function
\[
V(m)\coloneq  [m+1] + [m-k+1]+\cdots + [m-ak+1]
\]
where $[n] = \max(n,0)$ for an integer $n$. Then
\begin{align*}
\dim H^{0}(\Fk, \LL\otimes \pi^{*} \mc{O}(-d)) & = \dim H^{0}(\Fk, \mc{O}(a,\ell_1-d))=V(\ell_1-d).
\end{align*}
Now note that if we fix values of $\ell_1,k$ and vary $d$, the function $d+V(\ell_1-d)$ is non-increasing function in $d$ for $0\le d\le \ell_1$. Therefore, if $0\le d\le \ell_1$, we have $d+V(\ell_1-d) \le V(\ell_1)$ and subsequently
\[
q^{d}\left(q^{V(\ell_1-d)}-1\right) \le q^{d} \left(q^{V(\ell_1-d)}\right)-1 \le q^{V(\ell_1)}-1.
\]
Thus we have
\[
|H^{0}(\Fk, \LL\otimes \pi^{*}\mc{O}(-d))^{\nz}| = \left(q^{V(\ell_1-d)}-1\right)\le \left(q^{V(\ell_1)}-1\right)q^{-d}= |H^{0}(\Fk, \LL)^{\nz}| q^{-d}
\]
under the assumption $0\le d\le \ell_1$. On the other hand, if $d>\ell_1$, then the above inequality holds trivially since the left hand side is $0$. So we are done.
\end{proof}

Now we can prove \Cref{prop:sigma-root-bound}, bounding the number of $\sigma$-reducible sections.
\begin{proof}[Proof of \Cref{prop:sigma-root-bound}]
Note that 
\[
r_D(s) = \prod_{P\mid D} r_P(s)
\]
and furthermore for each $P$ we have $r_P(s) \le 3$ unless $s|_{P}$ vanishes. Let $\Fi(s,D)$ denote the maximal subdivisor $D_1\subseteq D$ where $s$ vanishes on the fibers above $D_1$. Then if $\Fi(s,D)=D_1$, we have
\[
r_D(s) \le  \prod_{P\mid D_1} (q^{\deg P}+1)\prod_{\substack{P\mid D\\ P\nmid D_1}} 3.
\]
Now let $\LL = \mc{O}(2, \ell-c)$. Then by \Cref{lemma:fib-bound} there are at most $q^{-D_1}|H^{0}(\Fk, \LL)|^{\nz}$ nonzero sections $s \in H^{0}(\Fk, \mc{O}(3,\ell))$ in the image of $H^{0}(\Fk, \LL)$ for which $\Fi(s,D) = D_1$. So we have

\begin{align*}
\sum_{\substack{s \in \Img(\sigma)\\ s \neq 0}}r_D(s) &= \sum_{D_1\subseteq D}\sum_{\substack{s \in \Img(\sigma)\\\Fi(s,D) = D_1\\s\neq 0}} r_D(s)\\
&\le \sum_{D_1\subseteq D} \left(q^{-D_1} |H^{0}(\Fk, \LL)^{\nz}|\right)\left(\prod_{P\mid D_1} (q^{\deg P}+1)\prod_{\substack{P\mid D\\ P\nmid D_1}} 3\right).\\
&\le |H^{0}(\Fk, \LL)^{\nz}|\sum_{D_1 \subseteq D}\left(\prod_{P \mid D_1} 2 \prod_{\substack{P\mid D\\ P\nmid D_1}} 3\right)\\
&= 5^{\omega(D)}|H^{0}(\Fk, \LL)^{\nz}| = O\left(5^{\omega(D)} q^{3\ell-3k-3c}\right).
\end{align*}

Similarly, we have $a_D(s) = \prod_{P \mid D} a_P(s)$ and $|a_P(s)|\le 2$ for each $P$ unless $s|_{P}$ vanishes, in which case $a_P(s) = q^{\deg P}$. So if $\Fi(s,D) = D_1$, we have
\[
|a_D(s)| \le \prod_{P \mid D_1}q^{\deg P} \prod_{\substack{P \mid D\\ P\nmid D_1}} 2 = q^{\deg D_1} 2^{\omega(D) -\omega(D_1)}
\]
So setting $\LL = \mc{O}(2,\ell-c)$, we have
\begin{align*}
\sum_{s \in \Img(\sigma)^{\nz}} |a_D(s)|&= \sum_{D_1 \subseteq D} \sum_{\substack{s \in \Img(\sigma)^{\nz}\\\Fi(s,D) = D_1}} |a_D(s)|\\
&\le \sum_{D_1 \subseteq D} \left(q^{-D_1} |H^{0}(\Fk, \LL)^{\nz}|\right)\left(q^{D_1}2^{\omega(D)-\omega(D_1)}\right)\\
&= 3^{\omega(D)}|H^{0}(\Fk, \LL)^{\nz}|= O\left(3^{\omega(D)}q^{3\ell-3k-3c}\right).
\end{align*}
\end{proof}

As a corollary, we get:

\begin{prop}
\label{prop:hr-er-z-bound}
Let $\ell,k \ge 0$ be integers such that $\ell\ge 3k$, and let $D$ be a reduced effective divisor on $\PP^1$. Then we have
\begin{align*}
R^{\hr}(\ell,k,D) &= O\left(5^{\omega(D)}q^{3\ell-3k}\right),\\
R^{\sr}(\ell,k,D) &= O\left(5^{\omega(D)} q^{3\ell-5k}\right),\\
R^{\z}(\ell,k,D) &= O\left(2^{\omega(D)}q^{\deg D}\right).
\end{align*}

Similarly, 
\begin{align*}
|\Phi^{\hr}(\ell,k,D)| &= O\left(3^{\omega(D)} q^{3\ell-3k}\right),\\
|\Phi^{\sr}(\ell,k,D)| &= O\left(3^{\omega(D)} q^{3\ell-5k}\right),\\
|\Phi^{\z}(\ell,k,D)| &= O\left(q^{\deg D}\right).
\end{align*}

\end{prop}
\begin{proof}
To count $R^{\hr}(\ell,k,D)$, note that any horizonally reducible section $s \in H^{0}(\Fk, \mc{O}(3,\ell))$ is either $x$-reducible or specially reducible.

The $x$-reducible forms contribute a count of
\[
\sum_{s \in \Img(x)^{\nz}} r_D(s) = O\left(5^{\omega(D)} q^{3\ell-3k}\right)
\] by \Cref{prop:sigma-root-bound}.

Now suppose $s \in \Img(\sigma)^{\nz}$ is specially reducible, where $\sigma \in H^{0}(\Fk, \mc{O}(1,c))$ and $c\ge k$. If $c>\ell-2k$, then we can write
\[
s = (B_0x^2 + B_1xy + B_2y^2)(C_0x+C_1y)
\]
where $\deg B_2 = \ell-2k-c<0$, so $B_2=0$. Therefore $s$ is $x$-reducible in this case, a contradiction.

So $\ell-2k\ge c$ in this case. Therefore by \Cref{prop:sigma-root-bound} we have
\begin{align*}
R^{\hr}(\ell,k,D)&\ll 5^{\omega(D)}q^{3\ell-3k} + \sum_{\substack{c \ge k\\c\le \ell-2k}}\sum_{\substack{\sigma \in H^{0}(\Fk, \mc{O}(1,c))^{\nz}\\\sigma\text{ irreducible}}} \sum_{s \in \Img(\sigma)^{\nz}} r_D(s)\\
&\le 5^{\omega(D)}q^{3\ell-3k} +\sum_{\substack{c\ge k\\c\le \ell-2k}} 5^{\omega(D)}q^{2c-k} q^{3\ell-3k-3c}\\
&\ll 5^{\omega(D)}q^{3\ell-3k}.
\end{align*}

By omitting the first term and by a similar argument, we obtain
\[
R^{\sr}(\ell,k,D) \ll \sum_{c \ge k} 5^{\omega(D)} q^{2c-k} q^{3\ell-3k-3c}\ll 5^{\omega(D)} q^{3\ell-5k}.
\]

Next,
\[
R^{z}(\ell,k,D) = \prod_{P\mid D} (q^{\deg P}+1)\le 2^{\omega(D)} q^{\deg D}.
\]

We obtain 
\[
|\Phi^{\hr}(\ell,k,D)|\ll 3^{\omega(D)}q^{3\ell-3k}
\]
and
\[
|\Phi^{\sr}(\ell,k,D)|\ll 3^{\omega(D)} q^{3\ell-5k}
\]
by nearly identical arguments.

Finally, we can exactly compute
\[
|\Phi^{z}(\ell,k,D)| = |\prod_{P\mid D} a_P(0)| = q^{\deg D}.
\]

\end{proof}

\subsection{Small Range}
Assume we have nonnegative integers $\ell\ge 3k \ge 0$ and a reduced effective divisor $D\subseteq \PP^1$ with $\deg D = d$. Recall that we consider $(\ell, k, D)$ to be in the \textbf{small range} if $\ell \ge \deg D$. In this range, we have an exact formula for $R^{\xir}(\ell,k,D)$.

\begin{lemma}
Let $\ell,k \ge0$ be integers and let $D$ be a reduced effective divisor on $\PP^1$ such that $\ell\ge 3k$ and $\ell\ge \deg D$. Then we have
\[
R^{\xir}(\ell,k,D) = \sum_{D_0 \subseteq D}\begin{cases}
q^{4\ell-6k+4 - \deg D_0} - q^{3\ell-3k+3}&\text{if } \ell-3k\ge \deg D_0,\\
0 & \text{else.}

\end{cases}
\]
\end{lemma}

\begin{proof}
We write
\[
R^{\xir}(\ell,k,D) = \sum_{\DD \in \Mar(D)} |\Van(\ell,k,\DD)^{\xir}|.
\]

Let  $B_k\subseteq \Fk$ be the negative curve cut out by the equation $x=0$. We do casework based on the subdivisor $D_0 \subseteq D$ which lies on $B_k$. 

Suppose we have a fixed partition $D = D_0 \sqcup D_1$ and a choice of $\DD \in \Mar(D)$ such that
\[
\DD = \DD_0 \sqcup \DD_1
\]
where $\DD_i$ lies above $D_i$, $\DD_0\subseteq B_k$, and $\DD_1$ is disjoint from $B_k$. There is a unique element $\alpha \in \mc{O}_{D_1}$ such that $\Van(\ell,k,\DD)$ is the set of sections
\[
s = A_0x^3 + A_1x^2y + A_2xy^2 + A_3y^3
\]
such that the restriction of $A_3$ to $D_0$ vanishes, $A_3\neq 0$, and 
\[
0 = \overline{A_0} + \overline{A_1}\alpha + \overline{A_2} \alpha^2 + \overline{A_3} \alpha^3 \in \mc{O}_{D_1}
\]
holds. If $\deg D_0>\ell-3k = \deg A_3$, then no choice of $A_3$ satisfies these constraints. On the other hand, if $\ell-3k\ge \deg D_0$, there are $q^{\ell-3k-\deg D_0+1}-1$ choices of $A_3 \in H^{0}(\PP^1, \mc{O}(\ell-3k))$ such that $A_3$ vanishes on $D_0$ but $A_3\neq 0$. Furthermore, for any triple $(A_1, A_2, A_3)$, the map
\[
H^{0}(\PP^1, \mc{O}(\ell)) \to \mc{O}_{D_1}
\]
given by
\[
A_0 \to  \overline{A_0} + \overline{A_1}\alpha + \overline{A_2} \alpha^2 + \overline{A_3} \alpha^3 
\]
is surjective, since $\ell \ge \deg D \ge \deg D_1$. Therefore, if $\ell-3k \ge \deg D_0$, then
\begin{align*}
|\Van(\ell,k,\DD)^{\xir}| &= q^{\ell-\deg D_1+1} q^{\ell-k+1}q^{\ell-2k+1} \left(q^{\ell-3k-\deg D_0+1}-1\right)\\
& = q^{4\ell-6k+4 - \deg D} - q^{3\ell-3k+3 - \deg D_1}.
\end{align*}
For any choice of $D_0 \subseteq D$, there are $q^{\deg D_1} = q^{\deg D - \deg D_0}$ choices of sections $\DD \in \Mar(D)$ such that $D_0$ is the subdivisor where $\DD$ lies on $B_k$. So
\begin{align*}
R^{\xir}(\ell,k,D) &= \sum_{D_0\subseteq D}q^{\deg D- \deg D_0}\begin{cases}
q^{4\ell-6k+4 - \deg D} - q^{3\ell-3k+3 - \deg D_1} & \text{if }\ell-3k \ge \deg D_0,\\
0 &\text{else.}
\end{cases}\\
&=  \sum_{D_0\subseteq D}\begin{cases}
q^{4\ell-6k+4 - \deg D_0} - q^{3\ell-3k+3} & \text{if }\ell-3k \ge \deg D_0,\\
0 &\text{else.}
\end{cases}
\end{align*}
as desired.
\end{proof}
As a corollary, we have an exact formula for $\Phi^{\xir}(\ell,k,D)$ in the small range.
\begin{cor}
\label{cor:small-range-hash}
If $\ell\ge 3k \ge 0$ and $\ell\ge \deg D$, then
\[
\Phi^{\xir}(\ell,k,D) =\begin{cases}
q^{4\ell-6k+4 - \deg D} - q^{3\ell-3k+3} &\text{if }\ell - 3k \ge \deg D,\\
0 &\text{else.}
\end{cases}
\]
\end{cor}
\begin{proof}
This follows directly from inclusion-exclusion.
\end{proof}

Note that we can write
\[
\Phi^{\ir}(\ell,k,D) = \Phi^{\xir}(\ell,k,D) - \Phi^{\sr}(\ell,k,D)
\]
recalling that for a subset $S \subseteq H^{0}(\Fk, \mc{O}(3,\ell))$, $S^{\xir}$ denotes the set of nonzero sections $s\in S$ which are not $x$-reducible, and $S^{\sr}$ denotes the set of nonzero sections $s \in S$ such which are specially reducible (that is, horizontally irreducible but not $x$-reducible). Putting together \Cref{cor:small-range-hash} with our previous bound for $|\Phi^{\sr}(\ell,k,D)|$ in \Cref{prop:hr-er-z-bound}, we obtain the following estimate for $\Phi^{\ir}(\ell,k,D)$ in the small range.
\begin{prop}\label{prop:phi-ir-small}
Let $\ell,k \ge 0$ be integers and let $D$ be a reduced effective divisor on $\PP^1$ such that $\ell\ge 3k$ and $\ell\ge \deg D$. Then
\[
\Phi^{\ir}(\ell,k,D) = \left(\begin{cases} q^{4\ell-6k+4-\deg D} - q^{3\ell-3k+3}&\text{if $\ell-3k \ge \deg D$}\\
0 &\text{else.}\end{cases}\right) + O\left(3^{\omega(D)}q^{3\ell-5k}\right).
\]
\end{prop}

\subsection{Medium Range}
The medium range is the critical range for the sieve. We will use the results of \Cref{sec:restrictions-line-bundles} to estimate $R(\ell,k,D)$. The rough strategy is to observe that any $\DD \in \Mar(D)$ either induces a ``perfect sieve'' or forces a section passing through $\DD$ to be horizontally reducible (or both!). Therefore our estimate will reduce to a bound on horizontally reducible sections.

More precisely, assume for now that we have integers $\ell,k\ge 0$ and a reduced effective divisor $D$ such that $\ell-3k \ge 0$ and $\ell-2k \ge \frac{\deg D + k}{2}$. By definition, we can write
\[
R(\ell,k,D) = \sum_{\DD \in \Mar(D)} |\Van(\ell,k,\DD)|.
\]

Let $\LL = \mc{O}(3,\ell)$ on $\Fk$. 

Now by \Cref{prop:dichotomy}, we know that for all $\DD \in \Mar(D)$, either
\[
\alpha_{\DD}: H^{0}(\Fk, \LL) \to H^{0}(\DD, \mc{O}_{D})
\]
is surjective, or every $s \in \Van(\ell,k,\DD)$ is horizontally reducible.

Let $S_0\sqcup S_1$ be a partition of $\Mar(D)$ such that $\DD \in S_0$ implies $\alpha_{\DD}$ is surjective and $\DD \in S_1$ implies all $s\in \Van(\ell,k,\DD)$ are horizontally reducible. (If some $\DD$ satisfies both conditions, then arbitrarily choose one of $S_0$ or $S_1$ to put it into.)

If $\DD \in S_0$, we have 
\[
|\Van(\ell,k,\DD)| = |H^{0}(\Fk, \LL)|q^{-\deg D}.
\]
On the other hand, if $\DD \in S_1$ we know at least that 
\[
|\Van(\ell,k,\DD)| \ge |H^{0}(\Fk, \LL)|q^{-\deg D}.
\]

So we have
\begin{align*}
R(\ell,k,D) &= \sum_{\DD \in S_0} |\Van(\ell,k,\DD)| + \sum_{\DD \in S_1} |\Van(\ell,k,\DD)|\\
&= \sum_{\DD \in \Mar(D)} |H^{0}(\Fk, \LL)|q^{-\deg D}+ \sum_{D \in S_1} \left(|\Van(\ell,k,\DD)| - |H^{0}(\Fk, \LL)|q^{-\deg D}\right)
\end{align*}
This is our division into a main term and an error term. To bound the error term we note
\begin{align*}
\sum_{D \in S_1} \left(|\Van(\ell,k,\DD)| - |H^{0}(\Fk, \LL)|q^{-\deg D}\right)&\le \sum_{D \in S_1} |\Van(\ell,k,\DD)|\\
&\le r_D(0) + \sum_{s \in H^{0}(\Fk, \LL)^{\hr}} r_D(s)\\
&= R^{z}(\ell,k,D) + R^{\hr}(\ell,k,D) \\
&= O\left(2^{\omega(D)}q^{\deg D} + 5^{\omega(D)}q^{3\ell-3k}\right)
\end{align*}
where the final bound follows from \Cref{prop:hr-er-z-bound}.

Let 
\[
\widehat{R}(\ell,k,D) \coloneq  \sum_{\DD \in \Mar(D)} |H^{0}(\Fk, \LL)|q^{-\deg D} = \prod_{P \mid D} \left(1+q^{-\deg P}\right)q^{4\ell-6k+4}.
\]
We have just shown
\[
R(\ell,k,D) = \widehat{R}(\ell,k,D) + O\left(2^{\omega(D)}q^{\deg D} + 5^{\omega(D)}q^{3\ell-3k}\right).
\]
Now we also know that
\begin{align*}
R^{\ir}(\ell,k,D) &= R(\ell,k,D) + O\left(R^{z}(\ell,k,D) + R^{\hr}(\ell,k,D)\right)\\
&= R(\ell,k,D) + O\left(2^{\omega(D)}q^{\deg D} + 5^{\omega(D)}q^{3\ell-3k}\right)
\end{align*}
So putting together both estimates, we find that
\[
R^{\ir}(\ell,k,D)=\widehat{R}(\ell,k,D) + O\left(2^{\omega(D)}q^{\deg D} + 5^{\omega(D)}q^{3\ell-3k}\right).
\]

Since we are in the range where $2\ell - 5k \ge \deg D$, in particular we have $3\ell - 3k \ge \deg D$, so we can simplify the error to
\[
R^{\ir}(\ell,k,D) = \widehat{R}(\ell,k,D) + O\left(5^{\omega(D)} q^{3\ell - 3k}\right).
\]

Now recall that we have
\[
\Phi^{\ir}(\ell,k,D) = \sum_{D_1\mid D} \mu(D-D_1) R^{\ir}(\ell,k,D_1).
\]
Therefore we have
\[
\Phi^{\ir}(\ell,k,D) = \widehat{\Phi}(\ell,k,D) + O\left(\sum_{D_1 \mid D} 5^{\omega(D_1)} q^{3\ell-3k}\right),
\]
where we define
\begin{align*}
\widehat{\Phi}(\ell,k,D)&\coloneq  \sum_{D_1 \mid D} \mu(D-D_1) \widehat{R}(\ell,k,D_1)\\
&= \sum_{D_1 \mid D} \mu(D-D_1)\prod_{P \mid D_1} \left(1+q^{-\deg P}\right)q^{4\ell-6k+4}\\
&= \sum_{D_1 \mid D} \mu(D-D_1) \sum_{D_2 \mid D_1} q^{-\deg D_2} q^{4\ell-6k+4}\\
&=q^{4\ell-6k+4-\deg D}.
\end{align*}

Here we are using the fact that if the condition $\ell-2k \ge \frac{\deg D + k}{2}$ holds and $D_1\mid D$, then $\ell-2k \ge \frac{\deg D_1 + k}{2}$ holds as well. If we note that $\sum_{D_1 \mid D} 5^{\omega(D_1)} = 6^{\omega(D_1)}$, we get a proof of the following proposition.

\begin{prop}
\label{prop:medium-estimate}
Let $\ell, k \ge 0$ be integers and let $D$ be a reduced effective divisor on $\PP^1$ such that $\ell\ge 3k$ and 
\[
\ell-2k \ge \frac{\deg D + k}{2}.
\]
Then
\[
\Phi^{\ir}(\ell,k,D) = q^{4\ell-6k+4-\deg D} + O\left(6^{\omega(D)} q^{3\ell-3k}\right).
\]
\end{prop}

As a corollary, if we additionally assume $\deg D > \ell$, the main term becomes smaller than the error term, so we get the following bound in the medium range.
\begin{prop}\label{prop:new-medium-estimate}
    Let $l, k \ge 0$ be integers and let $D$ be a reduced effective divisor on $\PP^1$. Assume $\ell \ge 3k$, and assume $(\ell,k,D)$ is in the medium range, meaning both of the following inequalities hold:
    \[
    \ell - 2k \ge \frac{\deg D+k}{2}\text{ and }     \deg D >\ell.
    \]
    Then we have
    \[
    \Phi^{\ir}(\ell, k, D) = O\left(6^{\omega(D)} q^{3\ell-3k}\right).
    \]
\end{prop}
\begin{proof}
    Since $\ell < \deg D$, in particular $\ell-3k < \deg D$, so we have $4\ell - 6k -\deg D\le 3\ell -3k$.
\end{proof}

\subsection{Large Range}
In the large range, we will use a naive uniformity estimate to give an upper bound for $\Phi^{\ir}(\ell,k,D)$.

\begin{prop}[Uniformity]\label{prop:uniformity}
Let $\ell,k \ge 0$ be integers such that $\ell\ge 3k$. Let $D$ be a reduced effective divisor on $\PP^1$. Then we have
\[
|\Phi^{ir}(\ell,k,D)| = O\left(3^{\omega(D)} q^{4\ell-6k}\right).
\]
\end{prop}

\begin{proof}
Let $\LL = \mc{O}(3,\ell)$. We can write
\[
|\Phi^{\ir}(\ell,k,D)| \le \sum_{s \in H^{0}(\Fk, \LL)^{\ir}} \prod_{P\mid D} |a_p(s)|
\]
where we recall $a_p(s) = r_p(s)-1$ and $r_p(s)$ counts the number of roots of $s|_{P}$.

Consider a given nonzero section $s \in H^{0}(\Fk, \LL)$ and a fixed closed point $P\mid D$. If $s$ vanishes on the fiber $\pi^{-1}(P)$, then $a_P(s) = q^{\deg P}$, and otherwise $|a_P(s)|\le 2$.

Let $\Fi(s,D)$ denote the maximal subdivisor $D_1\subseteq D$ such that $s$ vanishes on the fibers above $D_1$. We will do casework on the value of $D_1=\Fi(s,D)$. We have
\begin{align*}
|\Phi^{\ir}(\ell,k,D)|&\le \sum_{s \in H^{0}(\Fk, \LL)^{n}} \prod_{P \mid D} |a_P(s)|\\
&=\sum_{D_1 \subseteq D} \sum_{\substack{s \in H^{0}(\Fk \LL)^{n}\\ \Fi(s,D)=D_1}} \prod_{P \mid D} |a_P(s)|\\
&=\sum_{D_1 \subseteq D} \sum_{\substack{s \in H^{0}(\Fk,\LL)^{n}\\ \Fi(s,D)=D_1}}  2^{\omega(D)- \omega(D_1)} q^{\deg D_1}.
\end{align*}
Now by \Cref{lemma:fib-bound}, the number of sections $s \in H^{0}(\Fk, \LL)$ such that $\Fi(s,D) = D_1$ is at most
\[
|H^{0}(\Fk, \LL)^{\nz}| q^{-\deg D_1}.
\]
So
\begin{align*}
|\Phi^{\ir}(\ell,k,D)|&\le \sum_{D_1 \subseteq D} \left(|H^{0}(\Fk, \LL)^{\nz}| q^{-\deg D_1}\right)\left(2^{\omega(D) - \omega(D_1)} q^{\deg D_1}\right)\\
&= 3^{\omega(D)} |H^{0}(\Fk, \LL)^{\nz}|= O\left(3^{\omega(D)} q^{4\ell-6k}\right)
\end{align*}
as desired.
\end{proof}

\section{Counting Smooth Covers}\label{sec:smooth-counting}
We have just estimated $\Phi^{\ir}(\ell, k, D)$ for three different ranges. Our intermediate goal is to estimate

\[
\Psi^{\ir}(N) = \sum_{\substack{N', d \ge 0 \\ N' + d = N}} \sum_{(\ell, k) \in \Par(N')} \frac{1}{|\Gamma_k|}\sum_{\deg D = d} \mu(D) \Phi^{\ir}(\ell, k , D),
\]
which is a weighted sum of $\Phi^{\ir}(\ell, k, D)$ in all of the ranges.

\begin{defn}
Let 
\[
\widehat{\Phi}^{\ir}(\ell,k,D)\coloneq  \begin{cases} q^{4\ell-6k+4-\deg D} - q^{3\ell-3k+3}&\text{if $\ell-3k \ge \deg D$,}\\
0 &\text{else.}\end{cases}
\]

\end{defn}

By combining the previous three sections, we obtain the following estimate for the three ranges of $D$.
\begin{prop}
    Let $\ell,k$ be nonnegative integers such that $\ell\ge 3k$, and let $D$ be a reduced divisor on $\PP^1$. Then we have an estimate for $\Phi^{\ir}(\ell,k,D)$ in each range:
    \begin{itemize}
    \item Small range: if $\ell\ge \deg D$, then 
    \[
    \Phi^{\ir}(\ell,k,D) = \widehat{\Phi}^{\ir}(\ell,k,D) + O\left(3^{\omega(D)}q^{3\ell-5k}\right).
    \]
    \item Medium range: if $\ell<\deg D$ and $2\ell-5k> \deg D$, then 
    \[
    \Phi^{\ir}(\ell,k,D) = O\left( 6^{\omega(D)} q^{3\ell-3k}\right).
    \]
    \item Large range: if $\deg D \ge 2\ell-5k$ and $\ell <\deg D$, then 
    \[
    \Phi^{\ir}(\ell,k,D) = O\left(3^{\omega(D)}q^{4\ell-6k}\right).
    \]
    \end{itemize}
\end{prop}

\begin{proof}
    This follows from \Cref{prop:phi-ir-small}, \Cref{prop:new-medium-estimate}, and \Cref{prop:uniformity}.
\end{proof}

Note that $\widehat{\Phi}^{\ir}(\ell,k,D)$ is zero when $(\ell,k,D)$ is outside of the small range.

As a consequence, we can estimate $\Psi(N)$ as a sum of a main term and an error term.
\begin{defn} Let $N\ge 0$. We define
\[
\widehat{\Psi}(N)\coloneq  \sum_{\substack{d,N'\ge 0\\ d+N'=N}}\sum_{\substack{D\\ \deg D = d}}\mu(D)\sum_{(\ell,k) \in \Par(N')}\frac{1}{|\Gamma_k|}\widehat{\Phi}^{\ir}(\ell,k,D).
\]
\end{defn}

As a consequence, we can write
\[
\Psi(N) = \widehat{\Psi}(N) + O\left(E_{\sml}+E_{\med}+E_{\lrg}\right)
\]
where $E_{\sml}$, $E_{\med}$, and $E_{\lrg}$ are the errors coming from the small, medium, and large ranges respectively. More precisely, we have
\begin{align*}
E_{\sml} &\coloneq  \sum_{\substack{d,N'\ge 0\\ d+N'=N}}\sum_{\substack{D\\ \deg D = d}}\sum_{(\ell,k) \in \Par(N')}\frac{1}{|\Gamma_k|}\begin{cases} O\left(3^{\omega(D)} q^{3\ell-5k}\right) &\text{ if $\ell\ge d$,}\\
0&\text{ else.}\end{cases}\\
E_{\med} &\coloneq  \sum_{\substack{d,N'\ge 0\\ d+N'=N}}\sum_{\substack{D\\ \deg D = d}}\sum_{(\ell,k) \in \Par(N')}\frac{1}{|\Gamma_k|}\begin{cases} O\left(6^{\omega(D)} q^{3\ell-3k}\right) &\text{ if $\ell< d<2\ell-5k$,}\\
0&\text{ else.}\end{cases}\\
E_{\lrg} &\coloneq  \sum_{\substack{d,N'\ge 0\\ d+N'=N}}\sum_{\substack{D\\ \deg D = d}}\sum_{(\ell,k) \in \Par(N')}\frac{1}{|\Gamma_k|}\begin{cases} O\left(3^{\omega(D)} q^{4\ell-6k}\right) &\text{ if $2\ell-5k\le d<\ell$,}\\
0&\text{ else.}\end{cases}
\end{align*}
We bound each of these errors in turn.

\begin{prop}\label{prop:Erange-bounds}
We have 
\begin{align*}
    E_{\sml} &= O\left(q^{3N/2}\right),\\
    E_{\med} &= O\left(N^7q^{4N/3}+1\right),\\
    E_{\lrg} &= O\left(N^4 q^{3N/2}+1\right),
\end{align*}

and therefore
\[
\Psi(N) = \widehat{\Psi}(N) + O\left(N^4 q^{3N/2}+1\right).
\]
\end{prop}
To prove \Cref{prop:Erange-bounds}, we'll use two auxiliary lemmas, which each follow from a small computation. It will also be useful to note that $1/|\Gamma_k| = O(q^{-k})$.

\begin{lemma}\label{lem:k-bound}
    Let $N', r \ge 0$ be integers. Then
    \[
    \sum_{(\ell,k) \in \Par(N')} q^{-rk/2}\ll \begin{cases}
        N'+1 & \text{if $r=0$},\\
        1 & \text{if $r>0$}.
    \end{cases}
    \]
\end{lemma}
\begin{proof}
    If $(\ell,k) \in \Par(N')$, then $0\le k \le N'/3$, and conversely any such $k$ corresponds to at most one parameter pair $(\ell,k) \in \Par(N')$, so the lemma is clear.
\end{proof}
\begin{lemma}\label{lem:d-bound}
    Let $N,b,r$ be integers satisfying $N\ge 0$, $b \ge 1$, and $r \ge 2$. Then
    \[
    \sum_{0 \le d \le N}\sum_{\substack{D \\ \deg D = d}} b^{\omega(D)} q^{-rd/2} \ll \begin{cases} 
    N^b+1&\text{if $r=2$},\\
    1&\text{if $r \ge 3$}.
    \end{cases}
    \]
\end{lemma}
\begin{proof}
Let $\lambda = q^{-r/2}$. The coefficients of the formal generating function
\begin{equation}\label{eq:pre-zeta}
\sum_{d\ge 0}\sum_{\deg D = d} b^{\omega(D)} \lambda^{d} T^{d} = \prod_{P} (1+b T^{\deg P})
\end{equation}
are bounded from above by the coefficients of
\begin{equation}
\prod_{P}(1-(\lambda T)^{\deg P})^{-b} = Z(\lambda T)^{b}=\left(\frac{1}{(1-\lambda T)(1-q\lambda T)}\right)^{b}.\label{eq:zeta-cube} 
\end{equation}
Both are \eqref{eq:pre-zeta} and \eqref{eq:zeta-cube} are (uniformly) convergent power series for $|T|< \lambda^{-1} q^{-1-\epsilon}$, so if $r\ge 3$, then $\lambda \le q^{-3/2}$ and the sum of the first $N$ coefficients of \eqref{eq:pre-zeta} is at most 
\[
Z(\lambda)^{b}\le Z(q^{-3/2})^{b} = \left(\frac{1}{(1-q^{-3/2})(1-q^{-1/2})}\right)^{b}\ll_b 1.
\]
Now if $r=2$, then $\lambda = q^{-1}$, and we seek to bound the $N$th coefficient of
\begin{equation}\label{eq:iterate-bound}
\left(\frac{1}{1-T}\right)^{b+1}\left(\frac{1}{1-q^{-1}T}\right)^{b} = \left(\sum_{N\ge 0} \binom{N+b}{b}T^{N}\right)\left(\frac{1}{1-q^{-1}T}\right)^{b}.
\end{equation}
For any power series $\sum_{N \ge 0} c(N) T^{N}$ with $c(N)$ non-decreasing in $N$, we have
\begin{align*}
\left[T^{N}\right]\left(\sum_{N' \ge 0} c(N') T^{N'}\right)\left(\frac{1}{1-q^{-1}T}\right) = \sum_{0\le m\le N} c(N-m) q^{-m}&\le \sum_{0\le m \le N} c(N) q^{-m}\\
&\le c(N) \left(\frac{1}{1-q^{-1}}\right).
\end{align*}
So by applying this iteratively, the $N$th coefficient of \Cref{eq:iterate-bound} is at most
\[
\binom{N+b}{b} \left(\frac{1}{1-q^{-1}}\right)^{b}\ll_b N^b+1.
\]

\end{proof}

Now using \Cref{lem:k-bound} and \Cref{lem:d-bound}, we can prove \Cref{prop:Erange-bounds}.
\begin{proof}[Proof of \Cref{prop:Erange-bounds}]
    To bound $E_{\sml}$, recall that if $(\ell, k) \in \Par(N')$, then $2\ell - 3k=N'$. So if $N' = N-d$, then 
    \[
    \frac{1}{|\Gamma_k|} q^{3\ell-5k} \ll q^{3\ell -6k} = q^{3N'/2 - 3k/2} = q^{3N/2 - 3d/2 - 3k/2}.
    \]
    (We don't need to use the small range restriction $\ell \ge d$ to get this bound.) So we have
    \begin{align*}
        E_{\sml} &\ll \sum_{\substack{d,N' \ge 0 \\ d+N'=N}}\sum_{\deg D = d} \sum_{(\ell,k) \in \Par(N')} \frac{1}{|\Gamma_k|} 3^{\omega(D)} q^{3\ell-5k}\\
        &\ll\sum_{\substack{d,N' \ge 0 \\ d+N'=N}}\sum_{\deg D = d} 3^{\omega(D)} q^{3N/2-3d/2 - 3k/2}\\
        &\ll\sum_{\substack{d,N' \ge 0 \\ d+N'=N}}\sum_{\deg D = d}  3^{\omega(D)} q^{3N/2-3d/2}\\
        &=O(q^{3N/2}).
    \end{align*}
    by \Cref{lem:k-bound} and \Cref{lem:d-bound}.

    Now we bound $E_{\med}$. Suppose $N' = N-d$ and $(\ell,k) \in \Par(N')$, so $2\ell - 3k = N$. Then if  $\ell< d$, we have
    \[
    \frac{1}{|\Gamma_k|} q^{3\ell - 3k} \ll q^{3\ell - 4k} \le q^{8\ell/3+d/3 - 4k} = q^{4N'/3 + d/3} = q^{4N/3 - d}.
    \]
    (Here we use the bound $\ell<d$, but we don't use the bound $d<2\ell - 5k$.) So
    \begin{align*}
    E_{\med} &\ll \sum_{\substack{d,N' \ge 0 \\ d+N'=N}}\sum_{\deg D = d} \sum_{(\ell,k) \in \Par(N')} \frac{1}{|\Gamma_k|} 6^{\omega(D)} q^{3\ell-3k}\\
    &\ll \sum_{\substack{d,N' \ge 0 \\ d+N'=N}}\sum_{\deg D = d} \sum_{(\ell,k) \in \Par(N')} 6^{\omega(D)} q^{4N/3-d}\\
    &\ll \sum_{\substack{d,N' \ge 0 \\ d+N'=N}}\sum_{\deg D = d} (N+1) 6^{\omega(D)} q^{4N/3-d}\\
    &\ll (N^6+1)(N+1)q^{4N/3} = O\left(N^7q^{4N/3}+1\right).
    \end{align*}

    Finally, we bound $E_{\lrg}$. Suppose $N'=N-d$ and $(\ell,k) \in \Par(N')$, so $2\ell - 3k=N$. Then if $2\ell - 5k < d$, we have $N'-2k <d$, so $4N'-2k< 3N-2d$, implying
    \[
    \frac{1}{|\Gamma_k|} q^{4\ell - 6k} = \frac{1}{|\Gamma_k|} q^{2N'}\ll q^{2N'-k}\le q^{3N/2 - d}.
    \]
    So
    \begin{align*}
        E_{\lrg}& \ll \sum_{\substack{d,N' \ge 0 \\ d+N'=N}}\sum_{\deg D = d} \sum_{(\ell,k) \in \Par(N')} \frac{1}{|\Gamma_k|} 3^{\omega(D)} q^{4\ell-6k}\\
        &\ll \sum_{\substack{d,N' \ge 0 \\ d+N'=N}}\sum_{\deg D = d} \sum_{(\ell,k) \in \Par(N')}3^{\omega(D)} q^{3N/2-d}\\
        &\ll \sum_{\substack{d,N' \ge 0 \\ d+N'=N}}\sum_{\deg D = d}(N+1) 3^{\omega(D)} q^{3N/2-d}\\
        &\ll (N^3+1)(N+1) q^{3N/2} = O\left(N^4q^{3N/2}+1\right).
    \end{align*}
\end{proof}

Now that we have bounded the error, we compute asymptotics for the main term of $\Psi(N)$.
\begin{defn}
    Let $\widehat{F}(T) \in \C[[T]]$ be the formal power series defined by 
    \[
    \widehat{F}(T)\coloneq \sum_{N \ge 0} \widehat{\Psi}(N) T^{N}.
    \]
\end{defn}
\begin{prop}\label{prop:gen-func}
    We have 
    \[
    \widehat{F}(T)= \frac{q^2}{q^2-1}\frac{(1-q^3 T^3)(1-q^4T^3)(1+q^6 T^3)}{(1-q^5 T^3)(1-q^4 T^2)(1-q^3 T^2)}.
    \]
\end{prop}
\begin{proof}
By definition, we have
\[
\widehat{F}(T)=  \sum_{N\ge 0} T^{N}\sum_{\substack{d, N' \ge 0\\ d+N' = N}} \sum_{\substack{D \\ \deg D=d}} \mu(D) \sum_{(\ell,k) \in \Par(N')} \frac{1}{|\Gamma_k|} \widehat{\Phi}^{\ir}(\ell,k,D).
\]

For a pair of nonnegative integers $(\ell,k) \in \Par(N')$, let $n = \ell-3k$. We must have $n \ge 0$ and $N' = 2(n+3k)-3k=2n+3k$. Conversely, for any pair of integers $n, k\ge 0$, we have a pair $(n+3k,k)\in \Par(N')$ with $N'\coloneq  2n+3k$. Also note that $\widehat{\Phi}^{\ir}(\ell,k,D)$ is only nonzero when $n \ge \deg D$. Therefore we can re-write the above sum as
\begin{align*}
\widehat{F}(T) &= \sum_{N' \ge 0} T^{N'} \sum_{d \ge 0} T^{d} \sum_{\substack{D\\\deg D = d}}\mu(D) \sum_{(\ell,k) \in \Par(N')} \frac{1}{|\Gamma_k|} \widehat{\Phi}^{\ir}(\ell,k,D)\\
&=\sum_{d \ge 0}T^{d}\sum_{\substack{D \\ \deg D  = d}}\mu(D)\sum_{n \ge 0} \sum_{k \ge 0} T^{2n+3k} \frac{1}{|\Gamma_k|} \widehat{\Phi}^{\ir}(n+3k,k,D)\\
&= \sum_{d\ge 0} T^{d}\sum_{\substack{D \\ \deg D = d}} \mu(D)\sum_{n \ge d}\sum_{k \ge 0} T^{2n + 3k} \frac{1}{|\Gamma_k|}\left( q^{4(n+3k)-6k+4-d} - q^{3(n+3k)-3k+3}\right)\\
&= \sum_{d\ge 0} T^{d}\sum_{\substack{D \\ \deg D = d}} \mu(D)\sum_{k \ge 0}\frac{1}{|\Gamma_k|}q^{6k}T^{3k} \sum_{n \ge d} T^{2n} \left( q^{4n+4-d} - q^{3n+3}\right).
\end{align*}

Note that
\[
\sum_{n \ge d} T^{2n}\left(q^{4n + 4 - d} - q^{3n +3}\right) = T^{2d}q^{3d} \left(\frac{q^4}{1-q^{4}T^2} - \frac{q^3}{1-q^{3}T^2}\right) = T^{2d}q^{3d}\frac{q^4-q^3}{(1-q^4T^2)(1-q^3T^2)}.
\]
Also, by \Cref{lem:aut-series}, we have
\[
\sum_{k\ge 0} \frac{1}{|\Gamma_k|} q^{6k} T^{3k} = \frac{1}{|\Gamma_0|} \frac{1+q^6 T^3}{1-q^5 T^3}. 
\]
Therefore we have
\begin{align*}
\widehat{F}(T) &= \sum_{d \ge 0} T^{d} \sum_{\substack{D \\ \deg D = d}} \mu(D) \sum_{k \ge 0} \frac{1}{|\Gamma_k|} q^{6k} T^{3k} \left( q^{3d} T^{2d}  \frac{q^4-q^3}{(1-q^4 T^2)(1-q^3 T^2)}\right)\\
&=\sum_{d \ge 0} \sum_{\substack{D\\ \deg D = d}} \mu(D) q^{3d} T^{3d}\left(\frac{q^4-q^3}{|\Gamma_0|}\right)\frac{(1+q^6T^3)}{(1-q^5T^3)(1-q^4T^2)(1-q^3T^2)}.
\end{align*}

Finally, recall we have
\[
\sum_{D} \mu(D) T^{\deg D} = Z(T)^{-1} = (1-T)(1-qT).
\]
So
\[
\sum_{d \ge 0} \sum_{\substack{D \\ \deg D = d}} \mu(D) q^{3d} T^{3d} =Z(q^3 T^3)^{-1}= (1-q^3 T^3)(1-q^4 T^3).
\]
In conclusion we can write
\[
\widehat{F}(T) = \frac{q^4-q^3}{|\Gamma_0|}\frac{(1-q^3 T^3)(1-q^4 T^3)(1+q^6 T^3)}{(1-q^5 T^3)(1-q^4 T^2)(1-q^3 T^2)}.
\]

\end{proof}
\begin{lemma}\label{lem:aut-series}
As a formal power series in $\C[[T]]$, we have
\[
\sum_{k \ge 0} \frac{T^k}{|\Gamma_k|} = \frac{1}{|\Gamma_0|} \left(\frac{1+T}{1-T/q}\right).
\]
\end{lemma}
\begin{proof}
Note that $|\Gamma_0| = |\GL_2(\Fq)| = (q^2-q)(q^2-1)$, while $|\Gamma_k| = q^{k+1}(q-1)^2$ for $k \ge 1$. Therefore
\[
\sum_{k \ge 0} \frac{T^k}{|\Gamma_k|} = \frac{1}{|\Gamma_0|} + \frac{1+1/q}{|\Gamma_0|} \frac{T}{1-T/q} = \frac{1}{|\Gamma_0|} \frac{1+T}{1-T/q}.
\]
\end{proof}

\begin{defn}
    Let $\widehat{G}(T):= (1-T)(1-qT)\widehat{F}(T)$. Define $\widehat{\Theta}(N)$ by
    \[
    \widehat{G}(T) =: \sum_{N \ge 0} \widehat{\Theta}(N)  T^{N}.
    \]
\end{defn}
\begin{lemma}\label{lem:theta-near-tilde}
    We have
    \[
    \Theta(N) = \widehat{\Theta}(N) + O(N^4q^{3N/2}+1).
    \]
\end{lemma}
\begin{proof}
    Let 
    \[
    F(T) = \sum_{N \ge 0} \Psi^{\ir}(N) T^{N}.
    \]

    Since $\widehat{G}(T) = (1-T)(1-qT) \widehat{F}(T)$ and $G(T) = (1-T)(1-qT) F(T)$, we have both of the following identities:
    \begin{align*}
    \Theta(N) &= \Psi(N) - (q+1)\Psi(N-1) + q\Psi(N-2),\\
    \widehat{\Theta}(N) &= \widehat{\Psi}(N) - (q+1) \widehat{\Psi}(N-1)  + q \widehat{\Psi}(N-2).
    \end{align*}
    Here $\Psi(M)$ and $\widehat{\Psi}(M)$ are both $0$ if $M<0$.

    Now we have showed that
    \[
    \Psi(N) = \widehat{\Psi}(N) + O\left(N^4 q^{3N/2}+1\right),
    \]
    so it is clear that the same error bound (up to a constant) holds for $\Theta(N)$.
\end{proof}

Now by \Cref{prop:gen-func}, we have
\begin{align*}
 \widehat{G}(T) &= (1-T)(1-qT) \widehat{F}(T)\\
 &= \frac{q^2}{q^2-1}\frac{(1-T)(1-qT) (1-q^3 T^3)(1-q^4 T^3) (1-q^2T + q^4 T^2)}{(1-q^5 T^3)(1-q^2 T)(1-q^3 T^2)}.
\end{align*}
From this power series, we extract a main term, secondary term, and error term for $\widehat{\Theta}(N)$ and therefore for $\Theta(N)$.

We can write
\[
\widehat{G}(T) = \frac{(1-q^{-3})(1+q^{-1})}{1-q^2 T}  - q^{-1}(1-q^{-2})\left(\frac{1+q+q^2T+q^3T+q^4 T^2}{1-q^5 T^3}\right) + \frac{P(T)}{1-q^3 T^2}
\]
for some polynomial $P(T)$.

Therefore, if we write $N = 3m + i$ where $i \in \{0,1,2\}$, we have
\[
\widehat{\Theta}(N) = c_1 q^{2N} - c_2^{i} q^{5m}  + O\left(q^{3N/2}\right)
\]
where the constants are 
\begin{equation}\label{eq:explicit-constant}
\begin{aligned}
c_1&= (1-q^{-3})(1+q^{-1})\\
c_2^{i}&:= (1-q^{-2})q^{-1}\begin{cases}
    1+q & \text{if } i =0,\\
    q^2 + q^3&\text{if } i=1,\\
    q^4 &\text{if }i=2.
\end{cases}
\end{aligned}
\end{equation}

Putting this together with \Cref{lem:theta-near-tilde}, we conclude that
\[
\Theta(N) = c_1 q^{2N} - c_2^{i} q^{5m} + O(N^4 q^{3N/2}+1).
\]

Now by the results of \Cref{sec:c3-insep}, the total number of $C_3$ covers and inseparable covers $f: X\to \PP^1$ such that $X$ is smooth and irreducible with discriminant degree $2N$ is at most $O(Nq^{N}+1)$. This means
\[
\Cov_{3}(2N) = \Theta(N) + O(Nq^{N}+1).
\]
Finally, putting together all of the pieces above, we conclude our main theorem.
\begin{thm}\label{thm:main-thm}
    Let $N$ be a nonnegative integer, and let $m, i$ be integers such that $m\ge 0$, $i\in \{0,1,2\}$ and $N=3m+i$. 

    Then the number of degree 3 field extensions $K/\F_q(t)$ up to isomorphism satisfying $\Nm \Disc_{K/\F_q(t)} = q^{2N}$, or equivalently the number degree $3$ branched covers $X\to \PP^1$ up to isomorphism with branch degree $2N$, is equal to 
    \[
    \Cov_3(2N) = c_1 q^{2N} - c_2^{i} q^{5m} + O\left(N^4q^{3N/2}+1\right)
    \]
    where $c_1$ and $c_2^{i}$ are explicit constants given above in \eqref{eq:explicit-constant}. 

    (For a field extension or a branched cover, each isomorphism class is counted once.)
\end{thm}

We remark that the leading constant for the main term can be written as is
\[
c_1= \frac{(1-q^{-3})(1-q^{-2})}{(1-q^{-1})}= \frac{ Z(q^{-3})^{-1}}{1-q^{-1}} = \frac{1}{q^{-1}(q-1)\zeta_{\PP^1}(3)} 
\]
which aligns with the main term of Gunther \cite{Gun17} for branched covers of a nice genus $g$ curve $C$ over $\F_q$. In this case Gunther found the leading term
\[
c_1(C) = \frac{\Pic^{0}(C)}{q^{g-1}(q-1)\zeta_{C}(3)}.
\]
\section{Cyclic and Inseparable Covers}
\label{sec:c3-insep}
Let $C = \PP^1$. Recall a triple cover $X \to C$ is in $\Trip^{\sm,\ir}(C)$ if $X$ is smooth and irreducible. 

We call a smooth irreducible triple cover $X \to C$ a \textbf{$C_3$ cover} if $\Aut(X\to C) \cong C_3$. This is equivalent to $\F_q(C)/\F_q(t)$ being a $C_3$ Galois extension.

We call a smooth irreducible triple cover $X\xrightarrow{f} C$ an \textbf{inseparable cover} if the degree $3$ field extension $\F_q(X)/\F_q(t)$ is inseparable. This is equivalent to the vanishing of the discriminant $\Delta_f \in \mc{N}^{\otimes 2}$ (or equivalent to $f$ failing to be generically \'etale).

In this section, we bound the number of isomorphism classes of $C_3$ covers of $\PP^1$ and inseparable covers with a fixed discriminant degree.

\begin{prop}
    Let $N\ge 0$. The number of isomorphism classes of $C_3$ covers $X\to \PP^1$ with discriminant degree $2N$ is at most $O(Nq^{N}+1)$.
\end{prop}
\begin{proof}
    Let $C = \PP^1$. Recall the equivalence of groupoids
    \[
    \Trip^{\sm, \ir}(C) \xrightarrow{\sim} \TriSec^{\sm, \ir}(C).
    \]

    In particular, a smooth irreducible cover $X\to C$ corresponds to a nonzero section
    \[
    s \in H^{0}(F_k, \mc{O}(3,\ell))=H^0(F_k, \mc{O}(3) \otimes \pi^{*} \mc{O}(\ell))
    \]
    for some $k, \ell \ge 0$.

    The corresponding discriminant line bundle on $\PP^1$ satisfies $\mc{N}^{\otimes 2} \cong \mc{O}(4\ell - 6k)$, so the discriminant degree is $2N = 4\ell-6k$. Note that if $s$ is horizontally irreducible, then $\ell - 3k \ge 0$.

    Recall our notation for the automorphism group $\Gamma_k = \Aut(\mc{O}\oplus \mc{O}(-k))$. The automorphism group of a cover is identified with the stabilizer of $\Gamma_k$ acting on the corresponding section $s$. So in particular, if $X\to C$ has a nontrivial order $3$ automorphism, then there is a nontrivial order $3$ element $\gamma \in \Gamma_k$ fixing the section $s\neq 0$. Assuming such a pair $(\gamma, s)$ exists, we will determine its structure explicitly.

    We break our analysis into cases.

\noindent\textbf{Case 1:} $k=0$. 
    
    In this case we have $\Gamma_0 \cong \Aut(\mc{O} \oplus \mc{O}) \cong \GL_2(\F_q)$, and we can write $\gamma\in \Gamma_0$ in Jordan normal form. Up to conjugacy, $\gamma$ has the same action on $\F_q^{\oplus 2}$ as the action of $T$ on some $\F_q[T]$-module $M$, where either 
    \[
    M \cong \F_q[T]/f_1(T)\oplus \F_q[T]/f_2(T)
    \]
    with $\deg f_1=\deg f_2=1$, or
    \[
    M \cong \F_q[T]/f(T)
    \]
    with $\deg f = 2$. Furthermore we must have $f_1, f_2 \mid T^3-1$ in the first case and $f \mid T^3-1$ in the second case.

    Let's suppose $3 \mid q-1$. Choose a primitive third root of unit $\omega \in \F_q^{\times}$. The only case here is
    \[
    M \cong \F_q[T]/(T-\omega^{a_1}) \oplus \F_q[T]/(T-\omega^{a_2})
    \]
    meaning $\gamma$ is conjugate to $\begin{pmatrix} \omega^{a_1} & 0 \\ 0 & \omega^{a_2}\end{pmatrix}$ for some integers $a_1, a_2 \in \{0,1,2\}$.

    Now let's suppose $3\nmid q-1$. In this case, either 
    \[
    M \cong (\F_q[T]/(T-1))^{\oplus 2}
    \]
    meaning $\gamma = 1$, or
    \[
    M\cong \F_q[T]/(T^2+T+1),
    \]
    meaning $\gamma$ is conjugate to 
    \[\begin{pmatrix}
        0 & -1\\
        1 & -1
    \end{pmatrix}.
    \]
    Finally, suppose $\Char \F_q = 3$. Then $T^3-1 = (T-1)^3$, so if $\gamma \neq 1$, then 
    \[
    M \cong \F_q[T]/(T-1)^2
    \]
    and again $\gamma$ is conjugate to $\begin{pmatrix} 0 & -1\\  1 &-1 \end{pmatrix}$.
    
    To summarize, if $\gamma\neq 1$, then either $3\mid q-1$ and $\gamma$ is conjugate to $\begin{pmatrix} \omega^{a_1} & 0 \\ 0 & \omega^{a_2}\end{pmatrix}$ or $3\nmid q-1$ and $\gamma$ is conjugate to $\begin{pmatrix}
        0 & -1\\
        1 & -1
    \end{pmatrix}$, or $3\mid q$ and $\gamma$ is conjugate to $\begin{pmatrix} 0 & -1 \\ 1 & -1\end{pmatrix}$.

    Now we consider the action of $\gamma$ on a nonzero, horizontally irreducible section $s$. Write
    \[
    s  = A_0 x^3 + A_1 x^2 y + A_2 xy^2 + A_3 y^3
    \]
    for $A_0, A_1, A_2, A_3 \in \mc{O}(\ell)$.
    
\noindent\textbf{Subcase 1(a)}: $3\mid q-1$.

    In this case, $\gamma$ is conjugate to $\begin{pmatrix} \omega^{a_1} & 0 \\ 0 & \omega^{a_2}\end{pmatrix}$ for some $a_1, a_2 \in \{0,1,2\}$. We also require $\gamma \neq 1$, so we can't have $(a_1, a_2)=(0,0)$. Now (potentially after a change of basis) the action of $\gamma$ sends $s$ to 
    \[
    \gamma\cdot s = \omega^{2a_1-a_2} A_0 x^3 + \omega^{a_1} A_1 x^2 y + \omega^{a_2} A_2 xy^2 + \omega^{2a_2 - a_1} A_3 y^3.
    \]
    Suppose $\gamma \cdot s = s$ (and $s\neq 0$). Note that we must have $a_1 \neq a_2$ - else $\gamma$ acts on $s$ by a nonzero scalar, which is a contradiction.

    If only one entry in $(A_0, A_1, A_2, A_3)$ is nonzero, then $s$ must be horizontally reducible.

    If two adjacent entries of $(A_0, A_1, A_2, A_3)$ are nonzero, then $a_1 =a_2$, a contradiction.

    Now suppose $A_0, A_2 \neq 0$. Then $(2a_1-a_2, a_2) \equiv (0,0) \pmod 3$ implying $(a_1,a_2) = (0,0)$ and $\gamma=1$, a contradiction. Similarly, if $A_1, A_3 \neq 0$, then $\gamma=1$, a contradiction.

    Finally, suppose $A_1=A_2=0$ and $A_0, A_3 \neq 0$. Then we must have $a_1 \equiv 2a_2 \pmod{3}$, and conversely any form $A_0X^3 + A_3 Y^3$ is stabilized by the action of $\gamma = \begin{pmatrix} \omega^{a_1} & 0\\ 0 & \omega^{2a_1}\end{pmatrix}$.

    So if $3\mid q-1$, then the only horizontally irreducible forms which could possibly have stabilizer $C_3$ are $\Gamma_0$-equivalent to a form $A_0X^3 + A_3Y^3$. The total number of such forms is bounded above by $O(|\Gamma_0|q^{2\ell})=O(q^{2\ell})$.

\noindent\textbf{Subcase 1(b)}: $3\nmid q-1$.
    
    In this case, (potentially after a change of basis) the action of $\gamma=\begin{pmatrix}0 & -1 \\ 1 & -1\end{pmatrix}$ sends $s$ to 
    \[
    \gamma(s) = A_0(-y)^3+A_1(-y)^2(x-y) + A_2(-y) (x-y)^2 + A_3 (x-y)^3.
    \]
    Solving $\gamma(s)=s$, we obtain $A_0 = A_3$ and $A_1+A_2 = - 3A_3$.

    So in this case, the number of forms with a $C_3$ stabilizer is also $O(q^{2\ell})$.

\noindent\textbf{Subcase 1(c):} $\Char \F_q = 3$. 

By the same argument as Subcase 1(b), the number of forms with a $C_3$ stabilizer is $O(q^{2\ell})$.

\noindent\textbf{Case 2:} $k\ge 1$.

    In this case, we can write $\Gamma_k \cong (\F_q)^{k+1} \rtimes (\F_q^{\times})^2$, thought of as matrices in
    \[
    \begin{pmatrix}
        \mc{O}^{\times}& 0\\
        \mc{O}(k) & \mc{O}^{\times} 
    \end{pmatrix}.
    \]
    Write an element $\gamma \in \Gamma_k$ as $\gamma=ng$, where $n:= \begin{pmatrix} 1 & 0 \\n & 1\end{pmatrix}$ and $g:= \begin{pmatrix} g_1 & 0 \\  0& g_2\end{pmatrix}$. The action of an element $g$ on an element $n$ sends $n$ to $gng^{-1} = \left(\frac{g_2}{g_1}\right)n$.

    In particular, we have $(ng)^3 = n'g^3$ where $n' = \left(1 + \left(\frac{g_2}{g_1}\right) + \left(\frac{g_2}{g_1}\right)^2 \right) n$. So if $\gamma = ng$ satisfies $\gamma^3 =1$, then $g_1, g_2$ are third roots of unity in $\F_q$. Furthermore, if $g_1=g_2$ and $\Char \F_q \neq 3$, then $n=0$ as well.

    On the other hand, if $g_1\neq g_2$ are distinct third roots of unity, then we can choose $m$ such that $n = \left(\frac{g_2}{g_1}\right)m-m$. Then
    $ng=m^{-1}gm$, and $\gamma=ng$ is conjugate to $g$.

    Finally, if $\Char \F_q = 3$, then if $\gamma = ng$ satisfies $\gamma^3=1$, we must have $g=1$. So $\gamma = \begin{pmatrix} 1 & 0 \\ n & 1 \end{pmatrix}$ for some $n$.

    We have three subcases:

\noindent\textbf{Subcase 2(a):} $\Char \F_q\neq 3$ and $g_1=g_2$.

    In this case we established that we must have $n=0$ and $\gamma = g$. If $g^3=1$ and $g\neq 1$, then we must have $3\mid q-1$, and $\gamma = \begin{pmatrix} \omega^{a} & 0 \\ 0 & \omega^{a}\end{pmatrix}$ for $a \in \{1,2\}$.

    Note that $\gamma$ acts on a form $s$ by the nonzero scalar $\omega^{a}$. So there are no nonzero forms $s$ fixed by $\gamma$.

\noindent\textbf{Subcase 2(b):} $\Char \F_q \neq 3$ and $g_1 \neq g_2$.

    Then $\gamma$ is conjugate to $g = \begin{pmatrix} \omega^{a_1} & 0 \\ 0 & \omega^{a_2} \end{pmatrix}$, for distinct $a_1, a_2 \in \{0,1,2\}$. (Again we must have $3\mid q-1$, or else we would have $\gamma=1$.)

    By a similar argument to Subcase 1(a), the only horizontally irreducible forms that could be fixed by $\gamma$ are $\Gamma_k$-equivalent to a form $A_0 X^3 + A_3 Y^3$. In total there are at most $O(|\Gamma_k| q^{2\ell -3k}) = O(q^{2\ell - 2k})$ many such forms.

\noindent\textbf{Subcase 2(c):} $\Char \F_q =3$. 

We have $\gamma = \begin{pmatrix} 1 & 0 \\ n & 1 \end{pmatrix}$ with $n\neq 0$. Suppose $\gamma$ fixes a nonzero section 
\[
s = A_0x^3 + A_1 x^2y + A_2 xy^2 + A_3 y^3.
\]
Then $\gamma$ acts by
\begin{align*}
\gamma\cdot s &= A_0 x^3 + A_1 x^2(nx+y) + A_2 x(nx+y)^2 + A_3 (nx+y)^3\\
&= (A_0 + nA_1 + n^2 A_2 + n^3 A_3)x^3 + (A_1 + 2nA_2)x^2y + (A_2)xy^2 + A_3 y^3.
\end{align*}
So if $\gamma\cdot s = s$, matching the $x^3$ and $x^2y$ coefficients yields $A_2 = 0$, and $A_1 =- n^2 A_3$.

Such a form $s = A_0 x^3 - n^2 A_3 x^2y + A_3 y^3$ is therefore determined by a choice of $A_0 \in \mc{O}(\ell), A_3 \in \mc{O}(\ell-3k)$, and nonzero $n \in \mc{O}(k)$. Since $\ell - 3k \ge 0$, there are at most $O(q^{2\ell - 2k})$ many such forms.

We find that the total number of horizontally irreducible forms $s \in H^0(F_k, \mc{O}(3,\ell))$ over all cases with a $C_3$ stabilizer is at most $O(|\Gamma_k| q^{2\ell-3k}) = O(|\Gamma_k|q^{N})$ for all values of $k$. Therefore the number of isomorphism classes of $C_3$ covers is at most
\[
3\sum_{k \le N/3} \frac{1}{|\Gamma_k|} O(|\Gamma_k| q^{N}) = O(Nq^{N}+1)
\]
as desired.

\end{proof}

\begin{prop}
    Let $N \ge 0$ and $C = \PP^1$. The number of isomorphism classes of inseparable covers $X\to C$ with discriminant degree $N$ is at most $O(1)$.
\end{prop}
\begin{proof}
If $f: X \to \PP^1$ is a degree $3$ finite morphism between smooth curves such that $\F_q(X)/\F_q(t)$ is inseparable, then in particular the extension is purely inseparable of degree $3$. (This can only occur when $\Char \F_q = 3$.) 

Therefore, up to isomorphism there is at most one inseparable cover $X\to \PP^1$. Concretely, when $\Char \F_q=3$ there is a unique isomorphism $X^{(3)} \cong \PP^1$ exhibiting $f$ as the relative Frobenius, as in the commuting diagram below.
\[
\begin{tikzcd}
    X \ar[d, "f"] \ar[dr, "F_{X/\F_q}"] &\\
    \PP^1 & X^{(3)}\ar{l}[swap]{\sim}
\end{tikzcd}
\]
(See \cite[\href{https://stacks.math.columbia.edu/tag/0CCY}{Tag 0CCY}]{stacks-project}.)
\end{proof}
\section{Appendix A: Discriminants}\label{appendix}

In this appendix, we discuss the construction of the discriminant of a binary cubic form, as well as the equivalence between discriminants of a binary cubic form and the corresponding triple cover.
\begin{lemma}
    Let $S$ be a scheme, and let $\mc{E}$ and $\mc{L}$ be a rank $2$ vector bundle and a line bundle on $S$. Let $g \in \Sym^{3} \mc{E} \otimes \mc{L}$. There is a well-defined \textbf{discriminant} 
    \[
    \Delta(g) \in (\wedge^2 \mc{E})^{\otimes 6} \otimes \mc{L}^{\otimes 4}.
    \]

    Furthermore, when $\mc{E} \cong \mc{O}_S x \oplus \mc{O}_S y$ and $\mc{L} \cong \mc{O}_S z$ are free $\mc{O}_S$-modules, this construction sends a form
    \[
    g = ax^3z + bx^2yz + cxy^2z + dy^3z \in \Sym^{3} \mc{E} \otimes \mc{L}
    \]
    to the ``classical'' discriminant
    \[
    (-27a^2d^2+18abcd-4ac^3-4b^3d+b^2c^2) (x\wedge y)^{\otimes 6} (z)^{\otimes 4}.
    \]
\end{lemma}
\begin{proof}
    Affine locally, on an open $\Spec R \subseteq S$ where $\mc{E}$ and $\mc{L}$ trivialize as $\mc{E}|_{R} \cong R x \oplus R y$ and $\mc{L}|_{R} \cong R z$, we can define $\Delta(f)$ by the the explicit ``classical'' formula. We can easily check that this is well-defined by seeing that if we computed the discriminant with respect to a different basis $x', y'$ of $\mc{E}|_{R}$ and $z'$ of $\mc{L}|_{R}$ (explicitly by applying a transformation by $\GL_2(R) \times \GL_1(R)$), we would get the same section of $(\wedge^{2}\mc{E})^{\otimes 6}\otimes \mc{L}^{\otimes 4}|_{R}$. So we can glue together the discriminant globally as a section of $(\wedge^{2}\mc{E})^{\otimes 6}\otimes \mc{L}^{\otimes 4}$.
\end{proof}
Therefore we have:
\begin{cor}
    If $g \in \Sym^{3}\mc{E} \otimes (\wedge^2 \mc{E})^{\vee}$, then there is a well-defined discriminant $\Delta(g) \in (\wedge^2 \mc{E})^{\otimes 2}$.
\end{cor}

Now we prove \Cref{prop:disc-preserve}, which says that on an integral scheme $S$, the discriminant of a binary cubic $g \in \Sym^{3} \mc{E} \otimes (\wedge^2\mc{E})^{\vee}$ is isomorphic to the discriminant triple cover $f: X \to S$ corresponding to $g$ by the equivalence $\CubForm^{\Prim}(S)\xrightarrow{\sim}\Trip^{\Gor}(S)$.

\begin{proof}[Proof of \Cref{prop:disc-preserve}]
Let $S$ be an integral scheme and $f:X \to S$ be a triple cover. Let $\mc{E}$ be the rank $2$ Tschirnhausen vector bundle on $S$ such that
\[
0 \to \mc{O}_S \to f_{*} \mc{O}_X \to \mc{E}^{\vee} \to 0
\]
and let $g \in \Sym^{3} \mc{E} \otimes (\wedge^2 \mc{E})^{\vee}$ be the associated (twisted) binary cubic form.

We start by noting there is a canonical isomorphism
\[
\wedge^2\mc{E}^{\vee} \cong \mc{O}_S\otimes \wedge^2 \mc{E}^{\vee} \xrightarrow{\sim} \wedge^3 f_{*} \mc{O}_X.
\]
(See \cite[\href{https://stacks.math.columbia.edu/tag/0B38}{Tag 0B38}]{stacks-project}.) Therefore we have a canonical isomorphism
\[
\delta: (\wedge^3 f_{*} \mc{O}_X)^{\otimes -1} \to (\wedge^2 \mc{E}^{\vee})^{\otimes -1} \cong \wedge^2 \mc{E}.
\]
We aim to show that $\delta^{\otimes 2}$ maps $\Delta_f \in \left(\wedge^3 f_{*} \mc{O}_{X}\right)^{\otimes -2}$ to $\Delta(g) \in (\wedge^2 \mc{E})^{\otimes 2}$.

Let's first assume that $\mc{E}$ trivializes as $\mc{E} \cong \mc{O}_S x \oplus \mc{O}_S y$. Write the form as 
\[
g = (ax^3+bx^2y + cxy^2 + dy^3)(x\wedge y)^{\otimes -1}
\]
where $a,b,c,d \in \mc{O}_S$.

Recall (for example, from \cite{Woo11}) that if $\mc{E}$ is free, we can concretely construct the associated triple cover $X\to S$
by 
\[
X \cong \Spec_S(\mc{O}_S\oplus \mc{E}^{\vee})
\]
where we explicitly define an $\mc{O}_S$-algebra structure on $f_{*}\mc{O}_X = \mc{O}_S \oplus \mc{E}^{\vee}$.

Let $(\omega, \theta)$ be the basis of $\mc{E}^{\vee}$ dual to $(x,y)$. Then the algebra structure on 
\[
\mc{O}_S \oplus \mc{E}^{\vee} \cong \mc{O}_S \oplus \mc{O}_S \omega \oplus \mc{O}_S \theta
\]
is defined by
\begin{align*}
    \omega \theta &= -ad\\
    \omega^2 &= -ac + b \omega - a\theta \\
    \theta^2 &= -bd + d \omega - c\theta.
\end{align*}
We can check that
\begin{align*}
    \tr(1)&= 3\\
    \tr(\omega)&=b\\
    \tr(\theta)&= -c\\
    \tr(\omega \theta) &= -3ad\\
    \tr(\omega^2) &= b^2-2ac\\
    \tr(\theta^2) &= c^2-2bd
\end{align*}
so the trace pairing sends
\begin{align*}
    1 & \to 3 +bx-cy\\
    \omega &\to b + (b^2-2ac)x -3ad y\\
    \theta & \to -c -3ad x + (c^2-2bd) y
\end{align*}
where we identified $1,x,y \in \mc{O}_S \oplus \mc{E}$ with the duals of $1,\omega, \theta \in \mc{O}_S \oplus \mc{E}^{\vee}$. The determinant map $\wedge^3 (\mc{O}_S \oplus \mc{E}^{\vee}) \to \wedge^3(\mc{O}_S \oplus \mc{E})$ sends
\[
(1\wedge \omega \wedge \theta)\to (b^2c^2 - 4ac^3 - 4b^3d + 18abcd - 27a^2d^2) (1 \wedge x \wedge y).
\]

Now the canonical isomorphism
\[
\wedge^2 \mc{E}^{\vee} \cong \mc{O}_S \otimes \wedge^2 \mc{E}^{\vee} \to \wedge^3 f_{*} \mc{O}_X
\]
is the map sending $\omega \wedge \theta \to 1 \wedge \omega \wedge \theta$. (This is a well-defined map because $1 \wedge \omega \wedge \theta$ does not change if $\omega$ and $\theta$ are shifted by elements of $\mc{O}_S$.) So the discriminant
\[
\Delta_f =(b^2c^2 - 4ac^3 - 4b^3d + 18abcd - 27a^2d^2) (1 \wedge \omega \wedge \theta)^{\otimes -2} \in (\wedge^{3}f_{*} \mc{O}_X)^{\otimes -2}
\]
corresponds to the discriminant
\begin{align*}
\Delta(g) &= (b^2c^2 - 4ac^3 - 4b^3d + 18abcd - 27a^2d^2) (\omega \wedge \theta)^{\otimes -2}\\
&= (b^2c^2 - 4ac^3 - 4b^3d + 18abcd - 27a^2d^2) (x \wedge y)^{\otimes 2}
\end{align*}
after identifying $(\omega \wedge \theta)^{\otimes -1} = x\wedge y$.

Now for a general rank $2$ vector bundle $\mc{E}$, the map $\delta^{\otimes 2}: (\wedge^3 f_{*} \mc{O}_X)^{\otimes -2} \to (\wedge^2 \mc{E})^{\otimes 2}$ is well-defined globally and sends $\Delta_f$ to $\Delta(g)$ on small opens where $\mc{E}$ trivializes, so $\delta^{\otimes 2}(\Delta_f) = \Delta(g)$ globally.

\end{proof}
\section{Appendix B: Horizontal Reducibility and Finiteness}
\label{sec:appendix-hor-finite}
In this appendix, we show various equivalent characterizations of horizontal reducibility. Then we show that various sets of horizontally irreducible sections are nonempty only when certain numerical degree conditions hold. This helps us note that certain sums are finite.

We start by proving \Cref{prop:horiz-ir}, which we reproduce here.

\begin{prop}
    Let $C$ be a nice curve. Let $\mc{E}$ be a rank $2$ vector bundle on $C$, and let $\mc{L}$ be a line bundle on $C$.
    Let $s \in H^0(\PP(\mc{E}), \mc{O}_{\PP(\mc{E})}(3) \otimes \pi^{*} \mc{L})$. Let $X_0: = V(s) \subseteq \PP(\mc{E})$ be the vanishing scheme of $s$. The following are equivalent to $s$ being \textbf{horizontally reducible}:
    \begin{itemize}
    \item Either $s=0$ or there is an integral component $Y_i$ of $X_0$ such that $\mc{O}(Y_i) \cong \mc{O}_{\PP(\mc{E})}(1)\otimes \pi^{*} \mc{L}_i$ for some line bundle $\mc{L}_i$.
    \item There is a section $C_1\subseteq \PP(\mc{E})$ of $\pi: \PP(\mc{E}) \to C$ on which $s$ vanishes.
    \item The form $\Lambda(s_{\eta})$ associated to the pullback of $s$ to the generic point $\eta \in C$ is reducible as a binary cubic form over a field.
    \item The form $\Lambda(s) \in \Sym^3 \mc{E} \otimes \mc{L}$ is reducible after base change to the generic point of $C$.

    \end{itemize}

    Furthermore, the following is equivalent to $s$ being horizontally irreducible:

    \begin{itemize}
        \item $s\neq 0$ and the finite map $(X_0)_{\eta} \to \eta$ is a degree $3$ extension of fields.
    \end{itemize}
\end{prop}
\begin{proof}
    Recall our original definition of horizontal reducibility for a nonzero section $s \in \PP(\mc{E})$. The vanishing scheme $X_0:= V(s) \subseteq \PP(\mc{E})$ is an effective Cartier divisor, to which we can associate a Weil divisor
    \[
    [X_0] = \sum a_i [Y_i]
    \]
    where each $Y_i$ is an integral $1$-dimensional closed subscheme of $\PP(\mc{E})$, and $a_i >0$ for each $i$ in the sum.

    For each component $Y_i$ of $X_0$, we can write $\mc{O}(Y_i) \cong \mc{O}_{\PP(\mc{E})}(m_i) \otimes \pi^{*} \mc{L}$ for some line bundle $\mc{L}$ on $C$ and some nonnegative integer $m_i$. Let's call $m_i$ the vertical degree of $Y_i$. 
    
    So $Y_i$ is horizontal if and only if $m_i\ge 1$. We defined $s$ to be horizontally irreducible if and only if $X_0$ has one horizontal component $Y_i$ with multiplicity $a_i=1$.

    Let's assume from now on that $s$ is nonzero.

    To prove the first condition, it suffices to note that if $s$ is horizontally reducible, then at least one of the horizontal components of $X_0$ must have vertical degree $1$. If there are two horizontal components (possibly the same component with multiplicity) each of vertical degree at least $2$, then the vertical degree of $X_0$ is at least $3$, a contradiction.

    To prove the second condition, we just need to show that an integral (horizontal) 1-dimensional closed subscheme $Y_i \subseteq \PP(\mc{E})$ of vertical degree $1$ must be a section of $C$. Because $Y_i$ is integral, it is cut out by a section $s \in \mc{O}_{\PP(\mc{E})}(1)\otimes \pi^{*} \mc{L}$ which does not vanish on any fiber. The section $s$ corresponds to a nonzero (and therefore injective) map $0 \to \mc{L}^{\vee} \xrightarrow{f} \mc{E}$ with the property that $f|_{\mc{P}}$ is nonzero at every closed point $P \in C$. Consider the dual map 
    \[
    \mc{E} \cong \mc{E}^{\vee} \otimes \wedge^2 \mc{E} \xrightarrow{f^{\vee} \otimes \Id} \mc{L} \otimes \wedge^2 \mc{E}.
    \]

    Now $f^{\vee}|_{P} = f|_{P}^{\vee}$ is nonzero for each $P \in C$, so $f^{\vee}|_{P}$ is surjective for each $P$, and therefore $\mc{E}_{P} \to (\mc{L}\otimes \wedge^2 \mc{E})_{P}$
    is a surjection on fibers (by Nakayama). So $\mc{E} \to \mc{L}\otimes \wedge^2 \mc{E}$ is a surjection, implying that the Koszul sequence
    \[
    0 \to \mc{L}^{\vee} \to \mc{E} \to \mc{L} \otimes \wedge^2 \mc{E} \to 0
    \]
    is exact. Therefore the surjection on the right induces a section $C\to \PP(\mc{E})$ which is cut out by the section $s$ corresponding to the map $f: \mc{L}^{\vee} \to \mc{E}$.

    Now let $U\subseteq C$ be an open subscheme where $X_0$ has no vertical components. So as Weil divisors we now have
    \[
    [(X_0)_{U}] = \sum a_i [(Y_i)_{U}]
    \]
    where all $Y_i$ are horizontal. So now $X_0$ is horizontally irreducible if and only if $(X_0)_{U}$ is integral.

    Further note that $(X_0)_{\eta} \to \eta$ is the scheme cut out by the restriction of $s$ to $\PP^1_{\eta}$. Note that $s_{\eta} \in H^0(\PP^1_{\eta}, \mc{O}(3))$, and if $s$ is nonzero, then $s_{\eta}$ is nonzero.

    If $X_0$ is horizontally irreducible and $(X_0)_{U}$ is integral, then since $(X_0)_{U} \to U$ is dominant, the generic point of $(X_0)_{U}$ lies over $\eta$. So $(X_0)_{\eta} \to \eta$ is an extension of fields. This implies that the form $\Lambda(s_{\eta})$, thought of as the binary cubic form cutting out $(X_0)_{\eta} \subseteq \PP^1_{\eta}$, must be irreducible.

    If $X_0$ is horizontally reducible, then $(X_0)_{\eta}$ is the pullback of $\sum a_i [Y_i]$ to $\eta$, where $Y_i$ are horizontal components. The divisor $\sum a_i[{Y_i}_{\eta}]$ is either reducible or non-reduced. So the associated nonzero form $s_{\eta}$ is reducible. Therefore $(X_0)_{\eta} \to \eta$ is not an extension of fields.

    Finally note that $\Lambda(s_{\eta}) = \Lambda(s)_{\eta}$, so the conditions on both are the same.
\end{proof}

\begin{prop}\label{prop:nonempty-3g}
    Let $C$ be a nice curve and let $\mc{E}$ be a rank $2$ vector bundle on $C$.
    
    If $\mc{V}(\mc{E})^{\ir}$ is nonempty, then $\deg \mc{E} \ge -3g_C$. Furthermore, for any fixed integer $N \ge -3g_C$, there are finitely many isomorphism classes of rank $2$ vector bundles $\mc{E}'$ on on $C$ such that $\deg \mc{E}' = N$ and $\mc{V}(\mc{E})^{\ir}$ is nonempty.
\end{prop}
\begin{proof}
    By our previous characterization, a section $s \in \mc{V}(\mc{E})$ is horizontally reducible if there exists a section $C_1 \subseteq \PP(\mc{E})$ of $\pi: \PP(\mc{E}) \to C$ where $s$ vanishes.

    Let $\mc{E} \twoheadrightarrow \mc{M}$ be a surjection onto a line bundle. Let $s \in \mc{V}(\mc{E})$ correspond to a binary cubic form $\Lambda(s) \in \Sym^{3} \mc{E} \otimes (\wedge^2 \mc{E})^{\vee}$.

    The surjection induces a section $C \to \PP(\mc{E})$ with image $C_1\subseteq \PP(\mc{E})$. The restriction of $\mc{O}_{\PP(\mc{E}}(1)$ to $C_1\cong C$ is $\mc{M}$, so the restriction of $\mc{O}_{\PP(\mc{E})}(3) \otimes \pi^{*}(\wedge^2 \mc{E})^{\vee}$ to $C_1\cong C$ is $\mc{M}^{\otimes 3} \otimes (\wedge^2 \mc{E})^{\vee}$.

    One can check that the restriction sends a global section $s \in \mc{V}(\mc{E})$ to the image of $\Lambda(s)$ via the map
    \[
    \Sym^{3}\mc{E} \otimes (\wedge^2 \mc{E})^{\vee} \to \mc{M}^{\otimes 3}\otimes  (\wedge^2 \mc{E})^{\vee}.
    \]

    So if $\mc{V}(\mc{E})^{\ir}$ is nonempty, then for every line bundle $\mc{M}$ with a surjection $\mc{E} \to \mc{M}$, we have $3\deg \mc{M} - \deg \mc{E} \ge 0$.

    Now for any rank $2$ vector bundle $\mc{E}$, there is a pair of line bundles $\mc{N}$, $\mc{M}$ along with a short exact sequence
    \[
    0 \to \mc{N} \to \mc{E} \to \mc{M} \to 0
    \]
    such that $\deg \mc{N} - \deg \mc{M} = e$, the \textbf{skew degree} of $\mc{E}$. Following \cite{Gun17}, we note that $e\ge -g_C$, where $g_C$ is the genus of $C$.

    So we know that $2 \deg \mc{M} - \deg \mc{N} \ge 0$ and $\deg \mc{N} - \deg \mc{M} \ge -g_C$. This implies $deg \mc{M} \ge -g_C$ and $\deg \mc{N} \ge -2g_C$. In particular we must have $\deg \mc{E} \ge -3g_C$, proving the first part of the proposition.

    To prove the second part, fix an integer $N\ge -3g_C$ and consider the vector bundles $\mc{E}$ with $\deg \mc{E} = N$ and $\mc{V}(\mc{E})^{\ir}$ nonempty. Each such $\mc{E}$ determine a pair $(m,n):= (\deg \mc{M}, \deg \mc{N})$ satisfying $m+n=N$, $m \ge -g_C$, and $n\ge -g_C$. There are finitely many pairs of integers satisfying the three conditions.

    For each pair of integers $(m,n)$ there are finitely many pairs $(\mc{M}, \mc{N})$ with degrees $(m,n)$ (since $\Pic^{d}(C)$ is finite for all $d$). Finally, for a fixed pair $(\mc{M}, \mc{N}$), there are at most $|\Ext^{1}(\mc{M}, \mc{N})|<\infty$ many extension classes $0 \to \mc{N} \to \mc{E} \to \mc{N} \to 0$, which uniquely determine $\mc{E}$ up to isomorphism. So in total there are finitely many isomorphism classes of rank $2$ vector bundles $\mc{E}$ with degree $N$ and $\mc{V}(\mc{E})^{\ir}$ nonvanishing.
\end{proof}
\begin{prop}
    Let $D$ be a reduced effective divisor on $C$. If $\Bad(\mc{E}, D)^{\ir}$ is nonempty, then $\deg \mc{E} - \deg D\ge -3g_C$. 
\end{prop}
\begin{proof}
    Assume $\deg \mc{E} - \deg D < -3g_C$. We will prove $|\Bad(\mc{E}, D)^{\ir}| = 0$.

    We can write
    \[
    |\Bad(\mc{E}, D)^{\ir}| = \sum_{D_1+ D_2+ D_3 = D} \mu(D_2) |\Sing(\mc{E}, \DD_1) \cap \SingFib(\mc{E}, \DD_2) \cap \Fib(\mc{E}, D_3)^{\ir}|.
    \]
    Now we can write each term as
    \[
    |\Sing(\mc{E}, \DD_1) \cap \SingFib(\mc{E}, \DD_2) \cap \Fib(\mc{E}, D_3)^{\ir}| = |\Sing(\mc{E'}, \DD_1' \cap \Van(\mc{E}', \DD_2')^{\ir}|
    \]
    for $\mc{E}' = \mc{E}(-D_2 -D_3)$.

    Finally, there exists $\mc{E}''$ with $\deg \mc{E}'' = \deg \mc{E}-D_1-2D_2-2D_3$ such that
    \[
    |\Sing(\mc{E'}, \DD_1') \cap \Van(\mc{E}', \DD_2')^{\ir}| = |\Van(\mc{E}'', \DD_1'') \cap \Van(\mc{E}'', \DD_2'')^{\ir}|
    \]

    Now $\deg \mc{E}'' \le \deg \mc{E} - \deg D< -3g_C$, so $|\mc{V}(\mc{E}'')^{\ir}|=0$. Therefore $|\Bad(\mc{E}, D)^{\ir}| = 0$.
\end{proof}

Finally, if we assume that either $\Char \F_q \neq 3$ or $C\cong \PP^1$ holds, then we can improve the bound in \Cref{prop:nonempty-3g}.
\begin{prop}
    Assume either $\Char \F_q \neq 3$ or $C \cong \PP^1$. Let $\mc{E}$ be a rank $2$ vector bundle on $C$. If $|\mc{V}(\mc{E})^{\ir}| \neq 0$, then $\deg \mc{E} \ge 0$. Furthermore, if $\Theta(C,N) \neq 0$, then $N \ge 0$.
\end{prop}
\begin{proof}
    If $C \cong \PP^1$, then $s \in \mc{V}(\mc{E})$ can be written as $\Lambda(s) = A_0 x^3 + A_1 x^2y + A_2 xy^2 + A_3 y^3$. If $\deg \mc{E} <0$, then $A_3<0$ and $s$ is horizontally reducible.

    Now assume $\Char \F_q \neq 3$. Then any horizontally irreducible section is associated to a separable field extension. In particular, the base change of the discriminant to $\F_q(C)$ is nonzero, so the discriminant $\Delta(s)$ is a nonzero section of $(\wedge^2 \mc{E})^{\otimes 2}$. Therefore we must have $\deg \mc{E} \ge 0$.

The second part follows directly from the first, recalling our original definition
\[
\Theta(C,N) = \sum_{\substack{\mc{E} \in |\Bun(C)| \\ \deg \mc{E} = N} }\frac{1}{|\Aut \mc{E}|} |\mc{V}(\mc{E})^{\ir,\sm}|.
\]
\end{proof}
\section{Appendix C: Elementary Transformations, Blowups}
\label{sec:appendix-a}
In this appendix, we prove various propositions about the elementary transformation functor and our blowup construction.

Here we prove \Cref{prop:elm-groupoid-equiv}. We reproduce the statement here:
\begin{prop}
    Let $S$ be a Dedekind scheme. Let $\mc{N}$ be a line bundle on $S$, and let $D$ be a reduced effective divisor. The elementary transform functor $\ElmD$ induces an equivalence of groupoids
    \[
    \ElmD: \DBun_{\mc{N}} \to \DBun_{\mc{N}(-D)}.
    \]
\end{prop}
\begin{proof}[Proof of \Cref{prop:elm-groupoid-equiv}]
First, we show that if 
\[
0 \to \E' \xrightarrow{f} \E \to i_{*} \LL \to 0
\]
is an exact sequence of sheaves on $S$ with $\E, \E'$ rank $2$ vector bundles on $S$ and $\LL$ a line bundle on $D$, then 
\begin{equation}\label{eq:wedge-exact}
0 \to \wedge^2\mc{E}' \xrightarrow{\wedge^2 f} \wedge^2 \mc{E} \to i_{*}(\wedge^2 \mc{E})|_{D} \to 0
\end{equation}
is a short exact sequence. This would give in particular a canonical isomorphism
\[
\wedge^2 \E' \cong \wedge^2 \E \otimes \mc{O}(-D).
\]
We can check exactness of \eqref{eq:wedge-exact} at the stalks. For a stalk at $P$ away from $D$, $f$ is an isomorphism so $\wedge^2 f$ is an isomorphism as well.

Now consider a stalk at $P \in D$. The local ring $\mc{O}_P$ is a DVR and in particular a PID, so after identifying $\E'_P \cong \mc{O}_{P}^{\oplus 2}$ and $\E_P \cong \mc{O}_P^{\oplus 2}$ and picking appropriate bases, we can write $f_P: \E'_{P} \to \E_P$ in Smith Normal Form as
\[
f_P = \begin{pmatrix}
    \pi^{\lambda_1} & 0 \\ 0 & \pi^{\lambda_2}
\end{pmatrix}
\]
where $\pi$ is a uniformizer at $P$ and $0\le \lambda_1\le \lambda_2 $. Then since $\coker(f_P) \cong \mc{O}_P/\mathfrak{m}_P$, we must have $(\lambda_1, \lambda_2) = (0,1)$ and the map $\wedge^2 f_P$ is multiplication by $\pi$.

So at the stalk $P$ our sequence is isomorphic to
\[
0 \to \mc{O}_P \xrightarrow{\times \pi} \mc{O}_P \to \mc{O}_{P}/(\pi) \to 0
\]
which is exact. Therefore \eqref{eq:wedge-exact} is exact, so $\ElmD$ restricts to a well-defined functor $\DBun_{\mc{N}} \to \DBun_{\mc{N}(-D)}$.

Next, we need to show this is an equivalence of categories. 

First, we will note that for a $D$-marking $\phi_1: \mc{E}_1 \to i_{*} \mc{L}_1$ with $\ker(\phi_1) = \mc{E}_2$, we can write a short exact sequence
\[
0 \to \mc{E}_2 \xrightarrow{f_1} \mc{E}_1 \to i_{*} \mc{L}_1 \to 0.
\]
Now $\phi_1$ determines $f_1$ up to isomorphism and vice versa, so it's not hard to see that (yet another) equivalent characterization of $\DBun$ is the groupoid with objects given by injective maps $\mc{E}_2 \xrightarrow{f_1} \mc{E}_1$ with cokernel isomorphic to $i_{*} \mc{L}$ for some $\mc{L}$. An morphism from $(\mc{E}_2 \xrightarrow{f_1} \mc{E}_1)$ to $(\mc{E}_2' \xrightarrow{f_1'} \mc{E}_1')$ is a pair of isomorphisms $\mc{E}_2 \xrightarrow{\sim} \mc{E}_2'$ and $\mc{E}_1 \xrightarrow{\sim} \mc{E}_1'$ commuting with $f_1$ and $f_1'$.

Now we consider the application of $\ElmD$ twice. We first take $\mc{L}_2 = \ker(\mc{E}_1|_{D} \to \mc{L}_1)$ and let $\phi_2: \mc{E}_2 \to i_{*}\mc{L}_2$ be the induced morphism. Let $\mc{E}_3 \xrightarrow{f_2} \mc{E}_2$ be the kernel of $\phi_2$, so $\ElmD(\mc{E}_2 \xrightarrow{f_1} \mc{E}_1) = (\mc{E}_3 \xrightarrow{f_2} \mc{E}_2)$.

Similarly define $\mc{L}_3 = \ker(\mc{E}_2|_{D} \to \mc{L}_2)$, $\phi_3: \mc{E}_3 \to i_{*}\mc{L}_3$, and $(\mc{E}_4 \xrightarrow{f_3} \mc{E}_3) = \ElmD(\mc{E}_3\xrightarrow{f_2} \mc{E}_2)$.

We claim that there is a natural isomorphism
\[
(\mc{E}_4 \xrightarrow{f_3} \mc{E}_3) =(\Elm_D \circ \Elm_D)(\mc{E}_2 \xrightarrow{f_2} \mc{E}_1)\cong \left(\mc{E}_2(-D) \xrightarrow{f_1(-D)} \mc{E}_1(-D)\right).
\]

First, consider the diagram
\[
\begin{tikzcd}
& \mc{E}_3 \ar{r}{\sim} \ar{d} & \mc{E}_1(-D) \ar{r} \ar{d} & 0 \ar{d}\\
   0 \ar{r} &  \mc{E}_2 \ar{r}\ar{d} &\mc{E}_1 \ar{r}\ar{d} & \mc{L}_1 \ar{r} \ar{d}& 0\\
   0 \ar{r}&  \mc{L}_2 \ar{r}\ar{d} & \mc{E}_1|_{D} \ar{r} \ar{d}& \mc{L}_1 \ar{r} \ar{d}& 0\\
   & 0 \ar{r} & 0 \ar{r} & 0
\end{tikzcd}
\]
where the two middle rows are exact, and so the top left map $\mc{E}_3 \to \mc{E}_1(-D)$ is an isomorphism by (say) the snake lemma. Above we write $\mc{L}_1$ for the sheaf $i_{*} \mc{L}_1$ and $\mc{E}_1|_{D}$ for $i_{*} \mc{E}_1|_{D}$, by a slight abuse of notation.

Now consider the diagram
\[
\begin{tikzcd}
0 \ar{d} & 0 \ar{d} & 0 \ar{d}& 0 \ar{d} \\
\mc{E}_4 \ar{r}{\sim}\ar{d}{f_3} & \mc{E}_2(-D) \ar{r}\ar{d} & \mc{E}_3 \ar{r}{\sim} \ar{d}{f_2}& \mc{E}_1(-D)\ar{d} &\\
\mc{E}_3 \ar{r}{f_2}\ar{d}& \mc{E}_2 \ar[equal]{r} \ar{d}& \mc{E}_2 \ar{r}{f_1}\ar{d} & \mc{E}_1\ar{d} &\\
\mc{L}_3 \ar{r} \ar{d}& \mc{E}_2|_{D} \ar{r}\ar{d} &  \mc{L}_2 \ar{r}\ar{d} & \mc{E}_1|_{D}\ar{d}& \\
0 & 0 & 0 &0
\end{tikzcd}
\]
where the four vertical columns are short exact sequences. We first claim that the composite map $\mc{E}_4 \to \mc{E}_2(-D) \to \mc{E}_3$ along the first row is the same as the vertical map $\mc{E}_4 \xrightarrow{f_3} \mc{E}_3$ in the first column.

Indeed, by isolating the first and third columns we get a commutative diagram
\[
\begin{tikzcd}
    0 \ar{d} & 0 \ar{d} \\
    \mc{E}_4\ar{r} \ar{d}{f_3}&\mc{E}_3\ar{d}{f_2}\\
    \mc{E}_3 \ar{r}{f_2} \ar{d} & \mc{E}_2 \ar{d}\\
    \mc{L}_3 \ar{r}{0} \ar{d} & \mc{L}_2 \ar{d} \\
    0 & 0
\end{tikzcd}
\]
and we see that $\mc{E}_4 \xrightarrow{f_3} \mc{E}_3$ is the unique map in the first row that makes the diagram commute.

On the other hand, by isolating the second and fourth columns we see that the composite first row map $\mc{E}_2(-D) \to \mc{E}_3 \to \mc{E}_1(-D)$ must be the map $f_1(-D) = f_1 \otimes \mc{O}(-D)$. So by looking at the first row, we get a commutative diagram
\[
\begin{tikzcd}
    \mc{E}_4\ar{r}{f_3}\ar{d}{\sim} & \mc{E}_3 \ar{d}{\sim} \\
    \mc{E}_2(-D) \ar{r}{f_1(-D)} \ar[dashed]{ru}& \mc{E}_1(-D)
\end{tikzcd}
\]
as desired.

One can check that this construction is natural. So we get a natural isomorphism between $\ElmD \circ \ElmD$ and the functor $\otimes \mc{O}(-D)$, which sends
\[
(\mc{E}_2\xrightarrow{f_1} \mc{E}_1) \to (\mc{E}_2(-D) \xrightarrow{f_1(-D)} \mc{E}_1(-D)).
\]
In particular $(\otimes \mc{O}(D))\circ \ElmD$ is an inverse to $\ElmD$, which must be an equivalence of categories.
\end{proof}

Now we prove \Cref{prop:blow-up-down}. We reproduce the statement here.

\begin{prop}\label{prop:blow-up-down-appendix}
    Let $S$ be a Dedekind scheme, and let $D$ be a reduced effective divisor on $S$. Let $\E_1$ be a rank $2$ vector bundle on $S$, with associated projective bundle $\PP(\E_1)\to S$, and let $\mc{D}_1\subset \PP(\E_1)$ be a $D$-marking.

    Let $(\E_2, \mc{D}_2)$ be the image of $(\E_1, \mc{D}_1)$ by $\ElmD$. Let $\pi_2: \PP(\E_2)\to S$ be the associated projective bundle.

    Then the blowup of $\PP(\E_1)$ at $\DD_1$ is the blowup of $\PP(\E_2)$ at $\DD_2$.

    More precisely: there exists a scheme $Y$ with $r: Y \to S$ along with maps $f_1: Y \to \PP(\E_1)$, $f_2: Y \to \PP(\E_2)$ such that:
    \begin{itemize}
    \item 
    The following diagram commutes:
\[
\begin{tikzcd}[column sep=small]
& Y \arrow[dl, "f_1"] \arrow[dr, "f_2"]\arrow[dd,"r"] & \\
\PP(\E_1) \arrow[rd,"\pi_1"] & & \PP(\E_2)\arrow[ld, "\pi_2"]\\
&S & 
\end{tikzcd}
\]
\item We have that $f_1: Y\to \PP(\E_1)$ exhibits $Y$ as the blowup of $\PP(\E_1)$ at $\DD_1$, and $f_2: Y \to \PP(\E_2)$ exhibits $Y$ as the blowup of $\PP(\E_2)$ at $\DD_2$.
\item Let $E_1\subseteq Y$ be the exceptional divisor of $f_1$, and let $E_2\subseteq Y$ be the exceptional divisor of $f_2$. Then $E_2 = \pi_1^{-1}(D)^{st}$ and $E_1 = \pi_2^{-1}(D)^{st}$.
\item We have
\[
f_1^{*} \mc{O}_{\PP(\E_1)}(1) \otimes \mc{O}_{Y}(-E_1) \cong f_2^{*}\mc{O}_{\PP(\E_2)}(1).
\]
\end{itemize}
\end{prop}
\begin{proof}
    
    Our strategy will be to define the blowup $Y$ of $\PP(E_1)$ at $\DD_1$, then to construct a map $Y\to \PP(\E_2)$ by universal property using the universal property, and finally to show that this constructed map is the blowup of $\PP(\E_2)$ along $\DD_2$.

    We remark that $\PP(\mc{E}_1), \PP(\mc{E}_2)$ and $Y:=\Bl_{\DD_1} \PP(\mc{E}_1)$ are all regular, integral, $2$-dimensional schemes that are projective over $S$.

    Let $i: D \to S$ be the inclusion map. Let $X_1:= \PP(\mc{E}_1)$ and let $X_2:= \PP(\mc{E}_2)$. Recall that $\E_2$ is defined as the kernel fitting in the exact sequence
    \begin{equation}\label{eq:struct-seq}
    0 \to \E_2 \to \E_1 \xrightarrow{\phi_1} i_{*} \mc{O}_D\to 0,
    \end{equation}
    where $\phi_1: \E_1 \to i_{*} \mc{O}_D$ is the morphism associated to the $D$-marking $\DD_1\subseteq X_1.$

    Let $j: D \to X_1$ be the section associated with the $D$-marking $\DD_1\subseteq X_1$, meaning $j$ is a closed embedding with image $\DD_1$, and $i = \pi \circ j$.

    On $X_1=\PP(\E_1)$ there is an exact sequence of sheaves associated to the closed subscheme $\DD_1 \subseteq X_1$:
    \[
    0 \to \mc{I}_{\DD_1} \to \mc{O}_{X_1} \to j_{*} \mc{O}_D\to 0.
    \]
    After twisting by $\mc{O}_{X_1}(1)$, we get an exact sequence
    \[
    0 \to \mc{I}_{\DD_1}(1) \to \mc{O}_{X_1}(1)\to j_{*} \mc{O}_D \to 0
    \]
    after choosing an isomorphism $\mc{O}_D \otimes j^{*} \mc{O}(1) \cong \mc{O}_D$, since $\mc{O}_D$ is a PID.

    Let us define the coherent sheaf $\mc{K}:= \mc{I}_{\DD_1}(1)$ on $X_1$, so we can write this exact sequence as 
    \begin{equation} \label{eq:X1-seq}
    0 \to \mc{K}\to \mc{O}_{X_1}(1) \to j_{*} \mc{O}_D\to 0.
    \end{equation}

    If we push forward by $\pi_1$, we get an exact sequence
    \begin{equation}\label{eq:base-ses}
    0 \to (\pi_{1})_{*}\mc{K}\to \E_1\xrightarrow{\phi_1} i_{*} \mc{O}_D\to 0 
    \end{equation}
    which is the same as \eqref{eq:struct-seq} by \Cref{lem:taut-push}. (In particular, the sequence is exact on the right because we already know $\phi_1$ to be surjective.) This induces an isomorphism $(\pi_1)_{*} \mc{K} \cong \mc{E}_2$. 

    On the other hand, $\pi_1$ is flat, so we can pull back \eqref{eq:base-ses} to get an exact sequence
    \[
    0 \to \pi_1^{*}\E_2 \to \pi_1^{*} \E_1 \to \pi_1^{*} i_{*} \mc{O}_D \to 0
    \] on $X_1$, giving a commutative diagram
    \[
    \begin{tikzcd}0 \ar[r] & \pi_1^{*} \E_2 \ar[r]\ar[d] & \pi_1^{*} \E_1 \ar[r]\ar[d] &\pi_1^{*} i_{*} \mc{O}_D \ar[r]\ar[d] & 0 \\
    0 \ar[r] & \mc{K} \ar[r]& \mc{O}_{X_1}(1) \ar[r]& j_{*} \mc{O}_D\ar[r]& 0.
    \end{tikzcd}
    \]

    Let $U = S - D$. We have an open subset $(X_1)_{U} = \pi_1^{-1}(U) = X_1 - \pi_1^{-1}(D)$. On this open subset of $X_1$, the first vertical map is surjective, and in fact isomorphic to the second vertical map. So we get an induced isomorphism $(X_1)_{U} \xrightarrow{\sim} (X_2)_{U}$.
    
    Now consider the blowup $Y:= \Bl_{\DD_1} X_1$, with structure morphism $f_1: Y \to X_1$. Let $E_1$ be the exceptional divisor. Define $r: Y \to S$ by $\pi_1 \circ f_1$.

    We define the line bundle 
    \[
    \mc{M}:= f_1^{*} \mc{O}_{X_1}(1) \otimes \mc{O}_Y(-E_1)
    \]
    on the blowup $Y$. We will prove that $\mc{M}$ is relatively globally generated over $S$, and therefore it will define our morphism to $X_2 = \PP(\mc{E}_2)$. It will also extend the isomorphism $(X_1)|_{U} \xrightarrow{\sim} X_2|_{U}$.

    Consider again the exact sequence
    \[
    0 \to \mc{I}_{\DD_1} \to \mc{O}_{X_1} \to j_{*} \mc{O}_D \to 0.
    \]

    Then we have an exact sequence
    \[
    0 \to (f_1^{-1} \mc{I}_{D_1})\cdot \mc{O}_Y \to \mc{O}_Y \to f_1^{*} j_{*} \mc{O}_D\to 0.
    \]
    Since $E_1$ is the pullback of $D_1$ along $f_1$, we have $(f_1^{-1} \mc{I}_{D_1})\cdot \mc{O}_Y = \mc{I}_{E_1} = \mc{O}_{Y}(-E_1)$, and this is the same sequence as
    \[
    0 \to O_Y(-E_1) \to \mc{O}_Y \to f_1^{*} j_{*} \mc{O}_D \to 0 
    \]
    on $Y$, where $f_1^{*}j_{*}\mc{O}_D$ is isomorphic to the inclusion of the structure sheaf $\mc{O}_{E_1}$ in $Y$.

    If we twist by $f_1^{*} \mc{O}_{X_1}(1)$, we get the exact sequence
    \begin{equation}\label{eq:Y-ses}
        0 \to \mc{M} \to f_1^{*} \mc{O}_{X_1}(1) \to f_1^{*} j_{*} \mc{O}_D \to 0.
    \end{equation}
    on $Y$.
    
    Note that ${f_1}_{*} \mc{O}_Y = \mc{O}_{X_1}$. This follows since $Y \to X_1$ is a birational projective morphism of Noetherian integral schemes and $X_1$ is normal (see \cite[\S III.11]{Har77}.)

    Therefore, if we push forward \eqref{eq:Y-ses} by ${f_{1}}_{*}$ to $X_1$, we get by the projection formula
    \[
    0 \to \mc{I}_{\DD_1}\otimes \mc{O}_{X_1}(1) \to \mc{O}_{X_1}(1) \to j_{*} \mc{O}_D \to 0.
    \]
    using the fact that ${f_1}_{*} \mc{O}_{Y}(-E_1) = \mc{I}_{\DD_1}$. This is the same exact sequence as \eqref{eq:X1-seq}. Therefore the composition $r_{*}={\pi_1}_{*}{f_{1}}_{*}$ sends the sequence \eqref{eq:Y-ses} to \eqref{eq:struct-seq}. 

    We also claim $Y\to S$ is flat. We will show later that each fiber above a closed point $P \in S$ is either isomorphic to $\PP^1_{P}$ or the pushout $F \cup_{\mc{P'}} E$ of two copies of $\PP^1_P$ glued at a point $\mc{P}' \cong P$. Then since $S$ is regular and purely $1$ dimensional, $Y$ is Cohen-Macaulay (in fact regular) and purely $2$-dimensional, $r: Y \to S$ is proper (and so sends closed points to closed points), and the fibers have pure dimension $1$, we conclude by Miracle Flatness that $Y\to S$ is flat (see \cite{EGA}).
    
    As such there is an induced commutative diagram
    \begin{equation}\label{eq:commut-const}
    \begin{tikzcd}
    0 \ar[r] &r^{*} \E_2 \ar[r] \ar[d, "\alpha_\mc{M}"] &r^* \E_1 \ar[r]\ar[d]&r^{*} \mc{O}_D \ar[r]\ar[d] &0\\
    0 \ar[r] & \mc{M} \ar[r] &f_1^{*} \mc{O}_{X_1}(1) \ar[r]& f_1^{*} j_{*} \mc{O}_{D}\ar[r]& 0 
    \end{tikzcd}
    \end{equation}
    where we can identify the first vertical map with $\alpha_{\mc{M}}: r^{*} r_{*} \mc{M} \to \mc{M}$. 

    Remark that the first two columns become isomorphic on the restriction to $U$, so 
    \[
    r^{*} r_{*} \mc{M}|_{U} \to \mc{M}|_{U}
    \]
    is the isomorphism $(X_1)_{U} \cong Y_{U} \xrightarrow{\sim} (X_2)_{U}$.

    So far we have constructed the diagram
    \[
    \begin{tikzcd}[column sep=small]
    & Y \arrow[dl, "f_1"] \arrow[dd,"r"] & \\
    \PP(\E_1) \arrow[rd,"\pi_1"] & & \PP(\E_2)\arrow[ld, "\pi_2"]\\
    &S. & 
    \end{tikzcd}
    \]
    We would like to construct a morphism $Y\to \PP(\mc{E}_2)$ by universal property from the map $\alpha_{\mc{M}}: r^{*} \mc{E}_2 \to \mc{M}$. We need to show this map is surjective - that is, we need to show that the map $r^{*}r_{*} \mc{M}\to \mc{M}$ is surjective, or that $\mc{M}$ is relatively globally generated.

    We will consider the map $r^{*}r_{*} \mc{M} \to \mc{M}$ fiber by fiber. Let $P \in S$ and let $Y_P$ be the fiber above $P$. Let $\ell: Y_P \to Y$ be the inclusion of the fiber. We show that $\ell^{*} r^{*} r_{*} \mc{M} \to \ell^{*} \mc{M}$ is surjective.

    We would like to identify $(r_{*}\mc{M})|_{P} $ with $H^0(Y_P, \mc{M}_P)$ using cohomology and base change. We showed earlier that $r: Y \to S$ is flat, so we need to show that $H^1(Y_P, \mc{M}_P)=0$ for all fibers.

    We'll start with a fiber over $P \in U$, where $U = S-D$. On $r^{-1}(U)$, the bottom left map $\mc{M} \to f_{1}^{*} \mc{O}_{X_1}(1)$ of \eqref{eq:commut-const} restricts to an isomorphism. In fact $f_1: r^{-1}(U) \to \pi_1^{-1}(U)$ is an isomorphism, so if $P \in U$, then $\mc{M}_{P}$ is isomorphic to $\mc{O}(1)$ on $Y_P \cong (X_1)_P \cong \PP^1_P$. Therefore in this case $H^1(Y_P, \mc{M}_P) = H^1(\PP^1_P, \mc{O}(1)) = 0$.

    Now suppose $P \in D$. Let $F:= (X_1)_{P} \cong \PP^1_{P}$ be the fiber of $P$ on $X_1$. Let $\mc{P} \in F$ be the point in $\mc{D}$ which is a relative degree $1$ closed point above $P$.

    Let $F^{st} \subseteq Y_P$ be the strict transform of the fiber. Note that $F^{st}$ can also be defined as the blowup of $F$ at $\mc{P}$, which is isomorphic to $F$. So $f_1: Y \to X_1$ restricts to an isomorphism $F^{st} \xrightarrow{\sim} F$.

    Let $E := f_1^{-1}(\mc{P})$. This is the component of the exceptional divisor $E_1$ above $P$. In particular, on an affine open $W\subseteq S$ containing $P$ but no other closed points of $D$, the map $Y_W\to (X_1)_W$ is simply the blowup of $(X_1)_{W}$ at $\mc{P}$. 

    Note that $X_1$ is a regular dimension $2$ scheme, so the exceptional divisor $E=f_1^{-1}(\mc{P})$ is the projectivized normal bundle at $\mc{P}$, which must be isomorphic then to $\PP^1_{\mc{P}}\cong \PP^1_{P}$.

    Now as divisors, we have that $Y_P$ is the pullback $f^{*} F = F^{st} + E$. This means $Y_P$ is the scheme theoretic intersection of $F^{st}$ and $E$. The scheme theoretic intersection of $F^{st}$ and $E$ is the preimage of $\mc{P}$ under the isomorphism $F^{st} \to F$. Call this preimage $\mc{P}'$, and note that it is a closed point equipped with an isomorphism $\mc{P}' \xrightarrow{\sim} \mc{P}$. So we can write $Y_P$ as the pushout
    \[
    Y_P \cong F^{st} \cup_{\mc{P}'} E.
    \]
    Let's understand the restriction of the line bundle $\mc{M}$ to $Y_P$. By \Cref{lem:p1-pushout}, the restriction $\mc{M}|_{Y_P}$ is determined by the restrictions $\mc{M}|_{F^{st}}$ and $\mc{M}|_{E}$.

    Recall the definition 
    \[
    \mc{M} = f^{*} \mc{O}_{X_1}(1) \otimes \mc{O}_Y(-E_1).
    \]
    By restricting to an open around $P$, we might as well assume that $E_1 = E$ is the blowup at $\mc{P}$ only.

    Now recall that on $E \cong \PP^1_{P}$ we have $\mc{O}(E)|_E \cong \mc{O}(-1)$ and $f^{*} \mc{L}|_{E} \cong \mc{O}_E$ for any line bundle $\mc{L}$ on $X_1$. So $\mc{M}|_E \cong \mc{O}(1)$.

    Next, note that the restriction of $f_1^{*} \mc{O}_{X_1}(1)$ to $F^{st}$ is the same as the pullback of $\mc{O}_{X_1}(1)|_{F}$ from $F$ to $Y_P$. Now $\mc{O}_{X_1}(1)|_{F}$ is isomorphic to $\mc{O}(1)$ on $F$, so it pulls back to $\mc{O}(1)$ on $F^{st}$.

    Finally, the restriction of $\mc{O}(E)$ to $F^{st}$ is the line bundle $\mc{O}(\mc{P}')\cong \mc{O}(1)$. So 
    \[
    \mc{M}|_{F^{st}} \cong (f^{*} \mc{O}_{X_1}(1)|_{F^{st}}) \otimes (\mc{O}(-E)|_{F^{st}})
    \]
    \[
    \cong \mc{O}_{F^{st}}(1) \otimes \mc{O}_{F^{st}}(-1) \cong \mc{O}_{F^{st}}.
    \]

    Therefore $\mc{M}_{P} = \mc{M}|_{Y_P}$ is exactly the type of line bundle in the second half of \Cref{lem:p1-pushout}, so we conclude $H^1(Y_P, \mc{M}_P) = 0$.

    So by Cohomology and Base Change, we know that $(r_{*} \mc{M})_{P} \cong H^0(Y_P, \mc{M}_P)$ for all closed points $P \in S$.

    Therefore, for $\ell: Y_P \to Y$, we can identify the map $\ell^{*}r^{*}r_{*} \mc{M} \to \ell^{*}\mc{M}$ with 
    \[
    r^{*} H^0(Y_P, \mc{M}_{P}) \to \mc{M}_P
    \]
    for each fiber $Y_P$. But we know that this is surjective for each fiber $Y_P$ with $P \in D$ by \Cref{lem:p1-pushout}, and for $P \not \in D$ it's clear that $\mc{O}(1)$ on $\PP^1_{P}$ is globally generated. So
    \[
    \ell^{*} r^{*} r_{*} \mc{M} \to \ell^{*} \mc{M}
    \]
    is surjective for all fibers over closed points $\ell: Y_P \to Y$. So the cokernel of $r^{*}r_{*} \mc{M} \to \mc{M}$ is a coherent sheaf which vanishes on the residue field of every closed point in every closed fiber on $Y$, and therefore on each stalk by Nakayama, so $r^{*}r_{*} \mc{M} \to \mc{M}$ is surjective on $Y$.

    In other words, we have successfully constructed a surjective morphism $r^{*} \mc{E}_2 \to \mc{M} = f_1^{*}\mc{O}_{X_1} \otimes \mc{O}_{Y}(-E_1)$ inducing a map $f_2: Y\to \PP(\mc{E}_2)$ of $S$-schemes. By definition of the universal property, we have guaranteed that $f_1^{*} \mc{O}_{X_1} \otimes \mc{O}_{Y}(-E_1) \cong f_2^{*} \mc{O}_{X_2}(1)$.

    Furthermore, our map restricts to an isomorphism over $U = S-D$, and on each fiber over $P \in D$, our map is isomorphic to the contraction morphism 
    \[
    Y_P = F^{st} \cup_{\mc{P}'} E \to E \xrightarrow{\sim}\PP(\mc{E}_2)_{P}.
    \]
    Therefore each strict transform $F^{st}$ contracts to a relative degree $1$ point $\mc{Q} \in \PP(\mc{E}_2)_P$. We would like to determine that the image of the contractions of all of the strict transforms $F^{st}$ is the closed subscheme $\mc{D}_2 \subseteq \PP(\mc{E}_2)$. We argue as follows:

    Recall the commutative diagram
    \begin{equation}\label{eq:X-commut}
    \begin{tikzcd}0 \ar[r] & \pi_1^{*} \E_2 \ar[r]\ar[d] & \pi_1^{*} \E_1 \ar[r]\ar[d] &\pi_1^{*} i_{*} \mc{O}_D \ar[r]\ar[d] & 0 \\
    0 \ar[r] & \mc{K} \ar[r]& \mc{O}_{X_1}(1) \ar[r]& j_{*} \mc{O}_D\ar[r]& 0.
    \end{tikzcd}.
    \end{equation}
    Let $V \subseteq X_1$ be the open subset $V:= X_1 - \mc{D}_1$. Note that the blowup $f_1: Y \to X_1$ induces an isomorphism $f_1^{-1}(V) \to V$. So applying the pushforward $(f_{1})_{*}$ to the surjection $r^{*} \mc{E}_2 \to \mc{M}$ on $f_1^{-1}(V)$ yields the isomorphic map $\pi_1^{*}{\mc{E}_2}|_{V} \to \mc{K}|_{V}$, which must therefore be a surjection.

    This gives a morphism $V \to \PP(\mc{E}_2)$, which can be identified with the restriction of $f_2: Y\to \PP(\mc{E}_2)$ to $f_1^{-1}(V) \cong V$. Since $f_2:Y \to \PP(\mc{E}_2)$ contracts each $F^{st}$ above $P \in D$ to a point $\mc{Q} \in \PP(\mc{E}_2)|_{P}$, we find that the same must be true on $V$. That is, when the map $V \to \PP(\mc{E}_2)$ is restricted to the fiber over a point $P \in D$, it sends the affine line $(V)|_{P}\cong (X_1)_P - \mc{P}$ to the same point $\mc{Q} \in \PP(\mc{E}_2)|_{P}$.

    To proceed, let's consider the commutative diagram \eqref{eq:X-commut} restricted to the open $V\subseteq X_1$. In this case the bottom left map $\mc{K}|_{V} \to \mc{O}_{X_1}(1)$ becomes an isomorphism.
    \begin{equation}\label{eq:V-commut}
    \begin{tikzcd}0 \ar[r] & (\pi_1^{*} \E_2)|_{V} \ar[r]\ar[d] & (\pi_1^{*} \E_1)|_{V} \ar[r]\ar[d] &\pi_1^{*} i_{*} \mc{O}_D \ar[r]\ar[d]|_{V} & 0 \\
    0 \ar[r] & \mc{K}|_{V} \ar{r}{\sim}& \mc{O}_{X_1}(1)|_{V} \ar[r]& 0\ar[r]& 0.
    \end{tikzcd}.
    \end{equation}
    Now consider the fiber $V_D$ of $V$ at $D$. The image of $\mc{E}_2|_{D} \to \mc{E}_1|_D$ has rank $1$, so there exists a line bundle $\mc{L}_2$ on $D$ which is the image of $\mc{E}_2|_{D} \to \mc{E}_1|_{D}$, exhibiting $\mc{E}_2|_D \twoheadrightarrow \mc{L}_2 \hookrightarrow \mc{E}_1|_{D}.$ The map $\mc{E}|_{D}\twoheadrightarrow \mc{L}_2$ exactly corresponds to the $D$-marking $(\mc{E}_2, \DD_2):= \ElmD(\mc{E}_1, \DD_1)$. Furthermore, restricting \eqref{eq:V-commut} to $V_D$ gives a commutative diagram

    \[
    \begin{tikzcd} & (\pi_1^{*} \E_2)|_{V_D} \ar[r]\ar[d] &(\pi_1^{*} \mc{L}_2)|_{V_D} \ar[r] & (\pi_1^{*} \E_1)|_{V_D} \ar[d]\\
     & \mc{K}|_{V_D} \ar{rr}{\sim}& &\mc{O}_{X_1}(1)|_{V_D} 
    \end{tikzcd}.
    \]
    which means that the first vertical map factors through the top left horizontal map. In other words, the induced map $V_D \to \PP(\mc{E}_2)_{D}$ factors through the subscheme determined by ${\mc{E}_2}|_{D} \twoheadrightarrow \mc{L}_2$, which is exactly $\DD_2 \subseteq \PP(\mc{E}_2)$. Therefore each $F^{st}$ must contract to a point in $\DD_2$.

    So we have a morphism of $S$-schemes $f_2: Y \to \PP(\mc{E}_2)$ which is an isomorphism above $U \subseteq S$. On the fibers above $P \in D$, the morphism is the contraction map on $Y_P \cong F^{st} \cup_{\mc{P}'} E$. Here the contraction is onto $E \cong \PP(\mc{E}_2)_{P}$ and each fiber $F^{st}$ contracts to a point $\mc{Q} \in \DD_2$. By universal property of blowups, $f_2$ factors through a map $Y\to Y':= \Bl_{\DD_2}\PP(\mc{E}_2)$, as the preimage $f_2^{-1}(\DD_2) = \cup_P F^{st}$ is an effective Cartier divisor.

    We claim $Y\to Y'$ is an isomorphism.

    Consider a fiber $Y_P \to Y'_P$. Write $Y_P \cong F^{st} \cup_{\mc{P} '} E$ and $Y'_P\cong G^{st} \cup_{\mc{Q}'} E'$, where $G$ is the fiber $\PP(\mc{E}_2)_{P}$ and $E'$ is the exceptional divisor above some point $\mc{Q} \in \PP(\mc{E}_2)_{P}$.

    Now we know that our map must induce an isomorphism $E \to G^{st}$. Since $Y\to Y'$ is a proper, dominant morphism onto an integral scheme, it must be a surjection onto its scheme-theoretic image $Y'$. In particular, the map $F^{st} \to E'$ must be surjective. Furthermore, under this map the preimage of $\mc{Q}' \in E' \cong \PP^1_P$ is $\mc{P}' \in F^{st} \cong \PP^1_{P}$, so the induced map of projective lines must have degree $1$. Therefore it is an isomorphism, and $Y_P\to Y'_P$ is an isomorphism.

    Finally, we have constructed a morphism $Y \to Y'$ which is an isomorphism above $U$ and an isomorphism above $D$. Furthermore this map is proper, birational, bijective, $Y$ and $Y'$ are integral, and $Y'$ is normal. So $Y\to Y'$ is an isomorphism by Zariski's Main Theorem.

\end{proof}
In the proof, we used a few lemmas.

\begin{lemma}\label{lem:p1-pushout}
    Let $P$ be a closed point. Let $X_1, X_2$ be schemes over $P$ such that $X_1 \cong \PP^1_{P}$ and $X_2 \cong \PP^1_{P}$. Let $Z$ be a scheme with a given isomorphism $Z\xrightarrow{\sim} P$. Let $a_1: Z \to X_1$ and $a_2: Z \to X_2$ be closed immersions of schemes over $P$.

    Then there is a unique pushout scheme $X = X_1\cup_{Z} X_2$. The restriction of line bundles to the closed subschemes $X_1\subseteq X$ and $X_2 \subseteq X$ induces an isomorphism
    \[
    \Pic(X) \xrightarrow{\sim} \Pic(X_1) \times \Pic(X_2).
    \]
    Let $\mc{L}\in \Pic(X)$ denote the unique line bundle which maps to  $\left( \mc{O}_{X_1}, \mc{O}_{X_2}(1)\right)$ under the previous isomorphism. Let $\pi: X\to P$ be the structure map. Then the map
    \[
    \pi^{*}\pi_{*} \mc{L} \to \mc{L}
    \]
    is a surjection and corresponds to the canonical contraction morphism $X\to X_2$, which contracts $X_1$ to $Z \in X_2$. Also, we have $H^1(X, \mc{L})=0$.

\end{lemma}
\begin{proof}
    If $a_1:Z \to X_1$ and $a_2: Z \to X_2$ are closed immersions, then the pushout exists in the category of schemes (see \cite[\href{https://stacks.math.columbia.edu/tag/0B7M}{Lemma 0B7M}]{stacks-project}). Furthermore we have a short exact sequence of sheaves on $X$
    \begin{equation*}
    0 \to \mc{O}_X \to \mc{O}_{X_1} \oplus \mc{O}_{X_2} \to \mc{O}_{Z} \to 0
    \end{equation*}
    where the first map sends $s$ to $(s|_{X_1}, s|_{X_2})$ and the second map sends $(s_1, s_2)$ to $s_1|_{Z} - s_2|_{Z}$.

    We claim there is also a short exact sequence
    \begin{equation}\label{eq:full-ses}
    1 \to \mc{O}_X^{*} \to \mc{O}_{X_1}^{*} \oplus \mc{O}_{X_2}^{*}\to \mc{O}_Z^{*}\to 1.
    \end{equation}
    where the first map sends $s$ to $(s|_{X_1}, s|_{X_2})$ and the second map sends $(s_1, s_2)$ to $s_1|_{Z} (s_2|_{Z})^{-1}$. We can check exactness on stalks. Away from $Z$, it's trivially exact. Now consider the stalks at the closed point $Z$.
    \begin{equation}\label{eq:mult-stalk}
    1 \to \mc{O}_{X,Z}^{*} \to \mc{O}_{X_1, Z}^{*} \oplus \mc{O}_{X_2, Z}^{*} \to \mc{O}_{Z}^{*} \to 1.
    \end{equation}
    These are simply the groups of units in the local rings. Note that an element of a local ring is a unit if and only if its reduction to $Z$ is nonzero.

    Now we prove exactness by comparison with the additive exact sequence at the stalk:
    \begin{equation}\label{eq:add-stalk}
    0 \to \mc{O}_{X,Z} \to \mc{O}_{X_1, Z} \oplus \mc{O}_{X_2, Z} \to \mc{O}_{Z} \to 0.
    \end{equation}

    Consider the first map in \eqref{eq:mult-stalk}, which sends an invertible element $s \in \mc{O}_{X,Z}^{*}$ to $(s|_{X_1}, s|_{X_2})$. So if $s$ maps to $(1,1)$ in \eqref{eq:mult-stalk}, then $s-1$ maps to $(0,0)$ in \eqref{eq:add-stalk}, which means $s=1$ because \eqref{eq:add-stalk} is left exact.

    Now let's consider exactness at the middle. Suppose that $(s_1, s_2)$ maps to $1$, meaning $(s_1|_{Z}) (s_2|_{Z})^{-1} =1$. This means $s_1|_{Z} = s_2|_{Z}$, so there exists an element $s \in \mc{O}_X$ mapping to $(s_1, s_2)$. Furthermore this element satisfies $s|_{Z} = s_1|_{Z} = s_2|_{Z} \neq 0$, so $s$ is a unit in $\mc{O}_{X,Z}$. Thus \eqref{eq:mult-stalk} is exact at the middle.

    Finally, let's consider right exactness. This is clear, as every element of $\mc{O}_Z$ comes from an element of (say) $\mc{O}_{X_1, Z}$, and this element is invertible if and only if the target element in $\mc{O}_Z$ is invertible.

    So the sequence \eqref{eq:full-ses} is exact. Taking cohomology, we get
    \[
\begin{tikzcd}[arrows=to]
0 \rar & H^0(X, \mc{O}_X^{*}) \rar & H^0(X_1, \mc{O}_{X_1}^{*}) \oplus  H^0(X_2, \mc{O}_{X_2}^{*})\arrow[d, phantom, ""{coordinate, name=Z}] \rar & H^0(Z, \mc{O}_Z^{*})
\arrow[dll,rounded corners,to path={ -- ([xshift=2ex]\tikztostart.east)
|- (Z) [near end]\tikztonodes
-| ([xshift=-2ex]\tikztotarget.west)
-- (\tikztotarget)}] \\
 & \Pic(X) \rar & \Pic(X_1) \oplus \Pic(X_2)\rar& \Pic(Z).
\end{tikzcd}
\]
But $\Pic(Z) = 0$ and $H^0(X_1, \mc{O}_{X_1}^{*})\to H^0(Z, \mc{O}_Z^{*})$ is surjective (both are the group of units in $Z\cong P$). So we get the isomorphism $\Pic(X) \xrightarrow{\sim} \Pic(X_1) \oplus \Pic(X_2)$.

Because of this, we can prove the second part of the lemma ``backwards.'' In other words, if we define the canonical contraction morphism $c: X \to X_2$ over $P$, and we let $\mc{L}':= c^{*} \mc{O}_{X_2}(1)$, then $\mc{L}'$ must be globally generated and defines the map $c: X \to X_2$ by the surjection $\pi^{*} \pi_{*} \mc{L}' \to \mc{L}'$. So it suffices to show that $\mc{L}' = c^{*} \mc{O}_{X_2}(1)$ is isomorphic to $\mc{L} = \Phi^{-1}(\mc{O}_{X_1}, \mc{O}_{X_2}(1))$. In other words, it suffices to show that the restriction of $\mc{L}'$ to $X_1$ is $\mc{O}_{X_1}$ and the restriction of $\mc{L}$ to $X_2$ is $\mc{O}_{X_2}(1)$.

But these are both clear. For the first part, note that notice that the composition $X_2 \to X \to X_2$ is the identity, so $\mc{O}_{X_2}(1)$ pulls back to $\mc{O}_{X_2}(1)$.

For the first part, note that the restriction of $c$ to $X_1$ factors through $f: X_1 \to Z$. So the pullback of $\mc{O}_{X_2}(1)$ factors through $\mc{O}_{X_2}(1)|_{Z} \cong \mc{O}_{Z}$, implying that the restriction of $\mc{L}'$ to $X_1$ is $f^{*} \mc{O}_Z = \mc{O}_{X_1}$.

So the second part of the lemma is proved.

Finally, to compute $H^1(X, \mc{L})$ consider the long exact sequence
    \[
\begin{tikzcd}[arrows=to]
0 \rar & H^0(X, \mc{L}) \rar & H^0(X_1, \mc{L}) \oplus  H^0(X_2, \mc{L})\arrow[d, phantom, ""{coordinate, name=Z}] \rar & H^0(Z, \mc{L})
\arrow[dll,rounded corners,to path={ -- ([xshift=2ex]\tikztostart.east)
|- (Z) [near end]\tikztonodes
-| ([xshift=-2ex]\tikztotarget.west)
-- (\tikztotarget)}] \\
 & H^1(X, \mc{L}) \rar & H^1(X_1, \mc{L})\oplus H^1(X_{2}, \mc{L})\rar& H^1(Z, \mc{L}).
\end{tikzcd}
\]

Note that $H^1(X_1, \mc{L}) = H^1(X_1, \mc{O}_{X_1}) = 0$ and $H^1(X_2, \mc{L}) = H^1(X_2, \mc{O}(1)) = 0$.

Let $k$ be the field with $Z = \Spec k$. Then $H^0(X_1, \mc{L}) = k$ and restricts to $H^0(Z, \mc{L}) = k$ surjectively. So $H^1(X, \mc{L})=0$ as desired.

\end{proof}

In this lemma, we prove that two ``canonical'' maps are the same.

\begin{lemma}\label{lem:taut-push}
    Let $S$ be a Dedekind scheme. Let $i: D \to S$ be a reduced effective divisor on $S$. Let $\E$ be a rank $2$ vector bundle on $S$, let $\LL$ be a line bundle on $D$, and let $\phi: \E \twoheadrightarrow i_{*} \LL$ be the data associated to a $D$-marking $\DD\subseteq \PP(\E)$.

    Let $j:D \to \PP(\E)$ be the section associated to the $D$-marking. There is an associated restriction morphism $\mc{O}_{\PP(\E)} \to j_{*} \mc{O}_{D}$.

    Let $\mc{M} = \mc{O}_{\PP(\E)}(1)$, and let $\mc{M} \to j_{*} j^{*} \mc{M}$ be the restriction morphism. Push forward this morphism by $\pi$ to get a map $\phi': \E \to i_{*} \LL'$ for a line bundle $\LL' = j^{*} \mc{M}$.

    Then $\mc{L} \cong \mc{L}'$ and $\phi= \phi'$.
\end{lemma}
\begin{proof}
    Recall $\pi \circ j = i$. The map $\psi_1: i^{*} \E \to \LL$ associated to the $D$-marking is induced by pulling back $\pi^{*} \E \to \mc{O}_{\PP(\E)}$ by $j$. Therefore $\LL = j^{*} \mc{M} = \LL'$. Furthermore we have
    \[
    \phi: \E \to i_{*} i^{*} \E \to i_{*} \LL
    \]
    is the composite map
    \[
    \pi_{*} \mc{M} \to i_{*} i^{*} \pi_{*} \mc{M} = i_{*} j^{*} \pi^{*} \pi_{*} \mc{M} \to i_{*} j^{*} \mc{M}.
    \]
    On the other hand, our map $\phi'$ induced by the pushforward of $\pi$ is 
    \[
    \phi': \pi_{*} \mc{M} \to \pi_{*} j_{*} j^{*}\mc{M} = i_{*} j^{*} \mc{M}.
    \]

    The equality then follows by the triangle identity and the commutative diagram
    \[
    \begin{tikzcd}
        \pi_{*} \mc{M} \ar[r] \ar[rd, "\Id"] &\pi_{*} \pi^{*} \pi_{*} \mc{M} \ar[d]\ar[r]& \pi_{*} j_{*} j^{*} \pi^{*} \pi_{*} \ar[d]\\
        &\pi_{*} \mc{M} \ar[r] & \pi_{*}j_{*} j^{*} \mc{M}.
    \end{tikzcd}
    \]
\end{proof}
\singlespacing

\printbibliography
\addcontentsline{toc}{chapter}{Bibliography}
\end{document}